\documentclass[10pt]{amsart}
\usepackage{amssymb}
\usepackage{amsmath}
\usepackage{mathrsfs}
\usepackage{verbatim}
\usepackage{amsthm}
\usepackage{wasysym}
\usepackage{upgreek}
\usepackage{color}
\usepackage{accents}
\usepackage{fancyhdr}
\usepackage{hyperref}

\numberwithin{equation}{subsection}

\newtheorem{proposition}{Proposition}[subsection]
\newtheorem{lemma}[proposition]{Lemma}
\newtheorem{corollary}[proposition]{Corollary}
\newtheorem{theorem}{Theorem}[section]

\theoremstyle{definition}
\newtheorem{definition}{Definition}[subsection]
\newtheorem{remark}{Remark}[subsection]

\newcommand{\eqdef}{\overset{\mbox{\tiny{def}}}{=}}

\newcommand{\gzerozeronorm}[1]{S_{g_{00}+1;#1}}
\newcommand{\gzerostarnorm}[1]{S_{g_{0*};#1}}
\newcommand{\hstarstarnorm}[1]{S_{h_{**};#1}}
\newcommand{\gnorm}[1]{S_{g;#1}}
\newcommand{\fluidnorm}[1]{S_{P - \bar{p},u^*;#1}}
\newcommand{\totalnorm}[1]{\mathbf{S}_{#1}}

\newcommand{\newufluidenergy}[1]{\mathscr{E}_{u^*;#1}}

\newcommand{\gzerozeroenergy}[1]{E_{g_{00}+1;#1}}
\newcommand{\gzerostarenergy}[1]{E_{g_{0*};#1}}
\newcommand{\hstarstarenergy}[1]{E_{h_{**};#1}}
\newcommand{\genergy}[1]{E_{g;#1}}
\newcommand{\fluidenergy}[1]{E_{P - \bar{p},u^*;#1}}
\newcommand{\totalenergy}[1]{\mathbf{E}_{#1}}

\newcommand{\Ent}{\eta}
\newcommand{\speed}{c_s}
\newcommand{\underpartial}{\underline{\partial}}

\newcommand{\decayparameter}{\varkappa}

\begin{document}

\title{The Nonlinear Future-Stability of the FLRW Family of Solutions to the Euler-Einstein System with a Positive Cosmological Constant}

\email{jspeck@math.princeton.edu}
\author{Jared Speck$^*$}

\begin{abstract}
	In this article, we study small perturbations of the family of Friedmann-Lema\^{\i}tre-Robertson-Walker 
	cosmological background solutions to the $1 + 3$ dimensional Euler-Einstein system with a positive cosmological constant. 
	These background solutions describe an initially uniform quiet fluid of positive energy density evolving in a spacetime 
	undergoing accelerated expansion. Our nonlinear analysis shows that under the equation of state 
	$p = \speed^2 \rho,$ $0 < \speed^2 < 1/3,$ the background solutions are globally future-stable. In particular, we prove that 
	the perturbed spacetime solutions, which have the topological structure $[0,\infty) \times \mathbb{T}^3,$ are future 
	causally geodesically complete. These results are extensions of previous results derived by the author in a collaboration 
	with I. Rodnianski, in which the fluid was assumed to be irrotational. Our novel analysis of a fluid with non-zero vorticity
	is based on the use of suitably-defined energy currents.
\end{abstract}

\thanks{$^*$Princeton University, Department of Mathematics, Fine Hall, Washington Road, Princeton, NJ 08544-1000, USA. This article was initiated while the author was a postdoctoral researcher at the University of Cambridge.}

\thanks{The author was supported in part by the Commission of the European Communities, ERC Grant Agreement No 208007. He was also supported in part by an NSF All-Institutes Postdoctoral Fellowship administered by the Mathematical Sciences Research Institute through its core grant DMS-0441170.}



\date{Version of \today}
\maketitle
\setcounter{tocdepth}{2}

\pagenumbering{roman} \tableofcontents \newpage \pagenumbering{arabic}


\section{Introduction}

The Euler-Einstein system models the evolution of a dynamic spacetime\footnote{A spacetime is defined to be a $4-$dimensional time-orientable Lorentzian manifold $\mathcal{M}$ together with a metric $g_{\mu \nu}$ on $\mathcal{M}$ of signature $(-,+,+,+).$} $(\mathcal{M},g_{\mu \nu})$ containing a perfect fluid. As is discussed below, in applications to cosmology, an ``additional term'' of the form $\Lambda g_{\mu \nu}$ is often added to the equations in order to alter the nature of solutions. The constant $\Lambda,$ which was first contemplated by Einstein \cite{aE1917}, is known as the \emph{cosmological constant}. For physical reasons to be discussed below, we assume throughout the article that $\Lambda > 0$ is a fixed constant, and that the fluid equation of state is $p=\speed^2 \rho,$ where $p$ is the fluid \emph{pressure}, $\rho$ is the \emph{proper energy density}, and the non-negative constant $\speed$ is the \emph{speed of sound}. As is fully discussed in Section \ref{S:EE}, under these assumptions, the Euler-Einstein system comprises the equations\footnote{Throughout the article, we work in units with $8 \pi G = c = 1,$ where $c$ is the speed of light propagation in Maxwell's theory of electromagnetism, and $G$ is Newton's universal gravitational constant.}

\begin{subequations}
\begin{align} 
	\mbox{Ric}_{\mu \nu} - \frac{1}{2}R g_{\mu \nu} + \Lambda g_{\mu \nu} & = T_{\mu \nu}, && (\mu, \nu = 0,1,2,3),  
		\label{E:EinsteinIntro} \\
	D_{\alpha} T^{\alpha \mu} & = 0,&& (\mu = 0,1,2,3),  \label{E:Fluidenergymomentumintro} 
\end{align}
\end{subequations}
where $\mbox{Ric}_{\mu \nu}$ is the \emph{Ricci curvature tensor}, $R = g^{\alpha \beta} \mbox{Ric}_{\alpha \beta}$ is the \emph{scalar curvature}, $T_{\mu \nu} = (\rho + p)u_{\mu} u_{\nu} + p g_{\mu \nu}$ is the energy-momentum tensor of a perfect fluid, and 
$u^{\mu}$ is the \emph{four-velocity}, a future-directed, timelike vectorfield that is subject to the unit normalization constraint $g_{\alpha \beta}u^{\alpha} u^{\beta} = -1.$ The fundamental unknowns in the above system may be considered to be $(\mathcal{M},g_{\mu \nu},p,u^{\mu}).$ Although we limit our discussion to the physically relevant case of $1 + 3$ dimensions, we expect that our work can be easily generalized to apply to the case of $1 + n$ dimensions, $n \geq 3.$

Our choice of the equation of state $p=\speed^2 \rho,$ which is often made in the cosmology literature, is sufficient to close the above system of equations. We remark that without specifying an equation of state, there would be too many fluid variables, and not enough fluid equations. There are far too many references on this subject to provide a comprehensive list, but for an overview of the field, readers may consult e.g. the mathematical references \cite{sC2001}, \cite{aR2005b}, \cite{rW1984}, and the cosmological references \cite{jPbR2003}, \cite{vS2004}. As is explained in Section \ref{S:backgroundsolution}, under the equation of state $p=\speed^2 \rho,$ there exists a well-known family of Friedmann-Lema\^{\i}tre-Robertson-Walker (FLRW) solutions to \eqref{E:EinsteinIntro} - \eqref{E:Fluidenergymomentumintro} that are frequently used to model a fluid-filled universe undergoing accelerated expansion; these are the solutions that we investigate in detail in this article. We remark that \textbf{the accelerated nature of the expansion in the FLRW family is generated by the positivity of $\Lambda$}; when $\Lambda = 0,$ the FLRW expansion is no longer accelerated. 

The cases $p=0$ and $p=(1/3) \rho,$ which are known as the ``pressureless dust'' and ``radiative'' equations of state, are of special significance in the cosmological literature. The latter is often used as a simple model for a ``radiation-dominated'' universe, while the former for a ``matter-dominated'' universe. Unfortunately, as we will see, these two equations of state lie just outside of the scope of our main results, which are restricted to the parameter range $0 < \speed < \sqrt{1/3}.$ We expect that an analogous stability result can be proved in the case $\speed = 0,$ i.e., that the failure of our proof in the case of the pressureless dust is only due to the fact that our methods degenerate: in the case $\speed = 0,$ it can be checked that the energy-density loses one degree of differentiability, which requires a re-working of all of the estimates. We will treat this case in detail in an upcoming article. In contrast, \emph{there may be instability in the cases} $\speed \geq \sqrt{1/3}.$ This is an interesting possibility worthy of further investigation. We state our main results roughly here; they are are stated more precisely as theorems in Sections \ref{S:GlobalExistence} and \ref{S:Asymptotics}.

\begin{changemargin}{.25in}{.25in} 
\textbf{Main Results.} \
If $0 < \speed < \sqrt{1/3},$ then the FLRW background solution 
$([0,\infty) \times \mathbb{T}^3, \widetilde{g}_{\mu \nu}, \widetilde{p},
\widetilde{u}^{\mu}),$ $(\mu, \nu = 0,1,2,3),$ to \eqref{E:EinsteinIntro} - \eqref{E:Fluidenergymomentumintro}, which describes an initially uniform quiet fluid of constant \emph{positive} pressure $\bar{p},$ is globally future-stable under small perturbations. In particular, small perturbations of the initial data corresponding to the background solution have maximal globally hyperbolic developments that are future\footnote{Throughout this article, $\partial_t$ is future-directed.} causally geodesically complete. Here, $\widetilde{g} = -dt^2 + e^{2 \Omega(t)}\sum_{i=1}^3 (dx^i)^2,$ $\widetilde{p} = e^{-(3 + \decayparameter) \Omega}\bar{p},$ and $\widetilde{u}^{\mu} = (1,0,0,0),$ where $\decayparameter  = 3 \speed^2$ and $\Omega(t) \sim (\sqrt{\Lambda/3})t$ is defined in \eqref{E:BigOmega}. Furthermore, in the wave coordinate system introduced in Section \ref{SS:harmoniccoodinates}, the components $g_{\mu \nu}$ of the perturbed metric, its inverse $g^{\mu \nu},$ $p,$ $u^{\mu},$ and various coordinate derivatives of these quantities converge as $t \rightarrow \infty.$
\end{changemargin}

\begin{remark}
	Note that our results do not address the question of whether or not the perturbed solution is converging back to 
	exactly the FLRW solution.
\end{remark}

\begin{remark}
	In this article, we restrict our attention to spacetimes that have spatial slices that are diffeomorphic to $\mathbb{T}^3$ 
	(i.e., $[-\pi,\pi]^3$ with the ends identified). However, we expect that the local patching arguments developed
	in \cite{hR2008} could be used to prove future-stability results for a larger class of background
	solutions featuring alternative compact spatial topologies.
\end{remark}

The \textbf{Main Results} above are an extension of recent work by Rodnianski and the author \cite{iRjS2009}, which provides a proof of analogous results in the case that the fluid is \emph{irrotational}. In a future article, we plan to further extend these results by analyzing the behavior of the fluid when the spacetime expansion is sub-exponential. In the case of an irrotational fluid, the fluid equations (that is, the relativistic Euler equations \eqref{E:Fluidenergymomentumintro}, which are more fully discussed in Section \ref{S:EE}) are equivalent to the following scalar equation for the derivatives $\partial \Phi$ of the \emph{fluid potential} $\Phi:$

\begin{align} \label{E:IrrotationalFluid}
	D_{\alpha} (\sigma^{s} D^{\alpha} \Phi) & = 0,
\end{align}
where $\sigma = - g^{\alpha \beta}(\partial_{\alpha} \Phi)(\partial_{\beta} \Phi)$ is the \emph{enthalpy per particle},
and $s = (1 - \speed^2)/(2\speed^2).$ We remark that in the irrotational case, the pressure, proper energy density,
and four-velocity can be expressed in terms of $\partial \Phi$ as $p = \frac{1}{s+1} \sigma^{s+1},$ $\rho = \frac{2s+1}{s+1} \sigma^{s+1},$ and $u_{\mu} = - \sigma^{-1/2} \partial_{\mu} \Phi.$ Before discussing the differences between our analysis of the full Euler system \eqref{E:Fluidenergymomentumintro} and the scalar equation \eqref{E:IrrotationalFluid}, we provide some additional physical context for the problem at hand.

Our principal motivation for inserting a positive cosmological constant into the Einstein equations is that there is experimental evidence for accelerated expansion. The first such evidence came from experiments carried out in the 1990's, which involved measurements of the Cosmic Microwave Background and the \emph{red shift} of type IA supernovae. The notion of accelerated expansion is perhaps best illustrated through a specific example, such as the following solution to the Einstein-vacuum equations with a positive cosmological constant (i.e., \eqref{E:EinsteinIntro} with $T_{\mu \nu} = 0$): the metric $g = - dt^2 + e^{2Ht} \sum_{a=1}^3 (dx^a)^2$ on the manifold $\mathcal{M} =(-\infty, \infty) \times \mathbb{R}^3,$ where $H = \sqrt{\Lambda/3}.$ In this spacetime, the curves $\gamma^{\mu}(\tau) = (\tau,b,0,0),$ where $b$ is a constant, are (non-affine parameterized) geodesics representing the trajectories of a family of physical observers. Let us denote such a curve (observer) by $\gamma_{(b)}.$ Here, $(x^0,x^1,x^2,x^3)$ is a standard coordinate system on $(-\infty, \infty) \times \mathbb{R}^3.$ At $\tau =0,$ the observers $\gamma_{(0)}$ and $\gamma_{(b)}$ are initially separated by a distance $b,$ while at later times $\tau,$ their distance is $b e^{H \tau}.$ It is roughly this kind of growth in distance at an increasing rate that leads to the term ``accelerated expansion.'' Note that the FLRW background metrics from the \textbf{Main Results} feature this kind of behavior. Also note that the aforementioned distance is measured in the Riemannian manifold $(\Sigma_{\tau}, \underline{g}_{(\tau)}),$ where $\Sigma_{\tau} \eqdef \lbrace x \in \mathcal{M} \ | \ x^0 = \tau \rbrace,$ and $\underline{g}_{(\tau)}$ is the Riemannian metric on $\Sigma_{\tau}$ induced by $g,$ i.e., the first fundamental form of $\Sigma_{\tau}.$ Consequently, it is not a notion inherent to spacetime itself, but instead depends on the particular foliation $\lbrace \Sigma_{\tau} \rbrace_{\tau \in (-\infty,\infty)}.$

The \textbf{Main Results} stated above have made implicit use of the following fundamental facts concerning the Einstein equations:

\begin{enumerate}
	\item The question of the local existence of solutions can be formulated as an \emph{initial value problem}, in which
		suitably specified initial data give rise to unique local solutions.
	\item Each initial data set satisfying the \emph{Einstein constraint equations} (which are discussed below) launches a unique 
		maximal solution known as the \emph{maximal globally hyperbolic development} of the data. We emphasize that we use the 
		term ``maximal'' to mean the largest possible globally hyperbolic\footnote{Recall that a spacetime is globally hyperbolic 
		if and only if it contains a Cauchy hypersurface (see Definition \ref{D:Cauchy}).} spacetime \emph{determined uniquely} by 
		the data.
\end{enumerate}
These two facts, which were established by Choquet-Bruhat \cite{CB1952} and Choquet-Bruhat/Geroch \cite{cBgR1969} respectively, are by now well-known. However, their validity is disguised by the diffeomorphism invariance of the equations. Roughly speaking, this means that each spacetime ``solution'' is represented by an infinite number of solutions to the equations \eqref{E:EinsteinIntro} - \eqref{E:Fluidenergymomentumintro}, all of which are related by changes of coordinates. A closely related difficulty is that there is no canonical coordinate system for analyzing the Einstein equations. In her proof of $(1),$ Choquet-Bruhat overcame these difficulties through the use of \emph{wave coordinates}\footnote{A wave coordinate system is also known as \emph{harmonic gauge} or \emph{de Donder gauge}.}, a special class of coordinate systems that date back at least to 1921, where they are featured in the work of de Donder \cite{tD1921}. More specifically, the classic wave coordinate condition is $\Gamma^{\mu} = 0,$ $(\mu=0,1,2,3),$ where the $\Gamma^{\mu} = g^{\alpha \beta} \Gamma_{\alpha \ \beta}^{\ \mu}$ (see \eqref{E:EMBIChristoffeldef}) are the contracted Christoffel symbols of the metric. Although it is not difficult (see Section \ref{SS:IDREDUCED}) to construct a spacetime coordinate system in which the $\Gamma^{\mu} = 0$ along the ``initial'' Cauchy hypersurface $\mathring{\Sigma}$ (upon which the initial data are specified), it is not immediately clear whether or not this condition can always be extended (for solutions to the Einstein equations) to include a spacetime neighborhood of $\mathring{\Sigma}.$ Choquet-Bruhat was the first to answer this question in the affirmative.

For the purposes of proving a local existence result, the key point is that in a wave coordinate system, the ``gravitational part'' of the Einstein equations (i.e. equations \eqref{E:EinsteinIntro}) is equivalent to a modified system comprising \emph{quasilinear wave equations}. Roughly speaking, the modified system is obtained by setting $\Gamma^{\mu} = 0$ for certain
terms in \eqref{E:EinsteinIntro}. When the modified equations are coupled to the first order fluid equations \eqref{E:Fluidenergymomentumintro} (which are well-known to be hyperbolic), the result is a hyperbolic system of mixed order. For such hyperbolic systems, local existence in a suitable Sobolev space follows from rather standard methods based on energy estimates; see e.g. \cite{lH1997}, \cite{jSmS1998}, \cite{cS2008}, \cite{mT1997III}, and the discussion in Section \ref{SS:ClassicalLocalExistence}. From this perspective, the main contribution of Choquet-Bruhat's work \cite{CB1952} is as follows: under suitable assumptions, the $\Gamma^{\mu}= 0$ condition is propagated by the flow of the modified coupled system;
see Section \ref{SS:PreservationofHarmonicGauge} for more details.

Before commenting further on the subtleties of the analysis, let us discuss some standard facts concerning the initial data. The initial data for the Euler-Einstein system consist of a $3-$dimensional Riemannian manifold $\mathring{\Sigma},$ and the following tensorfields on $\mathring{\Sigma}:$ a Riemannian metric $\mathring{\underline{g}}_{jk},$ a symmetric two-tensor $\mathring{\underline{K}}_{jk},$ a function $\mathring{p},$ and a vectorfield $\underline{\mathring{u}}^j,$ 
$(j,k = 1,2,3).$ A solution consists of a $4-$dimensional manifold $\mathcal{M},$ a Lorentzian metric $g_{\mu \nu},$ a function $p,$ a future-directed, unit-normalized vectorfield $u^{\mu},$ $(\mu, \nu = 0,1,2,3),$ on $\mathcal{M}$ satisfying \eqref{E:EinsteinIntro} - \eqref{E:Fluidenergymomentumintro}, and an embedding $\mathring{\Sigma} \hookrightarrow \mathcal{M}$ such that $\mathring{\underline{g}}_{jk}$ is the first fundamental form\footnote{Recall that $\mathring{\underline{g}}$ is defined at the point $x$ by $\mathring{\underline{g}}(X,Y) = g(X,Y)$ for all $X,Y \in T_x \mathring{\Sigma}.$} of $\mathring{\Sigma},$ $\mathring{\underline{K}}_{jk}$ is the second fundamental form\footnote{Recall that $\mathring{\underline{K}}$ is defined at the point x by $\mathring{\underline{K}}(X,Y) \eqdef g(D_{X} \hat{N},Y)$ for all $X,Y \in T_x \mathring{\Sigma},$ where $\hat{N}$ is the future-directed unit normal to $\mathring{\Sigma}$ at $x.$} of $\mathring{\Sigma},$ the restriction of $p$ to $\mathring{\Sigma}$ is $\mathring{p},$ and $(\underline{\mathring{u}}^1, \underline{\mathring{u}}^2, \underline{\mathring{u}}^3)$ is the $g-$orthogonal projection of $(u^0,u^1,u^2,u^3)$ onto $\mathring{\Sigma};$ see Section \ref{SS:IDENTIFICATIONS} for a summary of the conventions we use for identifying tensors inherent to $\mathring{\Sigma}$ with spacetime tensors. It is important to note that the initial value problem is overdetermined, and that the data are subject to the \emph{Gauss} and \emph{Codazzi} constraints:

\begin{subequations}
\begin{align}
	\mathring{\mathring{\underline{R}}} - \mathring{\underline{K}}_{ab} \mathring{\underline{K}}^{ab} + (\mathring{\underline{g}}^{ab} 
		\mathring{\underline{K}}_{ab})^2 & = 2 T_{00}|_{\mathring{\Sigma}}, \label{E:Gaussintro} \\
	\mathring{\underline{D}}^a \mathring{\underline{K}}_{aj} - \mathring{\underline{g}}^{ab}  \mathring{\underline{D}}_j \mathring{\underline{K}}_{ab}  & = 
		T_{0j}|_{\mathring{\Sigma}}, \label{E:Codazziintro}
\end{align}
\end{subequations}
where $\mathring{\underline{R}}$ is the scalar curvature of $\mathring{\underline{g}},$ $\mathring{\underline{D}}$ is the Levi Civita connection corresponding to $\mathring{\underline{g}}.$ In the above formulas, indices are lowered and raised with
$\mathring{\underline{g}}$ and $\mathring{\underline{g}}^{-1},$ and the future-directed unit normal $\hat{N}$ to $\mathring{\Sigma}$ is assumed to have the components $\hat{N}^{\mu}= \delta_0^{\mu}.$ We note that the analysis of Section \ref{SSS:InitialDataOriginalSystem} implies that under the equation of state $p=\speed^2 \rho,$ we have that $T_{00} |_{\mathring{\Sigma}} = \frac{1 + \speed^2}{\speed^2}\mathring{p} (\mathring{u}_0)^2 - \mathring{p},$ and $T_{0j}|_{\mathring{\Sigma}}= \frac{1 + \speed^2}{\speed^2}\mathring{p} \mathring{u}_0 \underline{\mathring{u}}_j,$ where $\mathring{u}_0 = - \sqrt{1 + \mathring{\underline{g}}_{ab}\underline{\mathring{u}}^a \underline{\mathring{u}}^b}.$

\begin{remark}
In this article, we do not address the issue of solving the constraint equations for the system 
\eqref{E:EinsteinIntro} - \eqref{E:Fluidenergymomentumintro}.
\end{remark}

The fact $(2)$ from above was settled by Choquet-Bruhat and Geroch in 1969 \cite{cBgR1969}, who showed that every initial data set satisfying the constraints \eqref{E:Gaussintro} - \eqref{E:Codazziintro} launches a unique \emph{maximal globally hyperbolic development}. Roughly speaking, this is the largest spacetime solution to the Einstein equations that is uniquely determined by the data (see Section \ref{SS:IVP} for additional details). This is still a local existence result in the sense that the resulting spacetime may contain singularities. In particular, causal geodesics may terminate, which in physics terminology means that an observer (light ray in the case of null geodesics) may run into the end of spacetime in finite affine parameter. Such a singularity corresponds to the breakdown of the deterministic nature of the laws of physics. For spacetimes launched by initial data near that of the FLRW background solutions, \emph{our main result rules out the possibility of these singularities} for observers (light rays) traveling in the ``future direction.''

\subsection{Comparison with previous work}

A noteworthy feature of our main result is that it shows that fluids behave very differently on exponentially expanding backgrounds than they do in flat spacetime. More specifically, if one fixes a background metric on $[0,\infty) \times \mathbb{T}^3$ near $\widetilde{g}_{\mu \nu},$ then our proof can be easily adapted to show that the relativistic Euler equations \eqref{E:Fluidenergymomentumintro} on this expanding background with $0 < \speed < \sqrt{1/3}$ have global solutions arising from data that are close to that of an initial uniform quiet fluid state. In contrast, Christodoulou's monograph \cite{dC2007} shows that on the Minkowski space background, shock singularities can form in solutions to the Euler equations arising from data that are arbitrarily close to that of a uniform quiet fluid state, even if the solution is irrotational. Thus, we state the following as a corollary of our proof: \emph{exponentially expanding spacetimes can stabilize perfect fluids}.

The first author to obtain global stability results for the Einstein equations with $\Lambda > 0$ was Helmut Friedrich, first for $1 + 3-$dimensional vacuum spacetimes \cite{hF1986a}, then later for the Einstein-Maxwell and Einstein-Yang-Mills systems \cite{hF1991}. He developed a technique known as the \emph{conformal method}, which translates (via a suitable change of variables) the question of the global stability of a solution to the Einstein equations into the question of local stability for the \emph{conformal field equations}. Friedrich's work on vacuum spacetimes was later extended by Anderson \cite{mA2005} to apply to $1 + n$ dimensional manifolds, where $n$ is odd. Unfortunately, not all matter models seem to be amenable to the conformal method. 

An alternative to the conformal method, which was provided by Ringstr\"{o}m in \cite{hR2008}, inspired our work on the irrotational case \cite{iRjS2009}. In \cite{hR2008}, Ringstr\"{o}m used a wave coordinate approach to show the existence of an open family of future causally geodesically complete solutions the Einstein-scalar field system. In \cite{hR2008}, the scalar field is postulated to satisfy the equation $g^{\alpha \beta} D_{\alpha} D_{\beta} \Phi = V(\Phi)$ with associated energy-momentum tensor $T_{\mu \nu} = (\partial_{\mu} \Phi)(\partial_{\nu} \Phi) - \big[\frac{1}{2} g^{\alpha \beta} (\partial_{\alpha} \Phi) (\partial_{\beta} \Phi) + V(\Phi)\big]g_{\mu \nu},$ where $D$ is the Levi-Civita connection corresponding to $g,$ and $V'(\Phi)$ is a nonlinearity satisfying $V(0) > 0,$ $V'(0) = 0,$ $V''(0) > 0.$ Although the cosmological constant is set equal to $0$ in \cite{hR2008}, these three conditions allow the nonlinearity $V$ to emulate the presence of a \emph{positive} cosmological constant, if $\Phi$ is sufficiently small. In fact, the case $\Phi \equiv 0$ is equivalent to the vacuum Einstein equations with a cosmological constant $\Lambda = V(0).$ One of the main advantages of Ringstr\"{o}m's framework is that it can be adapted to handle a wide variety of matter models. Its robustness is what prompted us to adapt it to the problem of interest in this article.
 
Our results were further motivated by \cite{uBaRoR1994}, in which Brauer, Rendall, and Reula showed a Newtonian analogue of our main result. More specifically, they studied Newtonian cosmological models\footnote{Their models were based on Newton-Cartan theory, which is a slight generalization of ordinary Newtonian gravitational theory that can be endowed with a highly geometric interpretation.} with a positive cosmological constant and with perfect fluid sources under the equation of state $p = C \rho^{\gamma},$ where $\rho \geq 0$ is the density, $C > 0$ is a constant, and $\gamma > 1$ is a constant. They showed that small perturbations of a uniform quiet fluid state of constant positive density lead to a global solution. The fact that their work did not require the fluid to be irrotational strongly suggested that our results from \cite{iRjS2009} could be generalized to the \textbf{Main Results} stated above.

Finally, we remark that the mechanism of stability for the problem at hand is of a very different nature than that encountered in the well-known body of work on the stability of the Minkowski spacetime solution (the original proof is \cite{dCsK1993}, while alternate proofs/extensions can be found in e.g. \cite{lBnZ2009}, \cite{sKfN2003}, \cite{jL2009}, \cite{hLiR2010}, \cite{jS2010b}), which features $\Lambda = 0.$ In the latter case, there is a delicate competition between the size of the nonlinear terms and the dispersive nature of solutions to the corresponding linearized equations, while in the present case, we are aided by the energy-dissipative effect induced by $\Lambda > 0.$ See \cite{hR2008} for a more thorough comparison.

\subsection{Comments on the analysis} \label{SS:Commentsonanalysis}
We now discuss our strategy for proving global existence. Our analysis of the metric components $g_{\mu \nu}$ is essentially the same as in \cite{hR2008} and \cite{iRjS2009}, but our analysis of the fluid variables differs markedly from the analysis of the scalar fluid equation in \cite{iRjS2009}, and from the analysis encountered in \cite{hR2008}. Let us first discuss a simple model problem that captures the spirit of our analysis of the $g_{\mu \nu}.$ Consider the inhomogeneous wave equation $g^{\alpha \beta} D_{\alpha} D_{\beta} v = F$ for the model metric $g = -dt^2 + e^{2t} \sum_{a=1}^3 (dx^a)^2$ on the manifold-with-boundary $\mathcal{M} = [0,\infty) \times \mathbb{T}^3.$ Here, we are using a standard local coordinate
system $x^1, x^2, x^3$ on $\mathbb{T}^3.$ Simple calculations imply that relative to this coordinate system, the wave equation can be expressed as follows:

\begin{align} \label{E:modelequation}
	- \partial_t^2 v + e^{-2t} \delta^{ab} \partial_a \partial_b v = 3 (\partial_t v)^2 + F.
\end{align}
To estimate solutions to \eqref{E:modelequation}, one can define the ``usual'' energy $E(t) \geq 0$ by
$E^2(t) = \frac{1}{2} \int_{\mathbb{T}^3} (\partial_t v)^2 + e^{-2t} \delta^{ab} (\partial_a v)(\partial_b v) \, d^3 x,$ and a standard integration by parts argument leads to the inequality

\begin{align} \label{E:modelenergyinequality}
	\frac{d}{dt} E \leq -E + \|F\|_{L^2}.
\end{align}
It is clear from \eqref{E:modelenergyinequality} that sufficient estimates of $\|F\|_{L^2}$ in terms of $E$ will lead to 
energy decay\footnote{In our work below, we work with rescaled energies that are approximately constant in time.}. 

In Section \ref{SS:Decomposition}, we decompose our modified version of the Euler-Einstein system in order to better see the structure of its terms. The modified system is constructed in a manner such that the $g_{\mu \nu}$ satisfy equations with ``principal terms'' that are similar to the terms model equation \eqref{E:modelequation}. By ``principal,'' we are not referring to the degree of differentiability, but rather to the effect that these terms have on the solutions. As is typical in the theory of nonlinear hyperbolic PDEs, most of our effort goes towards analyzing the remaining ``error terms,'' which are analogous to the inhomogeneous term $F$ in the model equation. This analysis, which is carried out in Section \ref{S:BootstrapConsequences}, justifies the division into principal and error terms. We remark that the main tools used for estimating the inhomogeneous terms are Sobolev-Moser type estimates, which we have placed in the Appendix for convenience.

In \cite{iRjS2009}, our analysis of the irrotational fluid equation \eqref{E:IrrotationalFluid}, was in the spirit of the above discussion. The new aspect compared to Ringstr\"{o}m's work \cite{hR2008} was that the effective metric for \eqref{E:IrrotationalFluid} is not the inverse spacetime metric $g^{\mu \nu},$ but is instead the \emph{reciprocal acoustical metric} $(m^{-1})^{\mu \nu},$ an inverse Lorentzian metric that features a family of \emph{sound cones} as acoustical null hypersurfaces. The main challenge in \cite{iRjS2009} was that the acoustical metric can lead to instabilities if the pressure decays too quickly. In particular, it can degenerate if the pressure becomes $0$ at a point. Addressing this difficulty was unavoidable, since the FLRW background solutions feature pressures that exponentially decay towards $0$ in time. To close our global existence argument, we had to ensure that the pressure of a perturbed solution decayed at the same rate as the background pressure, and in particular, that the pressure never becomes $0$ in finite time.

Our analysis of the relativistic Euler equations \eqref{E:Fluidenergymomentumintro} is based on energy estimates derived from \emph{energy currents}. The energy currents $\dot{J}^{\mu}$ (see \eqref{E:EnergyCurrent}) are vectorfields that can be used via the divergence theorem to control the evolution of the $L^2$ norms of the fluid variables and their derivatives. This approach to analyzing the Euler equations differs from the well-known \emph{symmetric hyperbolic} framework (see e.g. \cite{rCdH1962}, \cite{cD2010}, \cite{kF1954}) in that it allows the analysis to take place directly in the \emph{Eulerian} variables $(p,u)$ rather than an artificially introduced collection of unphysical state-space variables; see the introduction of \cite{jS2008b} for a more detailed discussion on the merits of the energy current framework. Our energy currents ultimately owe their existence to the fact that the relativistic Euler equations are a hyperbolic system of PDEs derivable from a Lagrangian. That is, the relativistic Euler equations are the Euler-Lagrange equations corresponding to a Lagrangian. For such hyperbolic systems, Christodoulou has developed a geometric-analytic framework \cite{dC2000} for deriving energy estimates based on the availability of energy currents. In fact, the use of this technique in the context of deriving energy estimates for the relativistic Euler equations was first introduced in \cite{dC2000}. However, there is a subtlety: the Lagrangian formulation of relativistic fluid dynamics described in \cite{dC2000} naturally has as its unknowns the so-called \emph{Lagrangian} variables, while in this article, we takes our unknowns to be the Eulerian variables. In order to derive energy currents in the Eulerian variables, the natural way to proceed is to first derive them in the Lagrangian variables, and then translate them into expressions involving the Eulerian variables. This is a decidedly nontrivial task since the transformation from Lagrangian to Eulerian variables is nonlinear. Further complicating matters is the fact that the relativistic Euler equations in Lagrangian variables suffer from a degeneracy that renders them just outside of the scope of the framework of \cite{dC2000}.
In particular, the energy currents derived in \cite{dC2000} are not fully $L^2$ coercive, but are instead only semi-coercive. Nonetheless, one can still derive energy currents that are $L^2-$ coercive in the Eulerian variables. These issues are themselves of interest, and we will provide a detailed discussion of them in a future article. In the present article, we restrict our attention to adapting the Eulerian-variable energy currents used by Christodoulou in \cite{dC2007}, and by the author in \cite{jS2008b}, \cite{jS2008a}, without fully explaining their origin or motivating their coerciveness properties.

Finally, we would like to make some comments on the assumption $0 < \speed < \sqrt{1/3}$ stated in the \textbf{Main Results} above. In order to do this, we have to provide some details concerning the behavior of the fluid norms 
$U_{N - 1}(t),$ $\fluidnorm{N}(t),$ and the total solution norm $\totalnorm{N}(t),$ which are defined in \eqref{E:ujnorm}, \eqref{E:fluidnorm}, and \eqref{E:totalnorm} respectively. Here, $N \geq 3$ is an integer. Now if $\decayparameter < 1$ (i.e., $\speed < \sqrt{1/3}$), then the bootstrap assumption $U_{N-1}(t) \leq \epsilon$ plus Sobolev embedding together imply that $\| u^j \|_{L^{\infty}} \leq C e^{-(1 + q)H t} \epsilon,$ where the small positive constant $q$ is defined in \eqref{E:qdef}. From this decay estimate, it follows that the length of the spatial part of $u,$ as measured by the background metric $\widetilde{g},$ is exponentially decaying toward the quiet background fluid state: from standard Sobolev-Moser estimates (see Appendix \ref{A:SobolevMoser}) and the bootstrap assumption $\fluidnorm{N}(t) \leq \epsilon,$ it follows that $\| \widetilde{g}_{ab} u^a u^b \|_{H^N} = \|e^{2 \Omega} \delta_{ab} u^a u^b \|_{H^N}  
\leq C e^{2Ht} \sum_{j=1}^3 \|u^j\|_{L^{\infty}} \|u^j\|_{H^N} \leq C e^{-qHt} \fluidnorm{N}(t) \leq C \epsilon e^{-qHt};$ this last estimate shows that the four-velocity exponentially decays towards the quiet background state, which is $\widetilde{u}^j \equiv 0,$ $(j = 1,2,3).$ In contrast, the bootstrap assumption $\fluidnorm{N}(t) \leq \epsilon$ would by itself only lead to the estimate $\| u^j \|_{L^{\infty}} \leq C \epsilon e^{-Ht},$ which would lead to the estimate $\| \widetilde{g}_{ab} u^a u^b \|_{H^N} \leq C \epsilon.$ Note that this weaker estimate does not imply that the four-velocity is decaying towards the uniform quiet background state. In fact, the sole reason we introduce the quantity $U_{N-1}(t)$ is to capture the additional $L^{\infty}$ exponential decay of the four-velocity. As we will see in 
Proposition \ref{P:Nonlinearities}, the rapid decay of the fluid towards the quiet state plays a crucial role in many of our estimates. 

It is important to note that the above reasoning breaks down when $\decayparameter \geq 1$ (the degeneracy that occurs in the case $\speed = 0$ has been discussed above). We first remark that \emph{the influence of the positive cosmological constant on the behavior of the fluid four-velocity is essentially captured by the dissipative $(\decayparameter - 2) \omega u^j$ term on the right-hand side of} equation \eqref{E:partialtuj}. When $\decayparameter \geq 1,$ equation \eqref{E:partialtuj} suggests (we are assuming that $\triangle^{'j}$ is an error term) that we would at best be able to prove an $L^{\infty}$ estimate of the form $\| u^j(t) \|_{L^{\infty}} \leq C \epsilon e^{-(\decayparameter - 2)Ht};$ i.e., we no longer expect to have an improved estimate of the form $\| u^j \|_{L^{\infty}} \leq C \epsilon e^{-(1 + q)H t}$ (for some $q > 0$), and this lack of improvement reverberates throughout the remaining estimates. For example, the $g_{ab} u^a u^b$ term in \eqref{E:u0HNfirstinequality} is no longer expected to decay, so that inequality \eqref{E:u0upperHN} must be downgraded to at best $\|u^0 - 1\|_{H^N} \leq C \totalnorm{N}.$ By Sobolev embedding, we conclude that at best the downgraded estimate $\|u^0 - 1\|_{L^{\infty}} \leq C \epsilon$ holds. This weakened estimate returns to haunt us in equation \eqref{E:divergenceofdotJagain}, which involves the term $\frac{2(1 + \speed^2)P (u^0 - 1)}{\speed^2} (\partial_t g_{ab}) \dot{u}^a \dot{u}^b.$ This now-fatal term leads to the following downgraded version of inequality \eqref{E:ENintegral}: $\fluidenergy{N}^2(t) \leq \fluidenergy{N}^2(t_1) + C \int_{t_1}^t \totalenergy{N}^2(\tau) \, d \tau,$ in which we have lost the availability of the $e^{-q H \tau}$ damping factor. Here, $\fluidenergy{N}$ is the fluid energy defined in \eqref{E:fluidenergydef}, and $\totalenergy{N}$ is the total solution energy defined in \eqref{E:totalenergy}. By Proposition \ref{P:energynormomparison}, $\fluidenergy{N} \approx \fluidnorm{N}$ and $\totalenergy{N} \approx \totalnorm{N}$ for perturbations of the FLRW solutions. This resulting weakened estimate allows for the possibility that $\fluidenergy{N}(t)$ may grow unabatedly in time, which would eventually destroy the bootstrap assumption $\fluidnorm{N}(t) \leq \epsilon.$ We remark that \emph{these difficulties have nothing to do with the fact that we are studying the coupled system}: they would remain even if we were studying the relativistic Euler equations on the fixed spacetime background $\big((-\infty,\infty) \times \mathbb{T}^3, \widetilde{g}_{\mu \nu} \big).$

\subsection{Outline of the structure of the paper} 

\begin{itemize}
	\item In Section \ref{S:Notation}, we describe our conventions for indices and introduce some notation for
		differential operators and Sobolev norms.
	\item In Section \ref{S:EE}, we introduce the Euler-Einstein system.
	\item In Section \ref{S:backgroundsolution}, we use a standard ODE ansatz to derive a well-known family of 
		background Friedmann-Lema\^{\i}tre-Robertson-Walker (FLRW) solutions to the Euler-Einstein system.
	\item In Section \ref{S:ReducedEquations} we introduce our version of wave coordinates and use algebraic identities 
		valid in such a coordinate system to construct a modified version of the Euler-Einstein system. We then
		discuss how to construct data for the modified system from data for the unmodified system 
		in a manner that is compatible with the wave coordinate condition. Next, we discuss 
		classical local existence for the modified system, including the continuation principle that is used in
		Section \ref{S:GlobalExistence}. We also sketch a proof of the fact that the modified system is equivalent to 
		the unmodified system in a wave coordinate system.
	\item In Section \ref{S:NormsandEnergies}, we introduce the relevant norms and the related energies for the modified system 
		that we use in our global existence argument. In order to construct the fluid energies, we introduce the fluid
		equations of variation and the fluid energy currents. We also provide a preliminary analysis of the derivatives of the 
		energies, but the inhomogeneous terms are not estimated until Section \ref{S:BootstrapConsequences}.
	\item In Section \ref{S:LinearAlgebra}, we introduce some bootstrap assumptions on the spacetime metric $g_{\mu \nu}.$
		We then use these assumptions to provide some linear-algebraic lemmas that are useful for
		analyzing $g_{\mu \nu}$ and the inverse metric $g^{\mu \nu}.$
	\item In Section \ref{S:BootstrapAssumptions}, we introduce our main bootstrap assumption, which is a smallness condition
		on $\totalnorm{N},$ a norm of difference between the perturbed solution and the background solution. We also define
		the positive constants $q$ and $\upeta_{min},$ which play a fundamental role in the technical estimates of the following
		sections.
	\item Section \ref{S:BootstrapConsequences} contains most of the technical estimates. We assume the bootstrap assumptions 
		from the previous sections and use them to deduce estimates for $g_{\mu \nu},$ $g^{\mu \nu},$ $p,$ $u^{\mu},$ and for the 
		nonlinearities appearing in the modified equations. 
	\item In Section \ref{S:EnergyNormEquivalence}, we show that the Sobolev norms and energies defined in  
		Section \ref{S:NormsandEnergies} are equivalent.
	\item In Section \ref{S:GlobalExistence}, we use the estimates from the previous sections to prove our main theorem, which is 
		a small-data global (where ``small'' means close to the background solution) existence result for the modified equations. 
		Our theorem shows that initial data satisfying the Euler-Einstein constraints, 
		the wave coordinate condition, and the smallness condition lead to a future-geodesically complete solution of the 
		unmodified Euler-Einstein system. 
	\item In Section \ref{S:Asymptotics}, we provide a theorem concerning the convergence of the 
		global solutions as $t \to \infty.$
\end{itemize}

\section{Notation} \label{S:Notation}

In this section, we briefly introduce some notation that we use in this article.

\subsection{Index conventions}
Greek indices $\alpha, \beta, \cdots$ take on the values $0,1,2,3,$ while Latin indices $a,b,\cdots$ 
(which we sometimes call ``spatial indices'') take on the values $1,2,3.$ Pairs of repeated indices, with one raised and one lowered, are summed (from $0$ to $3$ if they are Greek, and from $1$ to $3$ if they are Latin). We raise and lower indices with the spacetime metric $g_{\mu \nu}$ and its inverse $g^{\mu \nu}.$ Exceptions to this rule include the constraint equations \eqref{E:Gaussintro} - \eqref{E:Codazziintro} and \eqref{E:Gauss} - \eqref{E:Codazzi}, in which we use the $3-$metric $\mathring{\underline{g}}_{jk}$ and its inverse $\mathring{\underline{g}}^{jk}$ to lower and raise indices, and in Section \ref{S:Asymptotics}, in which all indices are lowered and raised with $g_{\mu \nu}$ and $g^{\mu \nu}$ except for the $3-$metric $g_{jk}^{(\infty)},$ which has $g_{(\infty)}^{jk}$ as its corresponding inverse metric.

\subsection{Coordinate systems and differential operators}
Throughout this article, we work in a standard local coordinate system $x^1,x^2,x^3$ on $\mathbb{T}^3.$
Although strictly speaking this coordinate system is not globally well-defined, the vectorfields $\partial_j \eqdef \frac{\partial}{\partial x^{j}}$ are globally well-defined. This coordinate system extends to a local coordinate system $x^0,x^1,x^2,x^3$ on manifolds with boundary of the form $\mathcal{M} =[0,T) \times \mathbb{T}^3,$ and we often write $t$ instead of $x^0.$ In this local coordinate system, the background FLRW metric $\widetilde{g}$ is of the form \eqref{E:backgroundmetricform}. We write $\partial_{\mu}$ to denote the coordinate derivative $\frac{\partial}{\partial x^{\mu}},$ and we often write $\partial_t$ instead of $\partial_0.$ Throughout the article, we will perform all of our computations with respect to the fixed frame $\big\lbrace \partial_{\mu} \big\rbrace_{\mu = 0,1,2,3}.$

If $\vec{\alpha} = (n_1,n_2,n_3)$ is a triplet of non-negative integers, then we define the spatial multi-index coordinate differential operator $\partial_{\vec{\alpha}}$ by $\partial_{\vec{\alpha}} \eqdef \partial_1^{n_1} \partial_2^{n_2} \partial_3^{n_3}.$ We denote the order of $\vec{\alpha}$ by $|\vec{\alpha}|,$ where $|\vec{\alpha}| \eqdef n_1 + n_2 + n_3.$

We write 

\begin{align}
	D_{\mu} T_{\mu_1 \cdots \mu_s}^{\nu_1 \cdots \nu_r} = 
		\partial_{\mu} T_{\mu_1 \cdots \mu_s}^{\nu_1 \cdots \nu_r} + 
		\sum_{a=1}^r \Gamma_{\mu \ \alpha}^{\ \nu_{a}} T_{\mu_1 \cdots \mu_s}^{\nu_1 \cdots \nu_{a-1} \alpha \nu_{a+1} \nu_r} - 
		\sum_{a=1}^s \Gamma_{\mu \ \mu_{a}}^{\ \alpha} T_{\mu_1 \cdots \mu_{a-1} \alpha \mu_{a+1} \mu_s}^{\nu_1 \cdots \nu_r}
\end{align} 
(where $\Gamma_{\mu \ \nu}^{\ \alpha}$ is defined in \eqref{E:EMBIChristoffeldef}) to denote the components of the covariant derivative of a tensorfield on $\mathcal{M}$ with components (relative to the coordinate frame introduced above) 
$T_{\mu_1 \cdots \mu_s}^{\nu_1 \cdots \nu_r}.$ 

We write $\partial^{(N)} T_{\mu_1 \cdots \mu_s}^{\nu_1 \cdots \nu_r}$ to denote the array
containing of all of the $N^{th}$ order \emph{spacetime} coordinate derivatives (including time derivatives) of the component $T_{\mu_1 \cdots \mu_s}^{\nu_1 \cdots \nu_r}.$ Similarly, we write $\underpartial^{(N)} T_{\mu_1 \cdots \mu_s}^{\nu_1 \cdots \nu_r}$ to denote the array containing of all $N^{th}$ order \emph{spatial coordinate} derivatives of the component $T_{\mu_1 \cdots \mu_s}^{\nu_1 \cdots \nu_r}.$ We omit
the superscript $^{(N)}$ when $N=1.$

\subsection{Identification of spacetime tensors and spatial tensors} \label{SS:IDENTIFICATIONS}

We will often view $\mathbb{T}^3$ as an embedded submanifold of the spacetime $\mathcal{M}$ under an embedding
$\iota_t$ of the form $\iota_t: \mathbb{T}^3 \rightarrow \lbrace t \rbrace \times \mathbb{T}^3 \subset \mathcal{M},$ 
$\iota_t(x^1,x^2,x^3) \eqdef (t,x^1,x^2,x^3).$ Note that the embedding is a diffeomorphism between $\mathbb{T}^3$ and $\lbrace t \rbrace \times \mathbb{T}^3.$ We will often suppress the embedding by identifying $\mathbb{T}^3$ with its image $\iota_t(\mathbb{T}^3).$ Furthermore, if $T_{j_1 \cdots j_s}^{j_1 \cdots j_r}$ is a $\mathbb{T}^3-$inherent ``spatial'' tensorfield, then there is a unique ``spacetime'' tensorfield  $T_{\mu_1 \cdots \mu_s}^{'\nu_1 \cdots \nu_r}$ defined along $\iota_t (\mathbb{T}^3) \simeq \mathbb{T}^3$ such that $\iota_t^* T' = T$ and such that $T'$ is tangent\footnote{Recall that $T_{\mu_1 \cdots \mu_s}^{'\nu_1 \cdots \nu_r}$ is tangent to $\iota_t (\mathbb{T}^3)$ if any contraction of any upstairs (downstairs) index with the unit normal covector $n_{\mu}$ (unit normal vector $n^{\mu}$) results in $0;$ for downstairs indices, this notion depends on the spacetime metric $g_{\mu \nu}.$} to $\iota_t(\mathbb{T}^3).$ Here $\iota_t^*$ denotes
the pullback by $\iota_t.$ We will often identify $T$ with $T',$ and use the same symbol to denote both, e.g. $T_{j_1 \cdots j_s}^{j_1 \cdots j_r} \simeq T_{\mu_1 \cdots \mu_s}^{\nu_1 \cdots \nu_r};$ we especially apply this convention to the 
initial data. All of these standard identifications should be clear in context.

\subsection{Norms} \label{SS:NORMS}
All of the Sobolev norms we use are defined relative to the local coordinate system $x^1,x^2,x^3$ on $\mathbb{T}^3$ introduced above. We remark that our norms are not coordinate invariant quantities, since we work with the norms of the \emph{components} of tensorfields relative to this coordinate system. If $f$ is a function defined on the hypersurface 
$\lbrace x \in \mathcal{M} \ | \ t = const \rbrace \simeq \mathbb{T}^3,$ then relative to this coordinate system, we define the standard Sobolev norm $\big\| f \big\|_{H^N}$ as follows, where $d^3 x \eqdef dx^1 dx^2 dx^3:$

\begin{align} \label{E:SobolevNormDef}
	\big\| f \big\|_{H^N} \eqdef
		\bigg( \sum_{|\vec{\alpha}| \leq N} 
		\int_{\mathbb{T}^3} \big|\partial_{\vec{\alpha}} f(t,x^1,x^2,x^3) \big|^2 
		d^3x \bigg)^{1/2}.
\end{align}
The symbol $d^3x$ represents a slight abuse of notation\footnote{A precise definition of the norm \eqref{E:SobolevNormDef} would involve the use of an atlas on $\mathbb{T}^3$ and a partition of unity that is subordinate to the corresponding covering.} since the coordinate system $x^1,x^2,x^3$ is not globally well-defined on $\mathbb{T}^3$ (even though the differential operators $\partial_{\vec{\alpha}}$ are). 

Using the above notation, we can write the $N^{th}$ order homogeneous Sobolev norm of $f$ as

\begin{align}
	\big\| \underpartial^{(N)} f \big\|_{L^2} \eqdef
		\sum_{|\vec{\alpha}| = N} \big\| \partial_{\vec{\alpha}} f \big\|_{L^2}.
\end{align}

If $\mathfrak{K} \subset \mathbb{R}^n$ or $\mathfrak{K} \subset \mathbb{T}^n,$ then $C^N_b(\mathfrak{K})$ denotes the set of $N-$times continuously differentiable functions (either scalar or array-valued, depending on context) on the interior of
$\mathfrak{K}$ with bounded derivatives up to order $N$ that extend continuously to the closure of $\mathfrak{K}.$ The norm of a function $F 
    \in C^N_b(\mathfrak{K})$ is defined by
    \begin{equation} \label{E:CbMNormDef}
        |F|_{N,\mathfrak{K}} \eqdef \sum_{|\vec{I}|\leq N} \mbox{ess} \sup_{\cdot \in \mathfrak{K}}
        |\partial_{\vec{I}}F(\cdot)|,
    \end{equation}
    where $\partial_{\vec{I}}$ is a multi-indexed operator representing repeated partial differentiation
    with respect to the arguments $\cdot$ of $F,$ which may be either spacetime coordinates or metric/fluid components
    depending on context. When $N=0,$ we also use the notation
    
    \begin{align} 
    	|F|_{\mathfrak{K}} \eqdef \mbox{ess} \sup_{\cdot \in \mathfrak{K}} |F(\cdot)|.
    \end{align}
    Furthermore, we use the notation
    \begin{align} 
    	|F^{(N)}|_{\mathfrak{K}} \eqdef \sum_{|\vec{I}| = N}|\partial_{\vec{I}} F|_{\mathfrak{K}}.
    \end{align}
    The quantity $|F^{(N)}|_{\mathfrak{K}}$ is a measure of the size of the $N^{th}$ order derivatives of $F$ 
    on the set $\mathfrak{K}.$ In the case that $\mathfrak{K} = \mathbb{T}^3,$ we sometimes use the more familiar
    notation 
    
    \begin{align}
    	\| F \|_{L^{\infty}} & \eqdef \mbox{ess} \sup_{x \in \mathbb{T}^3} |F(x)|, \\
    	\| F \|_{C_b^N} & \eqdef \sum_{|\vec{\alpha}| \leq N} \big\| \partial_{\vec{\alpha}} F \|_{L^{\infty}}.
    \end{align}
    
	 	If $A$ is an $m \times n$ array-valued function with entries $A_{jk},$ 
	 	$(1 \leq j \leq m, 1\leq k \leq n),$ then in Section \ref{S:BootstrapConsequences}, we write e.g.
    $\| A \|_{H^N}$ to denote the $m \times n$ \emph{array} whose entries are $\|A_{jk} \|_{H^N}.$ We use similar notation for 
    other norms of $A.$ 
		
		If $I \subset \mathbb{R}$ is an interval and $X$ is a normed function space, then we use the notation
    $C^N(I,X)$ to denote the set of $N$-times continuously differentiable maps from $I$ into $X.$

\subsection{Running constants} \label{SS:runningconstants}
We use $C$ to denote a running constant that is free to vary from line to line. In general, it can depend on $N$ (see \eqref{E:Ndef}) and $\Lambda,$ but can be chosen to be independent of all functions $(g_{\mu \nu}, P, u^j),$ $(\mu, \nu = 0,1,2,3),$ $(j=1,2,3),$ that are sufficiently close to the background solution $(\widetilde{g}_{\mu \nu},	\widetilde{P}, \widetilde{u}^j)$ of Section \ref{S:backgroundsolution}. We sometimes use notation such as $C(N)$ to indicate the dependence of $C$ on quantities that are peripheral to the main argument. Occasionally, we use $c,$ $C_*,$ etc., to denote a constant that plays a distinguished role in the analysis. We remark that many of the constants blow-up as $\Lambda \rightarrow 0^+.$

\subsection{A warning on the sign of \texorpdfstring{$\hat{\Square}_g$}{}}

Although we often choose notation that agrees with the notation used by Ringstr\"{o}m in \cite{hR2008}, our reduced wave operator $\hat{\square}_g \eqdef g^{\alpha \beta} \partial_{\alpha} \partial_{\beta}$ has the opposite sign of the one in \cite{hR2008}. It has the same sign as the one used in \cite{iRjS2009}.

\section{The Euler-Einstein System} \label{S:EE}

In this section, we provide a detailed introduction to the Euler-Einstein system, including the notion of a perfect fluid. We then discuss several aspects of the initial value problem formulation of the Euler-Einstein system, including the initial data and the \emph{maximal globally hyperbolic development} of the data.

\subsection{Introduction} \label{SS:EEIntro}

The Einstein equations connect the \emph{Einstein tensor} $\mbox{Ric}_{\mu \nu} - \frac{1}{2}g_{\mu \nu} R,$
which contains information about the curvature of the spacetime $(\mathcal{M},g_{\mu \nu})$ 
to the \emph{energy-momentum tensor} $T_{\mu \nu},$ which models the matter content of spacetime:

\begin{align} \label{E:EinsteinField}
	\mbox{Ric}_{\mu \nu} - \frac{1}{2}R g_{\mu \nu} + \Lambda g_{\mu \nu} = T_{\mu \nu},
\end{align}
where $\mbox{Ric}_{\mu \nu}$ is the \emph{Ricci curvature tensor}, $R$
is the \emph{scalar curvature}, and $\Lambda$ is the \emph{cosmological constant}. We remark that \textbf{the stability results proved in this article heavily depend upon the assumption} $\Lambda > 0.$ Recall that the Ricci curvature tensor and scalar curvature are defined in terms of the \emph{Riemann curvature tensor}\footnote{Under our sign convention, $D_{\mu} D_{\nu} X_{\alpha} - D_{\nu} D_{\mu} X_{\alpha} = \mbox{Riem}_{\mu \nu \alpha}^{\ \ \ \ \beta} X_{\beta}.$} $\mbox{Riem}_{\mu \alpha \nu}^{\ \ \ \ \beta}$, which can be expressed in terms of the \emph{Christoffel symbols} $\Gamma_{\mu \ \nu}^{\ \alpha}$ of the metric. In an arbitrary local coordinate system, these quantities can be expressed as follows:

\begin{subequations}
\begin{align}
\mbox{Riem}_{\mu \alpha \nu}^{\ \ \ \ \beta} & \eqdef 
	\partial_{\alpha} \Gamma_{\mu \ \nu}^{\ \beta}
	- \partial_{\mu} \Gamma_{\alpha \ \nu}^{\ \beta}
  + \Gamma_{\alpha \ \lambda}^{\ \beta} \Gamma_{\mu \ \nu}^{\ \lambda}
	- \Gamma_{\mu \ \lambda}^{\ \beta} \Gamma_{\alpha \ \nu}^{\ \lambda},	
	\label{E:Riemanndef} \\
\mbox{Ric}_{\mu \nu} & \eqdef \mbox{Riem}_{\mu \alpha \nu}^{\ \ \ \ \alpha} 
	= \partial_{\alpha} \Gamma_{\mu \ \nu}^{\ \alpha}
	- \partial_{\mu} \Gamma_{\alpha \ \nu}^{\ \alpha}
  + \Gamma_{\alpha \ \lambda}^{\ \alpha} \Gamma_{\mu \ \nu}^{\ \lambda}
	- \Gamma_{\mu \ \lambda}^{\ \alpha} \Gamma_{\alpha \ \nu}^{\ \lambda},
	\label{E:Riccidef} \\
R & \eqdef g^{\alpha \beta} \mbox{Ric}_{\alpha \beta}, \label{E:Rdef} \\
\Gamma_{\mu \ \nu}^{\ \alpha} & \eqdef \frac{1}{2} g^{\alpha \lambda}(\partial_{\mu} g_{\lambda \nu} 
	+ \partial_{\nu} g_{\mu \lambda} - \partial_{\lambda} g_{\mu \nu}). \label{E:EMBIChristoffeldef}
\end{align}
\end{subequations}

The twice-contracted Bianchi identities (see e.g. \cite{rW1984}) are

\begin{align}
	D_{\alpha} \mbox{Ric}^{\alpha \mu} - \frac{1}{2}D^{\mu} R & = 0,&& (\mu = 0,1,2,3),
\end{align}
which by \eqref{E:EinsteinField} leads to the following equations satisfied by $T^{\mu \nu}:$ 

\begin{align} \label{E:Tdivergence}
	D_{\alpha} T^{\alpha \mu} & = 0, && (\mu = 0,1,2,3).
\end{align}

We will now briefly introduce the notion of a perfect fluid. For a more detailed introduction, including a very interesting account of the history of the subject, readers may consult Christodoulou's survey article \cite{dC2007b}. The energy-momentum tensor for a perfect fluid is

\begin{align} \label{E:Tlowerfluid}	
	T^{\mu \nu} = (\rho + p) u^{\mu} u^{\nu} + p g^{\mu \nu},
\end{align}
where $\rho \geq 0$ is the \emph{proper energy density}, $p \geq 0$ is the \emph{pressure}, and $u$
is the \emph{four-velocity}, a unit-length (i.e., $u_{\alpha} u^{\alpha} = -1)$ future-directed vectorfield
on $\mathcal{M}.$ The \emph{Euler equations}, which are the laws of motion for a perfect relativistic fluid, are the four equations \eqref{E:Tdivergence} together with a conservation law \eqref{E:Eulernumberdensity} for the number of fluid elements. In an arbitrary local coordinate system, they can be expressed as follows:

\begin{subequations}
\begin{align}	
  D_{\alpha} T^{\alpha \mu} & = 0,&& (\mu = 0,1,2,3), \label{E:Euler} \\
	D_{\alpha} (n u^{\alpha}) & = 0,&& \label{E:Eulernumberdensity}
\end{align}
\end{subequations}
where $n$ is the \emph{proper number density} of the fluid elements. We also introduce the thermodynamic variable $\Ent,$ the \emph{entropy per particle}, which we will discuss below.

Equations \eqref{E:Euler} - \eqref{E:Eulernumberdensity} do not form a closed system, even in a fixed spacetime $(\mathcal{M},g_{\mu \nu});$ there are too many fluid variables, and not enough equations. The standard way of remedying this difficulty is to appeal the laws of thermodynamics, which imply the following relationships between the fluid variables (see e.g. \cite{yGsTZ1998}, \cite{dC2007b}, \cite{dC2007}):    
    
    \begin{enumerate}
  		\item $\rho \geq 0$ is a function of $n \geq 0$ and ${\Ent} \geq 0.$
  		\item $p \geq 0$ is defined by
        \begin{align} \label{E:rhopnrelation}
            p= n \left. \frac{\partial \rho}{\partial n} \right|_{\Ent} -
            \rho,                                                                                   
        \end{align}
        where the notation $\left. \right|_{\cdot}$ indicates partial differentiation with $\cdot$ held
        constant.
        \item A perfect fluid satisfies
        \begin{align}
        \left. \frac{\partial \rho}{\partial n} \right|_{\Ent} >0, \left. \frac{\partial p}{\partial
        n} \right|_{\Ent}>0, \left. \frac{\partial \rho}{\partial {\Ent}} \right|_n \geq 0 \
        \mbox{with} \ ``=" \ \iff \ \Ent=0.                             \label{E:EOSAssumptions}
        \end{align}
        As a consequence, we have that $\zeta,$ the \emph{speed
        of sound}\footnote{In general, $\zeta$ is not constant. However, for the equations of
        state we study in this article, $\zeta$ is equal to the constant $\speed.$} in the fluid, is always real for $\Ent > 0:$
            \begin{align}
                \zeta^2 \overset{def}{=} \left.\frac{\partial p}{\partial
                \rho}\right|_{\Ent} = \frac{{\partial p / \partial n}|_{\Ent}}{{\partial \rho / \partial
                n}|_{\Ent}} > 0.                                                                         \label{E:SpeedofSound}
            \end{align}
        \item We also demand that the speed of sound is positive and less than or equal to the speed of light
        whenever $n > 0$ and $\Ent > 0$:
            \begin{align} \label{E:Causality}
                n>0 \ \mbox{and} \ \Ent > 0 \implies 0 \leq \zeta \leq 1.
            \end{align}
    \end{enumerate}
    
   	Postulates (1) - (3) are manifestations of the laws of thermodynamics and fundamental thermodynamic assumptions, while 
   	Postulate (4) is connected to the notion of causality. More specifically, it ensures that at each $x \in \mathcal{M},$
   	vectors that are causal with respect to the \emph{sound cone}\footnote{The
    sound cone is defined to be the subset of tangent vectors $X \in T_x \mathcal{M}$ such that $m_{\alpha \beta} X^{\alpha} 
    X^{\beta} = 0,$ where $m_{\mu \nu} \eqdef g_{\mu \nu} + (1 - \zeta^2)u_{\mu} u_{\nu}$ is the \emph{acoustical metric}; see
    e.g. \cite{dC2007}, \cite{iRjS2009}, \cite{jS2008a} for information concerning the role of the acoustical metric
    in the analysis of the Euler equations.} in $T_x \mathcal{M}$ are necessarily causal with respect to the gravitational null 
    cone\footnote{By gravitational null cone at $x,$ we mean the subset of tangent vectors $X \in T_x \mathcal{M}$ such that 
    $g_{\alpha \beta} X^{\alpha} X^{\beta} = 0.$} in $T_x \mathcal{M}.$ The physical interpretation of Postulate (4) is that 
    the speed of sound is no greater than the speed of propagation of gravitational waves, i.e., that sound waves are causal as
    measured by $g_{\mu \nu}.$ See \cite{jS2008a} for a more detailed analysis of the geometry of the sound cone and the 
    gravitational null cone.
    
    The assumptions $\rho \geq 0,$ $p \geq 0$ together imply that the energy-momentum tensor \eqref{E:Tlowerfluid} 
    satisfies both the \emph{weak energy condition} ($T_{\alpha \beta} X^{\alpha} X^{\beta} \geq 0$ holds whenever $X$ is 
    timelike and future-directed with respect to the gravitational null cone) and the \emph{strong energy condition}
    ($[T_{\mu \nu} - 1/2 g^{\alpha \beta}T_{\alpha \beta} g_{\mu \nu}]X^{\mu}X^{\nu} \geq 0$ holds whenever 
    $X$ is timelike and future-directed with respect to the gravitational null cone). Furthermore, if we assume that the 
    equation of state is such that $p=0$ when $\rho = 0,$ then \eqref{E:SpeedofSound} and \eqref{E:Causality} together imply 
    that $p \leq \rho.$ This latter inequality implies that the \emph{dominant energy condition} holds
    ($-g^{\mu \alpha}T_{\alpha \nu} X^{\nu}$ is causal and future-directed whenever $X$ is causal and future-directed 
    with respect to the gravitational null cone). We remark that many important theorems in general relativity 
    use the assumption that the matter model satisfies the dominant energy condition. As examples, we cite 
    the Hawking-Penrose singularity theorem \cite{rP1965}, \cite{sH1967} 
    and the positive mass theorem, which was first proved by Schoen-Yau 
    \cite{rSstY1979}, \cite{rSstY1981}, and later by Witten \cite{eW1981}.
		
Under the remaining postulates, Postulate $(1)$ is equivalent to making a choice of an \emph{equation of state}, which is a function that expresses $p$ in terms of $\Ent$ and $\rho.$ An equation of state is not necessarily a fundamental law of nature, but can instead be an empirical relationship between the fluid variables. In this article, we consider only the case of a 
\emph{ barotropic}\footnote{A \emph{barotropic} fluid is one for which $p$ is a function of $\rho$ alone.} fluid under the equation of state $p = \speed^2 \rho,$ where $0 < \speed < \sqrt{\frac{1}{3}},$ and according to \eqref{E:SpeedofSound}, the constant $\speed$ is the \emph{speed of sound}. In this case, $\Ent$ plays no role in our analysis of the Euler equations,
and this quantity is absent from the remainder of the article.

\subsection{The initial value problem} \label{SS:IVP}

In this section, we discuss various aspects of the initial value problem for the Einstein equations, including the initial data and the notion of the \emph{maximal globally hyperbolic development} of the data. We assume that we are given an equation of state $p=p(\rho)$ subject to the restrictions discussed in Section \ref{S:EE}. We remark that the discussion in this section is very standard, and finds its in origin in the seminal works \cite{CB1952} by Choquet-Bruhat and \cite{cBgR1969} by Choquet-Bruhat/Geroch. The discussion in this section concerns the initial value problem formulation ``in the abstract,'' without reference to a choice of gauge for the Einstein equations. In Section \ref{S:ReducedEquations}, we will provide a concrete formulation of the initial value problem for the Euler-Einstein system that is suitable for our purposes at hand. Our formulation is based on a modification of the wave coordinate gauge used in  \cite{hR2008}, which is itself a modification of the gauge used by Choquet-Bruhat in \cite{CB1952}.

\subsubsection{Summary of the Euler-Einstein system} \label{SSS:Summary}

We first summarize the results of the previous sections by stating that the Euler-Einstein system is the following system of equations:

\begin{subequations}
\begin{align} 
	\mbox{Ric}_{\mu \nu} - \frac{1}{2}R g_{\mu \nu} + \Lambda g_{\mu \nu} & = T_{\mu \nu},&& (\mu, \nu = 0,1,2,3), 
		\label{E:EinsteinFieldsummary} \\
	D_{\alpha} T^{\alpha \mu} & = 0,&& (\mu = 0,1,2,3), \label{E:DivergenceofTis0Summary} \\
	D_{\alpha}(n u^{\alpha}) & = 0,&& \label{E:ConservationofParticleSummary}
\end{align}
\end{subequations}
where $T_{\mu \nu} = (\rho + p) u_{\mu} u_{\nu} + p g_{\mu \nu},$ and $p = p(\rho)$ is given by an equation of state subject to the hypotheses discussed in Section \ref{SS:EEIntro}.

We remark that by taking the trace of each side of \eqref{E:EinsteinFieldsummary} implies that
$R - 4 \Lambda = T,$ which implies that \eqref{E:EinsteinFieldsummary} is equivalent to

\begin{align} \tag{\ref{E:EinsteinFieldsummary}'}
	\mbox{Ric}_{\mu \nu} - \Lambda g_{\mu \nu} - T_{\mu \nu} + \frac{1}{2}T g_{\mu \nu} = 0.
\end{align}
Since $T \eqdef g^{\alpha \beta} T_{\alpha \beta} = 3p - \rho,$ it follows that 

\begin{align} 
	T_{\mu \nu} - \frac{1}{2}T g_{\mu \nu} =  (\rho + p) u_{\mu} u_{\nu} + \frac{1}{2}(\rho - p) g_{\mu \nu}.
\end{align}
Under the equation of state of interest to us, namely $p = \speed^2 \rho,$ we have that

\begin{align}
	T_{\mu \nu} - \frac{1}{2}T g_{\mu \nu} = \frac{1 + \speed^2}{\speed^2}p u_{\mu} u_{\nu}
		+ \frac{1 - \speed^2}{2\speed^2}pg_{\mu \nu}.
\end{align}

\subsubsection{Summary and alternative formulation of the Euler-Einstein system under the equation of state $p=\speed^2 \rho$}
Using the results of Section \ref{SSS:Summary}, under the equation of state $p = \speed^2 \rho,$ the Euler-Einstein system comprises the equations

\begin{align} \label{E:EulerEinstein}
	\mbox{Ric}_{\mu \nu} - \Lambda g_{\mu \nu} - T_{\mu \nu} + \frac{1}{2}T g_{\mu \nu} = 0,&& (\mu, \nu = 0,1,2,3),
\end{align}
where $T_{\mu \nu} = \frac{(1 + \speed^2)p}{\speed^2}u_{\mu} u_{\nu} + p g_{\mu \nu},$
together with the equations of motion for a perfect fluid: 

\begin{subequations}
\begin{align} 
	D_{\alpha} T^{\alpha \mu} & = 0,&&  (\mu = 0,1,2,3), \label{E:fluid1} \\
	D_{\alpha} (n u^{\alpha}) & = 0, \label{E:fluid2} 
\end{align}
\end{subequations}
and the normalization condition $g_{\alpha \beta} u^{\alpha} u^{\beta} = -1.$

By projecting in the direction of $u$ and onto the $g-$orthogonal complement of $u,$ it can be checked that equations \eqref{E:fluid1} are equivalent to the following system:

\begin{subequations}
\begin{align}
	u^{\alpha} D_{\alpha} p + (1+ \speed^2)pD_{\alpha} u^{\alpha} & = 0, &&  \label{E:fluidu} \\
	(1+ \speed^2)p u^{\alpha} D_{\alpha} u^{j} + \speed^2 \Pi^{j \alpha} D_{\alpha} p & = 0,&& (j=1,2,3), \label{E:fluiduperp} 
\end{align}
\end{subequations}
where 

\begin{subequations}
\begin{align}
	g_{\alpha \beta} u^{\alpha} u^{\beta} & = -1, \label{E:fourvelocitynormalized} \\
	\Pi^{\mu \nu} & \eqdef u^{\mu} u^{\nu} + g^{\mu \nu}. \label{E:Pi}
\end{align}
\end{subequations}
In the above expression, $\Pi^{\mu \nu}$ is the projection onto the $g-$orthogonal complement of $u^{\mu},$
which is sometimes referred to as the \emph{simultaneous space} of $u^{\mu}.$ Furthermore, for a \emph{barotropic} equation of state, that is, one of the form $p = p(\rho),$ the fundamental thermodynamic law \eqref{E:rhopnrelation}, which reads $p= n \frac{d \rho}{d n} - \rho$ in this case, can be used to show that equation \eqref{E:fluid2} is an automatic consequence of \eqref{E:fluid1}. Therefore, \eqref{E:fluidu} - \eqref{E:Pi} is an equivalent formulation of the relativistic Euler equations under the equation of state $p=\speed^2 \rho;$ we will work with this formulation for the remainder of the article.

\subsubsection{Initial data for the Euler-Einstein system} \label{SSS:InitialDataOriginalSystem}

Initial data for the system \eqref{E:EinsteinFieldsummary} - \eqref{E:ConservationofParticleSummary} consist of a $3-$dimensional manifold $\mathring{\Sigma}$ together with the following fields on $\mathring{\Sigma}:$ a Riemannian metric $\mathring{\underline{g}}_{jk},$ a covariant two-tensor $\mathring{\underline{K}}_{jk},$ a function $\mathring{p},$ and a vectorfield $\underline{\mathring{u}}^j,$ $(j,k = 1,2,3).$

It is well-known that one cannot consider arbitrary data for the Einstein equations. The data are in fact subject to the following constraints, where $\mathring{\underline{D}}$ is the Levi Civita connection corresponding to $\mathring{\underline{g}}:$

\begin{subequations}
\begin{align}
	\mathring{\underline{R}} - \mathring{\underline{K}}_{ab} \mathring{\underline{K}}^{ab} + (\mathring{\underline{g}}^{ab} \mathring{\underline{K}}_{ab})^2 & = 2T_{00}|_{\mathring{\Sigma}}, \label{E:Gauss} \\
	\mathring{\underline{D}}^a \mathring{\underline{K}}_{aj} - \mathring{\underline{g}}^{ab}  \mathring{\underline{D}}_j \mathring{\underline{K}}_{ab}  & = 
		T_{0j}|_{\mathring{\Sigma}}. \label{E:Codazzi}
\end{align}
\end{subequations}
In the above expressions, indices are raised and lowered using $\mathring{\underline{g}}^{-1}$ and $\mathring{\underline{g}}.$
We remark that in the case of the equation of state $p = \speed^2 \rho,$ and under the assumptions discussed in the next paragraph, we have that $T_{00} |_{\mathring{\Sigma}} = \frac{1 + \speed^2}{\speed^2}\mathring{p} (\mathring{u}_0)^2 - \mathring{p},$ and $T_{0j}|_{\mathring{\Sigma}}= \frac{1 + \speed^2}{\speed^2}\mathring{p} \mathring{u}_0 \underline{\mathring{u}}_j,$ where $\mathring{u}_0 = - \sqrt{1 + \mathring{\underline{g}}_{ab}\underline{\mathring{u}}^a \underline{\mathring{u}}^b}.$
 
The constraints \eqref{E:Gauss} - \eqref{E:Codazzi} are known as the \emph{Gauss} and \emph{Codazzi} equations respectively. These equations relate the geometry of the ambient Lorentzian spacetime $(\mathcal{M},g)$ (which has to be constructed) to the geometry inherited by an embedded Riemannian hypersurface (which will be $(\mathring{\Sigma},\mathring{\underline{g}})$ after construction). Without providing the rather standard details (see e.g. \cite{dC2008}, \cite{rW1984}), we remark that equations \eqref{E:Gauss} - \eqref{E:Codazzi} can be derived as consequences of the following assumptions:

\begin{itemize}
	\item $\mathring{\Sigma}$ is a spacelike submanifold of the spacetime manifold $\mathcal{M}$
	\item $\mathring{\underline{g}}$ is the first fundamental form of $\mathring{\Sigma}$
	\item $\mathring{\underline{K}}$ is the second fundamental form of $\mathring{\Sigma}$
	\item The Euler-Einstein equations are satisfied along $\mathring{\Sigma}$
	\item We are using a coordinate system $(x^0 =t, x^1, x^2, x^3)$ on $\mathcal{M}$ such that $\mathring{\Sigma} = 
		\lbrace x \in \mathcal{M} \ | \ t = 0 \rbrace,$ 
	and along $\mathring{\Sigma},$ $g_{00} = -1,$ $g_{0j} = 0,$ $g_{jk} = \mathring{\underline{g}}_{jk},$ $\partial_t g_{jk} = 
	2\mathring{\underline{K}}_{jk},$ $p = \mathring{p},$ and $u^{j} = \underline{\mathring{u}}^j,$ with 
	$g_{\alpha \beta} u^{\alpha} u^{\beta} = -1.$ The next-to-last condition is a consequence of the 
	assumption that $(\underline{\mathring{u}}^1,\underline{\mathring{u}}^2,\underline{\mathring{u}}^3)$ is the 
	$g-$orthogonal projection of the four-velocity $(u^0,u^1,u^2,u^3)$ onto $\mathring{\Sigma}.$
\end{itemize}

We recall that under the above assumptions, $\mathring{\underline{g}}$ and $\mathring{\underline{K}}$ are defined by

\begin{subequations}
\begin{align}
	\mathring{\underline{K}}|_x(X,Y) & = g|_x(D_{\hat{N}}X,Y),&& \forall X,Y \in T_x \mathring{\Sigma}, \\
	g|_x(X,Y) & = \mathring{\underline{g}}|_x(X,Y)&& \forall X,Y \in T_x \mathring{\Sigma},
\end{align}
\end{subequations}
where $\hat{N}$ is the future-directed normal\footnote{Under the above assumptions, it follows that at every point $x \in \mathring{\Sigma},$ $\hat{N}^{\mu} = \delta_0^{\mu}.$} to $\mathring{\Sigma}$ at $x,$ and $D$ is the Levi-Civita connection corresponding to $g.$

\subsubsection{The definition of a solution to the Euler-Einstein system}
In this section, we provide the definition of a solution to the Euler-Einstein system launched by a given initial data set. The
discussion in this section and the next one follows the approach of \cite{hR2008}, which was replicated in \cite{iRjS2009}. We begin with the following definition, which describes the maximal region in which a solution is determined by its values on a set $S.$

\begin{definition}
Given any set $S \subset \mathcal{M},$ we define $\mathcal{D}(S),$ the Cauchy development of $S,$ to be the union
$\mathcal{D}(S) = \mathcal{D}^+(S) \cup \mathcal{D}^-(S),$ where $\mathcal{D}^+(S)$ is the set of all points $x \in \mathcal{M}$ such that every past inextendible causal curve through $x$ intersects $S,$ and $\mathcal{D}^-(S)$ 
is the set of all points $p \in \mathcal{M}$ such that every future-inextendible\footnote{A curve $\gamma :[s_0, s_{max}) \rightarrow \mathcal{M}$ is said to be \emph{future-inextendible} if there does not exist an immersed future-directed curve 
$\widetilde{\gamma} : I \rightarrow \mathcal{M}$ with $[s_0, s_{max}) \subset I,$ $[s_0, s_{max}) \neq I,$ and
$\widetilde{\gamma}|_{[s_0, s_{max})} = \gamma.$ Past inextendibility is defined in an analogous manner.} causal curve through $p$ intersects $S.$
\end{definition}

We also define a \emph{Cauchy hypersurface}.

\begin{definition} \label{D:Cauchy}
	A Cauchy hypersurface in a Lorentzian manifold $\mathcal{M}$ is a hypersurface $\Sigma$ 
	that is intersected exactly once by every inextendible timelike curve in $\mathcal{M}.$ 
\end{definition}
It is well-known that if $\Sigma \subset \mathcal{M}$ is a Cauchy hypersurface, then $\mathcal{D}(\Sigma) = \mathcal{M}$ (see e.g. \cite{bO1983}).

We now provide the definition a solution launched by an initial data set.

\begin{definition} \label{D:Solution}
Given sufficiently smooth initial data $(\mathring{\Sigma}, \mathring{\underline{g}}_{jk}, \mathring{\underline{K}}_{jk}, \mathring{p}, \underline{\mathring{u}}^j),$ $(j,k = 1,2,3)$ as described in Section \ref{SSS:InitialDataOriginalSystem}, a (classical) solution to the system \eqref{E:EinsteinFieldsummary} - \eqref{E:ConservationofParticleSummary} is a $4-$dimensional manifold $\mathcal{M},$ a Lorentzian metric $g_{\mu \nu},$ a function $p,$ a vectorfield $u^{\mu},$ 
$(\mu, \nu = 0,1,2,3),$ and an embedding $\mathring{\Sigma} \hookrightarrow \mathcal{M}$ subject to the following conditions:

	\begin{itemize}
		\item $g$ is a $C^2$ tensorfield, $p$ is a $C^1$ function, and $u$ is a $C^1$ tensorfield 
		\item Equations \eqref{E:EinsteinFieldsummary} - \eqref{E:ConservationofParticleSummary}
			are satisfied by the components of $g,$ $p,$ and $u$
		\item $\mathring{\Sigma}$ is a spacelike Cauchy hypersurface in $(\mathcal{M},g)$
		\item $\mathring{\underline{g}}$ is the first fundamental form of $\mathring{\Sigma}$
		\item $\mathring{\underline{K}}$ is the second fundamental form of $\mathring{\Sigma}$ 
		\item $\pi(u) = \underline{\mathring{u}}$ where $\pi$ denotes $g-$orthogonal projection 
			onto $\mathring{\Sigma}$
		\end{itemize}
	The array $(\mathcal{M}, g_{\mu \nu}, p, u^{\mu})$ is called a \emph{globally hyperbolic development} of the initial data.
\end{definition}

\subsubsection{The maximal globally hyperbolic development}

In this section, we state a fundamental abstract existence result of Choquet-Bruhat and Geroch \cite{cBgR1969}, which states that for initial data of sufficient regularity, there is a unique ``largest'' spacetime \emph{determined uniquely} by it. The following definition captures the notion of this ``largest'' spacetime. 

\begin{definition}
Given sufficiently smooth initial data for the Euler-Einstein system \eqref{E:EinsteinFieldsummary} - \eqref{E:ConservationofParticleSummary} that satisfy the constraints \eqref{E:Gauss} - \eqref{E:Codazzi}, a \emph{maximal globally hyperbolic development} of the data is a globally hyperbolic development $(\mathcal{M},g,p,u)$ together with an embedding $\iota: \mathring{\Sigma} \hookrightarrow \mathcal{M}$ with the following property: if $(\mathcal{M}',g',p', u')$ is any other globally hyperbolic development of the same data with embedding $\iota': \mathring{\Sigma} \rightarrow \mathcal{M}',$ then there is a map $\psi: \mathcal{M}' \rightarrow \mathcal{M}$ that is a diffeomorphism onto its image such that $\psi^* g = g', \psi^* p = p', \psi^* u = u'$ and $\psi \circ \iota' = \iota.$ Here, $\psi^*$ denotes the pullback by $\psi.$
\end{definition}

Before we can state the theorem, we also need the following definition, which captures the notion of having two different representations of the same spacetime.

\begin{definition}
The developments $(\mathcal{M},g,p,u)$ and $(\mathcal{M}',g',p',u')$ are said to be \emph{isometrically isomorphic} if the map $\psi$ from the previous definition is a diffeomorphism from $\mathcal{M}$ to $\mathcal{M}'.$
\end{definition}

We now state the aforementioned theorem. 

\begin{theorem}[\textbf{Existence of an MGHD}] \cite{cBgR1969} \label{T:MGHD}
Given sufficiently smooth\footnote{The article \cite{cBgR1969} only discusses the case of smooth data. However, as discussed in \cite[Section 6]{pCgGdP2010}, the regularity assumptions of Theorem \ref{T:LocalExistence} are sufficient for the conclusions of
Theorem \ref{T:MGHD} to be valid.} initial data for the Euler-Einstein system \eqref{E:EinsteinFieldsummary} - \eqref{E:ConservationofParticleSummary}, there exists a maximal globally hyperbolic development of the data which is unique up to isometric isomorphism. 
\end{theorem}

The remainder of this article concerns the ``future'' properties of the maximal globally hyperbolic developments of sufficiently smooth data near those corresponding to the FLRW background solutions introduced in Section \ref{S:backgroundsolution}.

\section{FLRW Background Solutions} \label{S:backgroundsolution}

Our main results address the future-stability of a well-known class of FLRW\footnote{Technically, the term ``FLRW'' is usually reserved for a class of solutions that have spatial slices diffeomorphic to $\mathbb{S}^3,$ $\mathbb{R}^3,$ or hyperbolic space (see \cite{rW1984}).} background solutions $([0,\infty) \times \mathbb{T}^3,\widetilde{g}_{\mu \nu}, \widetilde{p}, \widetilde{u}^{\mu}),$ $(\mu, \nu = 0,1,2,3),$ to the system \eqref{E:EulerEinstein} $+$ \eqref{E:fluidu} - \eqref{E:fluiduperp}. These background solutions physically represent the evolution of an initially \emph{uniform} \emph{quiet} fluid in a spacetime that is undergoing exponentially accelerated expansion. To find these solutions, we follow a standard
procedure that is outlined e.g. in \cite[Chapter 5]{rW1984} which, under appropriate ansatzes, reduces the Euler-Einstein equations to ODEs. The discussion in this section was essentially provided in \cite[Section 4]{iRjS2009}, but we repeat it here
for convenience. To proceed, we first make the ansatz that the background metric $\widetilde{g} = \widetilde{g}(t)$ is of the form

\begin{align} \label{E:backgroundmetricform}
	\widetilde{g} = -dt^2 + a^2(t) \sum_{i=1}^3 (dx^i)^2,
\end{align}
from which it follows that the only corresponding non-zero Christoffel symbols are

\begin{align} \label{E:BackgroundChristoffel}
	\widetilde{\Gamma}_{j \ k}^{\ 0} = \widetilde{\Gamma}_{k \ j}^{\ 0}  = a \dot{a} \delta_{jk}, 
		&& \widetilde{\Gamma}_{j \ 0}^{\ k} = \widetilde{\Gamma}_{0 \ j}^{\ k}  = \omega \delta_j^k, && (j,k=1,2,3), 
\end{align}
where 

\begin{align}
	\omega \eqdef \frac{\dot{a}}{a} \label{E:omegadef},
\end{align}
and $\dot{a} \eqdef \frac{d}{dt} a.$ Using definitions \eqref{E:Riccidef} and \eqref{E:Rdef}, together with \eqref{E:BackgroundChristoffel}, we compute that

\begin{subequations}
\begin{align} \label{E:BackgroundEinsteintensor}
	\widetilde{\mbox{Ric}}_{00} - \frac{1}{2} \widetilde{R} \widetilde{g}_{00} & = 3 \Big(\frac{\dot{a}}{a}\Big)^2, && \\
	\widetilde{\mbox{Ric}}_{0j} - \frac{1}{2} \widetilde{R} \widetilde{g}_{0j} & = 0,  && (j=1,2,3),\\
	\widetilde{\mbox{Ric}}_{jk} - \frac{1}{2} \widetilde{R} \widetilde{g}_{jk} & = - (2a \ddot{a} + \dot{a}^2)\delta_{jk},
		&& (j,k=1,2,3). 
\end{align}
\end{subequations}
We then assume that $\widetilde{\rho} = \widetilde{\rho}(t),$ $\widetilde{p} = \widetilde{p}(t),$ and 
$\widetilde{u}^{\mu} \equiv (1,0,0,0).$ We also assume that the equation of state $p = \speed^2 \rho$ holds, and for simplicity, we assume for the remainder of the article that

\begin{align} \label{E:ainitial}
	a(0) = 1.
\end{align}
Inserting these ansatzes into the Einstein equations \eqref{E:EulerEinstein}, we deduce (as in \cite{rW1984}) the following equations, which are known as the \emph{Friedmann equations} in the cosmology literature:

\begin{subequations}
\begin{align}
	\widetilde{\rho} a^{3(1 + \speed^2)} & \equiv \bar{\rho},
		\label{E:backgroundrhoafact} \\
	\dot{a} = a \sqrt{\frac{\Lambda}{3} + \frac{\widetilde{\rho}}{3}} & = a \sqrt{\frac{\Lambda}{3} 
		+ \frac{\bar{\rho}}{3a^{3(1 + \speed^2)}}}, \label{E:backgroundaequation}
\end{align} 
\end{subequations}
where the \emph{positive} constant $\bar{\rho}$ denotes the initial (uniform) energy density. We also denote the initial
pressure by $\bar{p} \eqdef \speed^2 \bar{\rho}.$ Observe that the rapid expansion of the background spacetime can be easily deduced from the ODE \eqref{E:backgroundaequation}, which suggests that the asymptotic behavior $a(t) \sim e^{Ht}, H \eqdef \sqrt{\Lambda/3}.$ A more detailed analysis of $a(t)$ is given in Lemma \ref{L:backgroundaoftestimate}.

For aesthetic reasons, we also introduce the quantities

\begin{align}
	\Omega(t) & \eqdef \ln a(t), \label{E:BigOmega} \\
	\decayparameter & \eqdef 3 \speed^2,
\end{align}
which implies that

\begin{align}
	a(t) & = e^{\Omega(t)}, \\
	\omega(t) & = \frac{d}{dt} \Omega(t).
\end{align}
For future use, we note the following simple consequences of the above discussion; we leave it to the 
reader to supply the details:

\begin{subequations}
\begin{align}
	3 \omega^2 - \Lambda & = \frac{1}{\speed^2} \widetilde{p}, \label{E:3omegasquaredminusLambdaidentity} \\
	\dot{\omega} & = - \frac{(1 + \speed^2)}{2\speed^2} \widetilde{p}, \label{E:omegadotidentity} \\
	3 \dot{\omega} + 3 \omega^2 - \Lambda & = -\frac{1 + 3\speed^2}{2\speed^2} \widetilde{p}.
\end{align}
\end{subequations}

\subsection{Analysis of Friedmann's equation}

The following lemma summarizes the asymptotic behavior of solutions to the ODE \eqref{E:backgroundaequation}.

\begin{lemma}\cite[Lemma 4.2.1]{iRjS2009} \label{L:backgroundaoftestimate}
	Let $\bar{\rho}, \varsigma > 0$ be constants, and let $a(t)$ be the solution to the following ODE: 
	
	\begin{align}
		\frac{d}{dt}a & = a \sqrt{\frac{\Lambda}{3} + \frac{\bar{\rho}}{3a^{\varsigma}}}, && a(0) = 1.
	\end{align}
	Then with $H \eqdef \sqrt{\Lambda/3},$ the solution $a(t)$ is given by
	
	\begin{align}
		a(t) & = \bigg\lbrace\mbox{sinh}\Big(\frac{\varsigma H t}{2}\Big) 
			\sqrt{\frac{\bar{\rho}}{3H^2} + 1}
		 	+	\mbox{cosh}\Big(\frac{\varsigma H t}{2} \Big) \bigg\rbrace^{2/\varsigma}, 
	\end{align}
	and for all integers $N \geq 0,$ there exists a constant $C_N > 0$ such that for all $t \geq 0,$ with 
	$A \eqdef \bigg\lbrace \frac{1}{2} \Big(\sqrt{\frac{\bar{\rho}}{3H^2} + 1} + 1 \Big) \bigg\rbrace^{2/\varsigma},$
	we have that
	
	\begin{subequations}
	\begin{align}
		(1/2)^{2/\varsigma} e^{Ht} \leq a(t) & \leq A e^{Ht}, \\
		\Big| e^{-Ht} \frac{d^N}{dt^N}a(t) - A H^N \Big| & \leq C_Ne^{-\varsigma Ht}.
	\end{align}
	\end{subequations}
	
	Furthermore, for all integers $N \geq 0,$ there exists a constant $\widetilde{C}_N > 0$ such that
	for all $t \geq 0,$ with 
	\begin{align} 
		\omega \eqdef \frac{\dot{a}}{a}, 
	\end{align}	
	 we have that
	
	\begin{subequations}
	\begin{align}
		H \leq \omega(t) & \leq \sqrt{H^2 + \frac{\bar{\rho}}{3}}, \\
		\Big| \frac{d^N}{dt^N}\big(\omega(t) - H\big) \Big| & \leq \widetilde{C}_N e^{-\varsigma Ht}.
	\end{align}
	\end{subequations}
\end{lemma}

\begin{remark}
	Because of equation \eqref{E:backgroundaequation}, we will assume for the remainder of the article that
	$\varsigma = 3(1 + \speed^2).$
\end{remark}

\section{The Modified Euler-Einstein System} \label{S:ReducedEquations}
In this section, we recall the wave coordinate system introduced in \cite{iRjS2009}, which was based on the framework used in \cite{hR2008}. We then use algebraic identities that are valid in wave coordinates to construct a modified version of the Euler-Einstein equations, which is a system of quasilinear wave equations coupled to the first-order Euler equations, and which contains \emph{energy-dissipative} terms. We then construct data for the modified system from given data for the unmodified system in a manner consistent with our wave coordinate system. Next, to facilitate our analysis in later sections, we algebraically decompose the modified system into principal terms and error terms. We then discuss local existence and a continuation principle for the modified system, and we sketch a standard proof of the fact that the modified system is equivalent to the unmodified system if the Einstein constraint equations and the wave coordinate condition are both satisfied along the initial Cauchy hypersurface $\mathring{\Sigma}.$ Finally, for convenience, we introduce some PDE matrix-vector notation for the Euler equations.

\subsection{Wave coordinates} \label{SS:harmoniccoodinates}
To hyperbolize the Einstein equations, we use a coordinate system in which the contracted Christoffel symbols $\Gamma^{\mu} \eqdef g^{\alpha \beta} \Gamma_{\alpha \ \beta}^{\ \mu}$ of the spacetime metric $g$ are equal to the contracted Christoffel symbols $\widetilde{\Gamma}^{\mu} \eqdef \widetilde{g}^{\alpha \beta} \widetilde{\Gamma}_{\alpha \ \beta}^{\ \mu}$ of the FLRW background metric $\widetilde{g}.$ This well-known condition is known as a \emph{wave coordinate} condition since $\Gamma^{\mu} \equiv \widetilde{\Gamma}^{\mu}$ if and only if the coordinate functions\footnote{The $x^{\mu}$ are scalar-valued functions, despite the fact that they have indices.} $x^{\mu}$ are solutions to the wave equation $g^{\alpha \beta} D_{\alpha} D_{\beta} x^{\mu} + \widetilde{\Gamma}^{\mu} = 0.$ Using \eqref{E:backgroundmetricform} and \eqref{E:BackgroundChristoffel}, we compute that in wave coordinates, we have

\begin{align} \label{E:HarmonicGauge}
	\Gamma^{\mu} & = \widetilde{\Gamma}^{\mu} = 3 \omega \delta_{0}^{\mu},&& \Gamma_{\mu} = 
		g_{\mu \alpha} \Gamma^{\alpha} = 3 \omega g_{0 \mu},
\end{align}
where $\omega(t)$ is defined in \eqref{E:omegadef}.

We now introduce the tensorfields\footnote{Technically, $P^{\mu}$ and $P_{\mu}$ 
do not have the coordinate transformation properties of a vectorfield/covectorfield. Nonetheless, we will treat them as 
vectorfields/covectorfields when we compute their covariant derivatives. On the other hand, $Q^{\mu}$ and $Q_{\mu}$ do have the transformation properties of a vectorfield/covectorfield.}

\begin{subequations}
\begin{align}
	P^{\mu} & \eqdef \widetilde{\Gamma}^{\mu} = 3 \omega \delta_0^{\mu}, && P_{\mu} = \widetilde{\Gamma}_{\mu} = 3 \omega g_{0 
		\mu}, \\
	Q^{\mu} & \eqdef P^{\mu} - \Gamma^{\mu}, && Q_{\mu} = P_{\mu} - \Gamma_{\mu}. \label{E:Qdef}
\end{align}
\end{subequations}

The idea behind wave coordinates is that when $Q^{\mu} \equiv 0,$ whenever it is expedient, we may replace $\Gamma^{\mu}$ with $3 \omega \delta_{0}^{\mu}$ (and vice-versa) without altering the content of the Einstein equations. The existence of such a coordinate system is nontrivial, and it was only in 1952 that Choquet-Bruhat \cite{CB1952} first showed that they exist in general (see Proposition \ref{P:Preservationofgauge}). With this idea in mind, we define (as in \cite[Equation $(47)$]{hR2008}) 
the \emph{modified Ricci tensor} $\widehat{\mbox{Ric}}_{\mu \nu}$ by

\begin{align} \label{E:modifiedRicci}
	\widehat{\mbox{Ric}}_{\mu \nu} & \eqdef \mbox{Ric}_{\mu \nu} + \frac{1}{2} \big(D_{\mu} Q_{\nu} + D_{\nu} Q_{\mu} \big)  \\
	& = -\frac{1}{2} \hat{\square}_g g_{\mu \nu} + \frac{1}{2} \big(D_{\mu} P_{\nu} + D_{\nu}P_{\mu} \big)
		+ g^{\alpha \beta} g^{\gamma \delta} (\Gamma_{\alpha \gamma \mu} \Gamma_{\beta \delta \nu}
		+ \Gamma_{\alpha \gamma \mu} \Gamma_{\beta \nu \delta}
		+ \Gamma_{\alpha \gamma \nu} \Gamma_{\beta \mu \delta}), \notag
\end{align}
where

\begin{align} \label{E:reducedwaveoperator}
	\hat{\square}_g \eqdef g^{\alpha \beta} \partial_{\alpha} \partial_{\beta}
\end{align}
is the \emph{reduced wave operator} corresponding to the metric $g.$

We now replace the $\mbox{Ric}_{\mu \nu}$ with $\widehat{\mbox{Ric}}_{\mu \nu}$ in \eqref{E:EulerEinstein}, expand the covariant differentiation in \eqref{E:fluidu} - \eqref{E:fluiduperp}, and add additional inhomogeneous terms $I_{\mu \nu}$ to the left-hand side of \eqref{E:EulerEinstein}, thereby arriving at the \emph{modified Euler-Einstein system}:

\begin{subequations}
\begin{align}
	\widehat{\mbox{Ric}}_{\mu \nu} - \Lambda g_{\mu \nu} - T_{\mu \nu} + \frac{1}{2} T g_{\mu \nu} + I_{\mu \nu} & = 0, 
		&& (\mu, \nu = 0,1,2,3), \label{E:FirstmodifiedRicci} \\
	u^{\alpha} D_{\alpha} p + (1+ \speed^2)pD_{\alpha} u^{\alpha} & = 0,&& \label{E:firstEulermodified} \\
	u^{\alpha} D_{\alpha} u^{j} + \frac{\speed^2}{(1 + \speed^2)p}\Pi^{j \alpha} D_{\alpha} p & = 0,
		&& (j = 1,2,3). \label{E:Eulerjmodified} 
\end{align}
\end{subequations}
Here, the additional terms are defined to be

\begin{subequations}
\begin{align} 
	I_{00} & \eqdef -2 \omega Q^0 = 2 \omega(\Gamma^0 - 3 \omega), \label{E:gaugetermI00} \\
	I_{0j} = I_{j0} & \eqdef 2 \omega Q_j = 2 \omega(3 \omega g_{0j} - \Gamma_j), \label{E:gaugetermI0j} \\
	I_{jk} = I_{jk} & \eqdef 0. \label{E:gaugetermIjk} 
\end{align}
\end{subequations}

We have several important remarks to make concerning the modified system \eqref{E:FirstmodifiedRicci} - \eqref{E:Eulerjmodified}. First, because the principal term on the left-hand side of \eqref{E:FirstmodifiedRicci} is $-\frac{1}{2} \hat{\square}_g g_{\mu \nu},$ the modified equations \eqref{E:FirstmodifiedRicci} are quasilinear wave equations, and are therefore of hyperbolic character. Since the Euler equations \eqref{E:firstEulermodified} - \eqref{E:Eulerjmodified} are also hyperbolic in a fixed spacetime, (see e.g. \cite{dC2007}, \cite{jS2008a}), it follows that the modified Euler-Einstein system is a hyperbolic system of mixed order. Second, the gauge terms $I_{\mu \nu}$ have been added to the system in order to produce an energy dissipation effect that is analogous to the effect created by the $3(\partial_t v)^2$ term on the right-hand side of the model equation \eqref{E:modelequation}. These dissipation-inducing terms play a key role in the global existence theorem of Section \ref{S:GlobalExistence}. Finally, in Section \ref{SS:ClassicalLocalExistence}, we will elaborate upon the following fact: if the initial data satisfy the Gauss and Codazzi constraints \eqref{E:Gauss} - \eqref{E:Codazzi}, and if the wave coordinate condition $Q_{\mu}|_{t=0} = 0$ is satisfied, then $Q_{\mu},$ $I_{\mu \nu} \equiv 0,$ and $\widehat{\mbox{Ric}}_{\mu \nu} \equiv \mbox{Ric}_{\mu \nu};$ i.e., under these conditions, the solution to \eqref{E:FirstmodifiedRicci} - \eqref{E:Eulerjmodified} is also a solution to the Euler-Einstein system \eqref{E:EinsteinFieldsummary} - \eqref{E:ConservationofParticleSummary}.

\subsection{Summary of the modified system \texorpdfstring{for the equation of state $p = \speed^2 \rho$}{for the FLRW equation of state}}

For convenience, we summarize the results of the previous section by listing the modified Euler-Einstein system 
(where $j=1,2,3)$: 

\begin{subequations}
\begin{align}
	\widehat{\mbox{Ric}}_{00} + 2 \omega \Gamma^0 - 6 \omega^2 - \Lambda g_{00} - p \Big(\frac{1 + \speed^2}{\speed^2}(u_0)^2 
		+ \frac{1 - \speed^2}{2\speed^2} g_{00} \Big) & = 0, \label{E:Rhat00} \\
	\widehat{\mbox{Ric}}_{0j} - 2 \omega (\Gamma_j - 3 \omega g_{0j}) - \Lambda g_{0j} - p \Big(\frac{1 + \speed^2}{\speed^2}u_0 u_j 
		+ \frac{1 - \speed^2}{2\speed^2} g_{0j} \Big) & = 0, \label{E:Rhat0j} \\
	\widehat{\mbox{Ric}}_{jk} - \Lambda g_{jk} - p \Big(\frac{1 + \speed^2}{\speed^2}u_j u_k 
		+ \frac{1 - \speed^2}{2\speed^2} g_{jk} \Big) & = 0, \label{E:Rhatjk} \\
	u^{\alpha} D_{\alpha} p + (1+ \speed^2)p D_{\alpha} u^{\alpha} & = 0, \label{E:firstEulersummary} \\
	u^{\alpha} D_{\alpha} u^{j} + \frac{\speed^2}{(1 + \speed^2)p} \Pi^{j \alpha} D_{\alpha} p & = 0, \label{E:secondEulersummary} 
\end{align}
\end{subequations}
where $g_{\alpha \beta} u^{\alpha} u^{\beta} = -1$ and $\Pi^{\mu \nu} = u^{\mu} u^{\nu} + g^{\mu \nu}.$

\subsection{Construction of initial data for the modified system} \label{SS:IDREDUCED}

In this section, we assume that we are given initial data $(\mathring{\Sigma}, \mathring{\underline{g}}_{jk}, \mathring{\underline{K}}_{jk}, \mathring{p}, \underline{\mathring{u}}^j),$ $(j,k = 1,2,3),$ for the Euler-Einstein equations \eqref{E:EinsteinFieldsummary} - \eqref{E:ConservationofParticleSummary} as described in Section \ref{SSS:InitialDataOriginalSystem}. In particular, we assume that they satisfy the constraints \eqref{E:Gauss} - \eqref{E:Codazzi}. We will use this data to construct initial data for the modified equations that lead to a solution $(\mathcal{M}, g_{\mu \nu}, p, u^{\mu}),$ $(\mu, \nu = 0,1,2,3),$ of both the modified system and the Einstein equations; recall that a solution solves both systems $\iff Q_{\mu} \equiv 0,$ where $Q_{\mu}$ is defined in \eqref{E:Qdef}. We remark that
one may consider arbitrary data for the modified equations \eqref{E:Rhat00} - \eqref{E:secondEulersummary}, but without further assumptions, the resulting solution is not necessarily a solution to the Einstein equations \eqref{E:EinsteinFieldsummary}
- \eqref{E:ConservationofParticleSummary}. 

To supply complete data for the modified equations, we can specify along $\mathring{\Sigma} = \lbrace t=0 \rbrace$ the full spacetime metric components $g_{\mu \nu}|_{t=0},$ their future-directed normal derivatives $\partial_t g_{\mu \nu}|_{t=0},$ 
the pressure $p,$ and the $g-$orthogonal projection of the four-velocity onto $\mathring{\Sigma}.$ To satisfy the requirements

\begin{itemize}
	\item $\mathring{\Sigma} = \lbrace t=0 \rbrace$
	\item $\mathring{\underline{g}}$ is the first fundamental form of $\mathring{\Sigma}$
	\item $\mathring{\underline{K}}$ is the second fundamental form of $\mathring{\Sigma}$
	\item $\partial_t$ is future-directed and normal to $\mathring{\Sigma}$
	\item $p|_{\mathring{\Sigma}} = \mathring{p}$
	\item $\underline{\mathring{u}}$ is the $g-$orthogonal projection of
		(the unit-normalized) four-velocity $u$ onto $\mathring{\Sigma},$
\end{itemize}
we set

\begin{subequations}
\begin{align}
	g_{00}|_{t=0} & = -1, \ \ g_{0j}|_{t=0} = 0, \ \ g_{jk}|_{t=0} = \mathring{\underline{g}}_{jk}, \\
	p|_{t=0} & = \mathring{p}, \ \ u_j |_{t=0} = \underline{\mathring{u}}_j, \ \ u_0|_{t=0} = - \sqrt{1 + 
		\mathring{\underline{g}}_{ab}\underline{\mathring{u}}^a 
		\underline{\mathring{u}}^b}, \\
	(\partial_t g_{jk})|_{t=0} & = 2 \mathring{\underline{K}}_{jk}.
\end{align}
\end{subequations}

To satisfy the initial wave coordinate condition $Q_{\mu}|_{t=0} = 0,$ we first compute that

\begin{subequations}
\begin{align}
	\Gamma_0|_{t=0} & = - \frac{1}{2} (\partial_{t} g_{00})|_{t=0} - \mathring{\underline{g}}^{ab} \mathring{\underline{K}}_{ab} 
		\label{E:Gamma0attequals0}, \\
	\Gamma_j|_{t=0} & = -(\partial_t g_{0j})|_{t=0} + \frac{1}{2} \mathring{\underline{g}}^{ab}(2 \partial_a \mathring{\underline{g}}_{bj}
		- \partial_j \mathring{\underline{g}}_{ab}). \label{E:Gammajattequals0}
\end{align}
\end{subequations}
Using \eqref{E:Gamma0attequals0} and \eqref{E:Gammajattequals0}, the condition $Q_{\mu}|_{t=0} = 0$ is easily seen to be 
equivalent to the following relations, where $\omega$ is defined in \eqref{E:omegadef}:

\begin{subequations}
\begin{align}
	(\partial_{t} g_{00})|_{t=0} & = 2 (- 3 \omega|_{t=0} \underbrace{g_{00}|_{t=0}}_{-1} - \mathring{\underline{g}}^{ab} 
		\mathring{\underline{K}}_{ab}) =  2 \big(3 \omega(0)  - \mathring{\underline{g}}^{ab} 
		\mathring{\underline{K}}_{ab}\big), \\
	(\partial_t g_{0j})|_{t=0} & = - 3 \omega|_{t=0} \underbrace{g_{0 j}|_{t=0}}_{0} + \frac{1}{2} 
		\mathring{\underline{g}}^{ab}(2 \partial_a \mathring{\underline{g}}_{bj} - \partial_j \mathring{\underline{g}}_{ab}) = 
		\mathring{\underline{g}}^{ab}( \partial_a \mathring{\underline{g}}_{bj} - \frac{1}{2} \partial_j 
		\mathring{\underline{g}}_{ab}).
\end{align}
\end{subequations}
We remark that in the above expressions, $\mathring{\underline{g}}^{jk}$ denotes a component of the
inverse of $\mathring{\underline{g}}.$ This completes our specification of the data for the modified equations.

\subsection{Decomposition of the modified system in wave coordinates} \label{SS:Decomposition}
In order to reveal the dissipative structure discussed in Section \ref{SS:Commentsonanalysis}, we decompose the modified equations \eqref{E:Rhat00} - \eqref{E:secondEulersummary} into principal terms and error terms, which we denote by variations of the symbol $\triangle.$ A central component of our global existence argument is the derivation of suitable bounds for the error terms; the estimates of Section \ref{S:BootstrapConsequences} will justify the claim that the $\triangle$ terms are in fact error terms. We begin by recalling the previously mentioned rescaling $h_{jk}$ of the spatial indices of the metric:

\begin{align} \label{E:hjkdef}
	h_{jk} \eqdef e^{-2 \Omega} g_{jk}.
\end{align}
We will also make use of the following rescaling of the pressure:

\begin{align} \label{E:RESCALEDPRESSURE}
	P \eqdef e^{3(1 + \speed^2) \Omega}p.
\end{align}
The above rescaled quantities will be order $1$ in our global existence theorem, which makes them more convenient to work with; i.e., we have rescaled purely for convenience. We remark that for the FLRW background solution, 
$P \equiv \bar{p}$ and $h_{jk} \equiv \delta_{jk}.$ The decomposition is carried out in the next proposition.

\begin{proposition}[\textbf{Decomposition of the Modified Equations}] \label{P:Decomposition}
	The equations \eqref{E:Rhat00} - \eqref{E:secondEulersummary} in the unknowns $(g_{\mu \nu},P,u^j),$ $(\mu, \nu = 0,1,2,3),$
	$(j = 1,2,3),$ can be written as 

\begin{subequations}
\begin{align}
	\hat{\square}_{g} (g_{00} + 1) & = 5 H \partial_t g_{00} + 6 H^2 (g_{00} + 1) + \triangle_{00},
		\label{E:finalg00equation} \\
	\hat{\square}_{g} g_{0j} & = 3H \partial_t g_{0j} + 2 H^2 g_{0j} - 2Hg^{ab}\Gamma_{a j b} + \triangle_{0j}, 
		\label{E:finallg0jequation} \\
	\hat{\square}_{g} h_{jk} & = 3H \partial_t h_{jk} + \triangle_{jk}, \label{E:finalhjkequation} 
\end{align}

\begin{align}
	u^{\alpha} \partial_{\alpha} (P - \bar{p})+ (1 + \speed^2) \Big(\frac{-1}{u_0}\Big) P u_a \partial_t u^a + (1 + \speed^2) P 
		\partial_{a}u^{a} & = \triangle, \label{E:finalfirstEuler} \\
	u^{\alpha} \partial_{\alpha} u^{j} + \frac{\speed^2}{(1 + \speed^2)P}\Pi^{j \alpha} \partial_{\alpha} (P - \bar{p})
		& = (\decayparameter - 2) \omega u^{j} + \triangle^j, \label{E:finalEulerj}
\end{align}
\end{subequations}
where 

\begin{align} \label{E:U0UPPERISOLATED}
	u^0 & = - \frac{g_{0a}u^a}{g_{00}} + \sqrt{1 + \Big(\frac{g_{0a}u^a}{g_{00}}\Big)^2 - \frac{g_{ab}u^a u^b}{g_{00}} 
		- \Big(\frac{g_{00} + 1}{g_{00}}\Big)}, \\
	\Pi^{\mu \nu} & = u^{\mu} u^{\nu} + g^{\mu \nu},
\end{align}

	\begin{align}
		H & \eqdef \sqrt{\frac{\Lambda}{3}},&& \decayparameter \eqdef 3\speed^2,
	\end{align}
	the error terms $\triangle_{\mu \nu},$ $\triangle,$ $\triangle^j$ can be expressed as
 
\begin{subequations} 
\begin{align}
	\frac{1}{2} \triangle_{00} & = \triangle_{A,00} + \triangle_{C,00}
		- \frac{3\speed^2 + 1}{2 \speed^2} (g_{00} + 1) e^{-3(1 + \speed^2) \Omega} \bar{p}
		- \frac{3\speed^2 + 1}{2 \speed^2}e^{-3(1 + \speed^2) \Omega}(P - \bar{p}), \label{E:triangle00} \\
	& \ \ - \frac{1 + \speed^2}{\speed^2} (u_0 + 1)(u_0 - 1) e^{-3(1 + \speed^2) \Omega}P   
		- \frac{1 - \speed^2}{2 \speed^2} \big(g_{00} + 1\big)e^{-3(1 + \speed^2) \Omega}P \notag \\
	& \ \ + \frac{5}{2} (\omega - H) \partial_t g_{00} + 3 (\omega^2 - H^2)(g_{00} + 1), \notag \\
	\frac{1}{2} \triangle_{0j} & = \triangle_{A,0j} + \triangle_{C,0j}
		+ \frac{1 - 3 \speed^2}{4 \speed^2} e^{-3(1 + \speed^2)} \bar{p} g_{0j} 
		- \frac{1 + \speed^2}{\speed^2}e^{-3(1 + \speed^2) \Omega}Pu_0u_j \label{E:triangle0j} \\
	& \ \ - \frac{1 - \speed^2}{2\speed^2} e^{-3(1 + \speed^2) \Omega}P g_{0j} 
		+ \frac{3}{2} (\omega - H) \partial_t g_{0j} + (\omega^2 - H^2) g_{0j} 
		- (\omega - H) g^{ab}\Gamma_{a j b}, \notag \\
	\frac{1}{2} \triangle_{jk} & = e^{-2 \Omega} \triangle_{A,jk} + \frac{1 + \speed^2}{2\speed^2}
		\bar{p}e^{-3(1 + \speed^2) \Omega}(g^{00} + 1)h_{jk} - 2 \omega g^{0a} \partial_{a} h_{jk} \label{E:trianglejk} \\
		& \ \ - \frac{1 - \speed^2}{2\speed^2}e^{-3(1 + \speed^2) \Omega}(P - \bar{p})h_{jk} 
			- \frac{1 + \speed^2}{\speed^2} e^{-2 \Omega} e^{-3(1 + \speed^2) \Omega} P u_j u_k
			+ \frac{3}{2}(\omega - H) \partial_t h_{jk}, \notag 
	\end{align}
	
	\begin{align}
	\triangle & = - (1 + \speed^2)P \triangle_{\alpha \ 0}^{\ \alpha}u^0  
		-(1 + \speed^2) P \triangle_{\alpha \ a}^{\ \alpha}u^a \label{E:triangledef} \\
		&  \ \ + \frac{(1 + \speed^2)P}{2u_0} 
			\Big\lbrace (\partial_t g_{00})(u^0)^2 + 2 (\partial_t g_{0a})u^0 u^a + (\partial_t g_{ab})u^a u^b 	
			\Big\rbrace, \notag \\
	\triangle^j & = (\decayparameter - 2)\omega (u^0 - 1) u^j - \triangle_{\alpha \ \beta}^{\ j} u^{\alpha} u^{\beta}
		+ \decayparameter \omega g^{0j}, \label{E:trianglejdef}
\end{align}
\end{subequations}
the $\triangle_{A,\mu \nu}$ are defined in \eqref{E:triangleA00def} - \eqref{E:triangleAjkdef}, 
$\triangle_{C,00},$ $\triangle_{C,0j}$ are defined in \eqref{E:triangleC00def} - \eqref{E:triangleC0jdef},
and the $\triangle_{\mu \ \nu}^{\ \alpha}$ are defined in \eqref{E:triangleGamma000} - \eqref{E:triangleGammaikj}.

Furthermore, $u^0$ is a solution to the following equation:
	
	\begin{align}
		u^{\alpha} \partial_{\alpha} u^{0}  + \frac{\speed^2}{(1 + \speed^2)P}\Pi^{0 \alpha} \partial_{\alpha} (P - \bar{p})
			& =	\triangle^0, \label{E:u0equation} 
	\end{align}
	where
	
	\begin{align}
		\triangle^0 & = \decayparameter \omega \Big\lbrace (u^0 - 1)(u^0 + 1) + (g^{00} + 1) \Big\rbrace
			- \omega g_{ab} u^a u^b - \triangle_{\alpha \ \beta}^{\ 0} u^{\alpha} u^{\beta}. 
			\label{E:triangle0def}
	\end{align}

\end{proposition}

\begin{remark}
	Equation \eqref{E:U0UPPERISOLATED} enforces the normalization condition $g_{\alpha \beta}u^{\alpha}u^{\beta} = - 1$ and 
	the future-directed condition $u^0 > 0.$
\end{remark}

\begin{proof}
	The proof is a series of tedious computations based Lemma \ref{L:modifiedRiccidecomposition} -
	Lemma \ref{L:Christoffeldecomposition}. We provide the proofs of \eqref{E:finalhjkequation} - \eqref{E:finalEulerj} and 
	leave the remaining details to the reader. To obtain \eqref{E:finalhjkequation}, we first use equation \eqref{E:Rhatjk}, 
	Lemma \ref{L:modifiedRiccidecomposition}, and Lemma \ref{L:Amunudecomposition} to obtain the following equation for $h_{jk} = 
	e^{-2\Omega} g_{jk}:$
	
	\begin{align} \label{E:boxhjkfirstformula}
		\hat{\square}_g h_{jk} & = 3 \omega \partial_t h_{jk}
			+ 2[3 \omega^2 + \partial_t \omega - \Lambda]h_{jk} 
			- 4 \omega g^{0a} \partial_a h_{jk} \\
		& \ \ + 2 e^{-2\Omega} \triangle_{A;jk} 
			- 2(\partial_t \omega)(g^{00} + 1)h_{jk}
			- \frac{2(1 + \speed^2)}{\speed^2}e^{-2\Omega} e^{-3(1 + \speed^2)\Omega}P u_j u_k \notag \\
		& \ \ - \frac{1 - \speed^2}{\speed^2} e^{-3(1 + \speed^2)\Omega}P h_{jk}. \notag
		\end{align}
		Now using \eqref{E:3omegasquaredminusLambdaidentity} and \eqref{E:omegadotidentity}, it follows that
		$2[3 \omega^2 + \partial_t \omega - \Lambda]h_{jk} = \frac{1-\speed^2}{\speed^2} \bar{p}e^{-3(1 + \speed^2)\Omega} h_{jk}.$
		Substituting into \eqref{E:boxhjkfirstformula}, and using 
		$\partial_t \omega = - \frac{(1 + \speed^2)}{2\speed^2} \bar{p} e^{-3(1 + \speed^2)\Omega}$ (i.e., 
		\eqref{E:omegadotidentity}), it follows that
	
		\begin{align} \label{E:boxhjksecondformula}
			\hat{\square}_g h_{jk} & = 3 \omega \partial_t h_{jk}
				- \frac{1 - \speed^2}{\speed^2} e^{-3(1 + \speed^2)\Omega}(P - \bar{p}) h_{jk} 
				- 4 \omega g^{0a} \partial_a h_{jk} \\
			& \ \ + 2 e^{-2\Omega} \triangle_{A;jk} 
				+ \frac{(1 + \speed^2)}{\speed^2} \bar{p} e^{-3(1 + \speed^2)\Omega}(g^{00} + 1)h_{jk} \notag \\
			& \ \ - \frac{2(1 + \speed^2)}{\speed^2}e^{-2\Omega} e^{-3(1 + \speed^2)\Omega}P u_j u_k. \notag
		\end{align}	
		Equation \eqref{E:finalhjkequation} now easily follows from \eqref{E:boxhjksecondformula}. We remark that the proofs of 
		\eqref{E:finalg00equation} and \eqref{E:finallg0jequation} require the use of Lemma \ref{L:AmunuplusImunudecomposition}.
		
		To obtain
		\eqref{E:finalfirstEuler}, we first expand the covariant differentiation in \eqref{E:firstEulersummary}
		to deduce the following equation:
		
		\begin{align} \label{E:firstEulerChristoffel}
			u^{\alpha} \partial_{\alpha} p + (1+ \speed^2)p \partial_{\alpha} u^{\alpha} =
				- (1+ \speed^2)p \Gamma_{\alpha \ \beta}^{\ \alpha} u^{\beta}.
		\end{align}
		Lemma \ref{L:Christoffeldecomposition} implies that 
		$\Gamma_{\alpha \ \beta}^{\ \alpha} u^{\beta} = 3 \omega u^0 + \triangle_{\alpha \ \beta}^{\ \alpha} u^{\beta},$
		while the normalization condition $g_{\alpha \beta} u^{\alpha} u^{\beta} = -1$ implies that 
		$\partial_t u^0 = - \frac{1}{u_0} \big\lbrace u_a \partial_t u^a 
		+ \frac{1}{2}(\partial_t g_{\alpha \beta})u^{\alpha} u^{\beta} \big\rbrace.$ 
		Equation \eqref{E:finalfirstEuler} now follows from multiplying both sides of
		\eqref{E:firstEulerChristoffel} by $e^{3(1 + \speed^2) \Omega}$ and using the fact that
		$e^{3(1 + \speed^2) \Omega} \partial_t p = \partial_t P - 3(1 + \speed^2) \omega P.$
		
		Similarly, to obtain \eqref{E:finalEulerj}, we first expand the covariant differentiation in \eqref{E:secondEulersummary} 
		to deduce the following equation: 		
		
		\begin{align}  \label{E:finalEulerChristoffel}
			u^{\alpha} \partial_{\alpha} u^j + \frac{\speed^2}{(1 + \speed^2)p} \Pi^{j \alpha} \partial_{\alpha}p
				& = - \Gamma_{\alpha \ \beta}^{\ j} u^{\alpha} u^{\beta}.
		\end{align}
		Lemma \ref{L:Christoffeldecomposition} implies that 
		$\Gamma_{\alpha \ \beta}^{\ j} u^{\alpha} u^{\beta} = 2 \omega u^0 u^j +
		\triangle_{\alpha \ \beta}^{\ j} u^{\alpha} u^{\beta}.$ Equation \eqref{E:finalEulerj}
		now follows from \eqref{E:finalEulerChristoffel} and the fact that
		$\frac{\speed^2}{(1 + \speed^2)p} \Pi^{j \alpha} \partial_{\alpha} p 
		= \frac{\speed^2}{(1 + \speed^2)P} \Pi^{j \alpha} \partial_{\alpha} (P - \bar{p})
			- \decayparameter \omega \big\lbrace u^0 u^j + g^{0j} \big\rbrace.$
		
\end{proof}

We now state the following four lemmas, which are needed for the proof of Proposition \ref{P:Decomposition}.

\begin{lemma} \cite[Lemma 4]{hR2008} \label{L:modifiedRiccidecomposition}
	The modified Ricci tensor from \eqref{E:modifiedRicci} can be decomposed as follows:
	\begin{align}
		\widehat{\mbox{Ric}}_{\mu \nu} & = - \frac{1}{2} \hat{\Square}_g g_{\mu \nu}
			+ \frac{3}{2}(g_{0 \mu} \partial_{\nu} \omega + g_{0 \nu} \partial_{\mu} \omega)
			+ \frac{3}{2} \omega \partial_t g_{\mu \nu} + A_{\mu \nu}, &&  (\mu, \nu = 0,1,2,3), \\
			\mbox{where} \notag \\
		A_{\mu \nu} & = g^{\alpha \beta} g^{\kappa \lambda}
			\big[(\partial_{\alpha} g_{\nu \kappa})(\partial_{\beta} g_{\mu \lambda})
			- \Gamma_{\alpha \nu \kappa} \Gamma_{\beta \mu \lambda}\big], && (\mu, \nu = 0,1,2,3).	\label{E:Amunudef}
	\end{align}
	
\end{lemma}
\hfill $\qed$

\begin{lemma} \cite[Lemma 5]{hR2008} \label{L:Amunudecomposition}
	The term $A_{\mu \nu}$ ($\mu, \nu = 0,1,2,3$) defined in \eqref{E:Amunudef} can be decomposed into principal terms and error 
	terms $\triangle_{A,\mu \nu}$ as follows:
	
	\begin{subequations}
	
	\begin{align}
		A_{00} & = 3 \omega^2 - \omega g^{ab}\partial_t g_{ab} + 2 \omega g^{ab} \partial_a g_{0b} + \triangle_{A,00}, \\
		A_{0j} & = 2 \omega g^{00} \partial_t g_{0j} - 2 \omega^2 g^{00}g_{0j} - \omega g^{00} \partial_j g_{00} 
			+ \omega g^{ab}\Gamma_{ajb} + \triangle_{A,0j}, && (j=1,2,3), \\
		A_{jk} & = 2 \omega g^{00} \partial_t g_{jk} - 2 \omega^2 g^{00} g_{jk} + \triangle_{A,jk}, && (j,k = 1,2,3),
	\end{align}
	\end{subequations}
	where
	
	\begin{subequations}
	\begin{align}
		\triangle_{A,00} & = (g^{00})^2 \Big\lbrace(\partial_t g_{00})^2 - (\Gamma_{000})^2 \Big\rbrace 
			+ g^{00}g^{0a} \Big\lbrace 2 (\partial_t g_{00})(\partial_t g_{0a} + \partial_a g_{00}) 
			- 4 \Gamma_{000} \Gamma_{00a} \Big\rbrace \label{E:triangleA00def} \\
		& \ \ + g^{00}g^{ab} \Big\lbrace (\partial_t g_{0a})(\partial_t g_{0b}) 
				+ (\partial_a g_{00}) (\partial_b g_{00}) 
				- 2 \Gamma_{00a} \Gamma_{00b} \Big\rbrace \notag \\
		& \ \ + g^{0a} g^{0b} \Big\lbrace 2(\partial_t g_{00})(\partial_a g_{0b}) + 2(\partial_t g_{0b})(\partial_a g_{00}) 
				- 2 \Gamma_{000} \Gamma_{a0b} - 2 \Gamma_{00b} \Gamma_{00a} \Big\rbrace \notag \\
		& \ \ + g^{ab} g^{0l} \Big\lbrace 2(\partial_t g_{0a})(\partial_l g_{0b}) + 2(\partial_b g_{00})(\partial_a g_{0l}) 
				- 4\Gamma_{00a} \Gamma_{l0b}  \Big\rbrace \notag \\
		& \ \ + g^{ab}g^{lm}(\partial_a g_{0l})(\partial_b g_{0m}) 
				+ \frac{1}{2} g^{lm}(\underbrace{g^{ab} \partial_t g_{al} - 2\omega \delta_l^b}_{
				e^{2 \Omega} g^{ab} \partial_t h_{al} - 2 \omega g^{0b}g_{0l}})(\partial_b g_{0m} + \partial_m g_{0b}) \notag \\
		& \ \ - \frac{1}{4} g^{ab} g^{lm}(\partial_a g_{0l} + \partial_l g_{0a})(\partial_b g_{0m} + \partial_m g_{0b}) \notag \\
		& \ \ - \frac{1}{4}(\underbrace{g^{ab} \partial_t g_{al} - 2\omega \delta_l^b}_{
				e^{2 \Omega} g^{ab} \partial_t h_{al} - 2 \omega g^{0b}g_{0l}}) 
				(\underbrace{g^{lm} \partial_t g_{bm} - 2 \omega \delta_b^l}_{e^{2 \Omega} g^{lm} \partial_t h_{bm} 
				- 2 \omega g^{0l}g_{0b}}), \notag 
		\end{align}
		
		\begin{align}
		\triangle_{A,0j} & =  (g^{00})^2 \Big\lbrace (\partial_t g_{00}) (\partial_t g_{0j}) - \Gamma_{000} \Gamma_{0j0} \Big\rbrace
			\label{E:triangleA0jdef} \\
		& \ \ + g^{00}g^{0a} \Big\lbrace (\partial_t g_{00})(\partial_t g_{aj} + \partial_a g_{0j}) 
			+(\partial_t g_{0j})(\partial_t g_{0a} + \partial_a g_{00}) \notag \\
		& \hspace{1in} - 2 \Gamma_{000} \Gamma_{0ja} - 2 \Gamma_{0j0} \Gamma_{00a} \Big\rbrace \notag \\
		& \ \ + g^{00} (\underbrace{g^{ab} \partial_t g_{bj} - 2 \omega \delta_j^a}_{e^{2 \Omega} g^{ab} \partial_t h_{bj}
			- g^{0a}g_{0j}}) \Big(\partial_t g_{0a} - \frac{1}{2} \partial_a g_{00}\Big) 
			+ \frac{1}{2} g^{00} g^{ab}(\partial_a g_{00})(\partial_b g_{0j} + \partial_j g_{0b}) \notag \\
		& \ \ + g^{0a} g^{0b} \Big\lbrace (\partial_t g_{00})(\partial_a g_{bj}) + (\partial_t g_{0b})(\partial_a g_{0j})  			
			+ (\partial_a g_{00})(\partial_t g_{bj}) + (\partial_a g_{0b})(\partial_t g_{0j})  \notag  \\
		& \hspace{1 in} - \Gamma_{000} \Gamma_{ajb} - 2 \Gamma_{00b} \Gamma_{0ja} - \Gamma_{a0b} \Gamma_{0j0} \Big\rbrace \notag \\
		& \ \ + g^{ab} g^{0l} \Big\lbrace (\partial_t g_{0a})(\partial_l g_{bj}) + (\partial_l g_{0a})(\partial_t g_{bj}) 
			+ (\partial_b g_{00})(\partial_a g_{lj}) + (\partial_b g_{0l})(\partial_a g_{0j}) - 2 \Gamma_{00a} \Gamma_{ljb} 
			\Big\rbrace \notag \\
		& \ \ - g^{ab} g^{0l} \Big\lbrace (\partial_l g_{0a} + \partial_a g_{0l})\Gamma_{0jb} 
			-\frac{1}{2}(\partial_t g_{la})(\partial_b g_{0j} - \partial_j g_{0b}) \Big\rbrace \notag \\
		& \ \ + \omega g^{0a}(\underbrace{\partial_t g_{aj} - 2 \omega g_{aj}}_{e^{2 \Omega} \partial_t h_{aj}})
			+ \frac{1}{2}g^{0l}(\underbrace{g^{ab} \partial_t g_{la} - 2\omega \delta_{l}^b}_{e^{2 \Omega} g^{ab} \partial_t h_{la} 
				- 2 \omega g^{0b}g_{0l}})\partial_t g_{bj} \notag \\
		& \ \ + g^{ab} g^{lm} \Big\lbrace (\partial_a g_{0l})(\partial_b g_{mj}) 
			- \frac{1}{2} (\partial_a g_{0l} + \partial_l g_{0a})\Gamma_{bjm} \Big\rbrace 
			+ \frac{1}{2} g^{ab} (\underbrace{g^{lm} \partial_t g_{la} - 2\omega \delta_a^m}_{e^{2 \Omega} g^{lm} \partial_t h_{la} - 
			g^{0m}g_{0a}})\Gamma_{bjm}, \notag
		\end{align}
		
		\begin{align}
		\triangle_{A,jk} & = (g^{00})^2 \Big\lbrace (\partial_t g_{0j}) (\partial_t g_{0k}) 
			- \Gamma_{0j0} \Gamma_{0k0} \Big\rbrace \label{E:triangleAjkdef} \\
		& \ \ + g^{00}g^{0a} \Big\lbrace (\partial_t g_{0j})(\partial_t g_{ak} + \partial_a g_{0k}) 
			+ (\partial_t g_{0k})(\partial_t g_{aj} + \partial_a g_{0j}) \notag \\
		& \hspace{1in} - 2 \Gamma_{0j0} \Gamma_{0ka} - 2 \Gamma_{0k0} \Gamma_{0ja} \Big\rbrace \notag \\
		& \ \ + g^{00}g^{ab} \Big\lbrace (\partial_a g_{0j})(\partial_b g_{0k})  
			- \frac{1}{2}(\partial_a g_{0j} - \partial_j g_{0a})(\partial_b g_{0k} - \partial_k g_{0b}) \Big\rbrace	\notag \\
		& \ \ - \frac{1}{2} g^{00} \Big\lbrace (\underbrace{g^{ab} \partial_t g_{aj} - 2 \omega \delta_j^b}_{
			e^{2 \Omega} g^{ab} \partial_t h_{aj} - 2 \omega g^{0b}g_{0j}})(\partial_b g_{0k} - \partial_k g_{0b})
			+ (\underbrace{g^{ab} \partial_t g_{bk} - 2 \omega \delta_k^a}_{e^{2 \Omega}g^{ab} \partial_t 
			h_{bk} - 2 \omega g^{0a}g_{0k}})(\partial_a g_{0j} - \partial_j g_{0a}) \Big\rbrace \notag \\
		& \ \ + \omega g^{00}(\underbrace{g_{bk}g^{ab} - \delta_k^a}_{-g_{0k}g^{0a}})\partial_t g_{aj} 
			+ \frac{1}{2} g^{00}(\underbrace{g^{ab}\partial_t g_{aj} - 2 \omega \delta_j^b}_{e^{2 \Omega} g^{ab} \partial_t 
			h_{aj} - 2 \omega g^{0b}g_{0j}})(\underbrace{\partial_t g_{bk} - 2 \omega g_{bk}}_{e^{2 \Omega}\partial_t h_{bk}}) 
			\notag \\
		& \ \ + g^{0a} g^{0b} \Big\lbrace (\partial_t g_{0j})(\partial_a g_{bk}) + (\partial_t g_{bj})(\partial_a g_{0k})  			+ 
			(\partial_a g_{0j})(\partial_t g_{bk}) + (\partial_a g_{bj})(\partial_t g_{0k})  \notag  \\
		& \hspace{1in} - \Gamma_{0j0} \Gamma_{akb} - 2 \Gamma_{0jb} \Gamma_{0ka} - \Gamma_{ajb} \Gamma_{0k0} \Big\rbrace \notag \\
		& \ \ + g^{ab} g^{0l} \Big\lbrace (\partial_t g_{aj})(\partial_l g_{bk}) + (\partial_l g_{aj})(\partial_t g_{bk}) 
			+ (\partial_b g_{0j})(\partial_a g_{lk}) + (\partial_b g_{lj})(\partial_a g_{0k}) \notag \\
		& \hspace{1in} - 2 \Gamma_{0ja} \Gamma_{lkb} - 2 \Gamma_{lja} \Gamma_{0kb} \Big\rbrace \notag \\
		& \ \ + g^{ab} g^{ml} \Big\lbrace (\partial_a g_{lj})(\partial_b g_{mk}) - \Gamma_{ajl} \Gamma_{bkm} \Big\rbrace. 
		\notag
	\end{align}
	\end{subequations}
	
\end{lemma}
\hfill $\qed$

\begin{lemma} \cite[Lemma 6]{hR2008} \label{L:AmunuplusImunudecomposition}
	The sums $A_{00} + I_{00}$ and $A_{0j} + I_{0j},$ $(j=1,2,3),$ can be decomposed into principal terms and error terms as 
	follows, where $I_{00}, I_{0j}$ are defined in \eqref{E:gaugetermI00} - \eqref{E:gaugetermI0j}; $A_{00},A_{0j}$ are 
	defined in \eqref{E:Amunudef}; and $\triangle_{A,00}, \triangle_{A,0j}$ are defined in \eqref{E:triangleA00def} - 
	\eqref{E:triangleA0jdef}:
	
	\begin{subequations}
	\begin{align}
		A_{00} + 2 \omega \Gamma^0 - 6 \omega^2 & = \omega \partial_t g_{00} + 3 \omega^2(g_{00} + 1)
			+ 3 \omega^2 g_{00} + \triangle_{A,00} + \triangle_{C,00},  \\
		A_{0j} + 2\omega(3 \omega g_{0j} - \Gamma_j) & = 4 \omega^2 g_{0j} - \omega g^{ab} \Gamma_{ajb} + \triangle_{A,0j} 
			+ \triangle_{C,0j}, 
	\end{align}
	\end{subequations}
	where
	
	\begin{subequations}
	\begin{align}
		\triangle_{C,00} & = -6 (g_{00})^{-1} \omega^2 \Big\lbrace (g_{00} + 1)^2 - g^{0a}g_{0a} \Big\rbrace 
			- \omega (g^{00} + 1)(\underbrace{g^{ab} \partial_t g_{ab} - 6 \omega}_{e^{2\Omega}g^{ab} \partial_t h_{ab} - 2 \omega 
			g^{0a}g_{0a}})  \label{E:triangleC00def} \\
		& \ \ + 2\omega(g^{00} + 1)g^{ab} \partial_a g_{0b} + \omega(g^{00} + 1)(g^{00} - 1) \partial_t g_{00} + 2\omega g^{00} 
			g^{0a}(\Gamma_{0a0} + 2 \Gamma_{00a}) \notag \\
		& \ \ + 4\omega g^{0a} g^{0b} \Gamma_{0ab} + 2 \omega g^{ab} g^{0l} \Gamma_{alb}, \notag \\
		\triangle_{C,0j} & = 2 \omega^2(g^{00} + 1) g_{0j}
		 - 2 \omega g^{0a} \Big\lbrace (\underbrace{\partial_t g_{aj} - 2\omega g_{aj}}_{e^{2 \Omega} \partial_t h_{aj}}) + 
		 \partial_a g_{0j} - \partial_j g_{0a} \Big\rbrace. \label{E:triangleC0jdef}
	\end{align}
	\end{subequations}
\end{lemma}
\hfill $\qed$

\begin{lemma} \label{L:Christoffeldecomposition}
	The Christoffel symbols $\Gamma_{\mu \ \nu}^{\ \alpha}$ can be decomposed into principal terms and
	error terms $\triangle_{\mu \ \nu}^{\ \alpha}$ as follows:
	
	\begin{subequations}
	\begin{align}
		\Gamma_{0 \ 0}^{\ 0} & = \triangle_{0 \ 0}^{\ 0},  \label{E:triangleGamma000} \\
		\Gamma_{j \ 0}^{\ 0} = \Gamma_{0 \ j}^{\ 0} & = \triangle_{j \ 0}^{\ 0} = \triangle_{0 \ j}^{\ 0}, 
			\label{E:triangleGammaj00}   \\
		\Gamma_{0 \ 0}^{\ j} & = \triangle_{0 \ 0}^{\ j}, \label{E:triangleGamma0j0} \\
		\Gamma_{0 \ k}^{\ j} = \Gamma_{k \ 0}^{\ j} & = \omega \delta_k^j +  \triangle_{0 \ k}^{\ j}  
			= \omega \delta_k^j +  \triangle_{k \ 0}^{\ j}, \label{E:triangleGamma0jk} \\
		\Gamma_{j \ k}^{\ 0} & =  \omega g_{jk} + \triangle_{j \ k}^{\ 0},
			\label{E:triangleGammaj0k}  \\
		\Gamma_{i \ j}^{\ k} & = \triangle_{i \ j}^{\ k},  \label{E:triangleGammaikj}
\end{align}
\end{subequations}
where
	
	\begin{subequations}
	\begin{align}
		2 \triangle_{0 \ 0}^{\ 0} & = g^{00} \partial_t g_{00} + 2 g^{0a} \partial_t g_{0a} - g^{0a}\partial_a g_{00}, 
			\label{E:triangle000} \\
		2 \triangle_{j \ 0}^{\ 0} & = g^{00} \partial_j g_{00} + g^{0a}(\partial_j g_{a0} - \partial_a g_{j0}) 
			+ 2 \omega g^{0a} g_{ja}
			+ g^{0a}(\underbrace{\partial_t g_{ja} - 2 \omega g_{ja}}_{e^{2 \Omega} \partial_t h_{aj}}), \label{E:trianglej00} \\
		2 \triangle_{0 \ 0}^{\ j} & = g^{j0} \partial_t g_{00} + 2 g^{ja}\partial_t g_{0a} - g^{ja} \partial_a g_{00},  
			\label{E:triangle00j}\\
		2 \triangle_{0 \ k}^{\ j} & = g^{j0} \partial_k g_{00} + g^{ja} \partial_k g_{0a} - g^{ja} \partial_a g_{0k}
			+ (\underbrace{g^{ja} \partial_t g_{ak} - 2 \omega \delta_k^j}_{e^{2 \Omega} g^{ja} \partial_t 
			h_{ak} - 2 \omega g^{0j}g_{0k}}), \label{E:triangle0kj} \\
		2 \triangle_{j \ k}^{\ 0} & = g^{00}(\partial_j g_{0k} + \partial_k g_{0j}) + g^{0a}(\partial_j g_{ak} 
			+ \partial_k g_{aj} - \partial_a g_{jk}) \label{E:trianglejk0}\\
		& \ \ + (\underbrace{\partial_t g_{jk} - 2 \omega g_{jk}}_{e^{2 \Omega} \partial_t h_{jk}}) - 2(g^{00} + 1) \omega g_{jk} 
			- (g^{00} + 1)(\underbrace{\partial_t g_{jk} - 2 \omega g_{jk}}_{e^{2 \Omega} \partial_t h_{jk}}), \notag \\
		2 \triangle_{i \ j}^{\ k} & = g^{ka}(\partial_i g_{aj} + \partial_j g_{ia} - \partial_a g_{ij}). 
			\label{E:triangleikj}
	\end{align}
	\end{subequations}

\end{lemma}

\begin{proof}
	The proof is again a series of tedious computations that follow from the definition 
	$\Gamma_{\mu \ \nu}^{\ \alpha} = \frac{1}{2} g^{\alpha \lambda} 
	(\partial_{\mu} g_{\lambda \nu} + \partial_{\nu} g_{\mu \lambda} - \partial_{\lambda} g_{\mu \nu}).$ 
\end{proof}

\subsection{Classical local existence and the continuation principle} \label{SS:ClassicalLocalExistence}

In this section, we discuss classical local existence results and continuation criteria for the modified system \eqref{E:finalg00equation} - \eqref{E:finalEulerj}. The theorems in this section are stated without proof; we provide references for the rather standard techniques that can be used to prove them.

\begin{theorem}[\textbf{Local Existence for the Modified System}] \label{T:LocalExistence}
	Let $N \geq 3$ be an integer. Let $\mathring{g}_{\mu \nu} = g_{\mu \nu}|_{t=0},$ $2 \mathring{K}_{\mu \nu} = 
	(\partial_t g_{\mu \nu})|_{t=0},$ $(\mu, \nu = 0,1,2,3),$
	$\mathring{P} = P|_{t=0} = p|_{t=0},$ $\mathring{u}^{j} = u^{j}|_{t=0},$ $(j=1,2,3),$ 
	$g_{\alpha \beta}u^{\alpha}u^{\beta}|_{t=0} = -1,$ 
	be initial data (not necessarily satisfying the Einstein constraints) on the manifold $\mathbb{T}^3$ 
	for the modified system \eqref{E:Rhat00} - \eqref{E:secondEulersummary} satisfying (for $j,k = 1,2,3$)
	
	\begin{subequations}
	\begin{align}
		\underpartial \mathring{g}_{jk} & \in H^{N}, \ \mathring{g}_{00} + 1 \in H^{N+1}, \ \mathring{g}_{0j} \in H^{N+1}, \\
		\mathring{K}_{jk} - \omega(0)\mathring{g}_{jk} & \in H^{N}, \ 
			\mathring{K}_{00} \in H^{N}, \ 
			\mathring{K}_{0j} \in H^{N}, \\ 
		\ \mathring{P} - \bar{p} & \in H^{N}, \ \mathring{u}^j \in H^N,
	\end{align}
	\end{subequations}
	where $\bar{p} > 0$ is a constant. Assume that $\inf_{x \in \mathbb{T}^3} \mathring{p} > 0.$
	Assume further that there is a constant $C > 0$ such that
	
	\begin{align} \label{E:localexistencemathringgjklowerequivalenttostandardmetric}
		C \delta_{ab} X^a X^b \leq \mathring{g}_{ab}X^a X^b \leq C^{-1} \delta_{ab} X^a X^b, &&
		\forall(X^1,X^2,X^3) \in \mathbb{R}^3,
	\end{align}
	and such that $\mathring{g}_{00} < 0,$ so that by Lemma \ref{L:ginverseformluas}, the $4 \times 4$ matrix $\mathring{g}_{\mu 
	\nu}$ is Lorentzian. Then these data launch a unique classical solution $(g_{\mu \nu}, P, u^{\mu})$ to the modified system 
	existing on a slab $[T_-, T_+] \times \mathbb{T}^3,$ with $T_- < 0 < T_+,$ such that
	
	\begin{align}
		g_{\mu \nu} \in C^2([T_-, T_+] \times \mathbb{T}^3), \ P \in C^1([T_-, T_+] \times \mathbb{T}^3), 
		\ u^{\mu} \in C^1([T_-, T_+] \times \mathbb{T}^3),
	\end{align}
	such that $g_{00} < 0,$ and such that the eigenvalues of the $3 \times 3$ matrix $g_{jk}$ are uniformly bounded
	below from $0$ and from above. 
	
	The solution has the following regularity properties:
	
	\begin{subequations}
	\begin{align}
		\underpartial g_{jk} & \in C^0([T_-, T_+],H^{N}), \ g_{00} + 1 \in C^0([T_-, T_+],H^{N+1}), 
			\ g_{0j} \in C^0([T_-, T_+],H^{N+1}), \\
		\partial_t g_{jk} - 2\omega(t) g_{jk} & \in C^0([T_-, T_+],H^{N}), \ \partial_t g_{00} \in C^0([T_-, T_+],H^{N}), 
			\ \partial_t g_{0j} \in C^0([T_-, T_+],H^{N}), \\ 
		P - \bar{p} & \in C^0([T_-, T_+],H^{N}), \ u^j \in C^0([T_-, T_+],H^{N}), \ u^0 - 1 \in C^0([T_-, T_+],H^{N}).
	\end{align}
	\end{subequations}
	
	Furthermore, $g_{\mu \nu}$ is a smooth Lorentzian metric on $(T_-,T_+) \times \mathbb{T}^3,$
	and the sets $\lbrace t \rbrace \times \mathbb{T}^3$ are Cauchy hypersurfaces in the Lorentzian manifold 
	$(\mathcal{M} \eqdef (T_-, T_+) \times \mathbb{T}^3,g_{\mu \nu})$ for $t \in (T_-, T_+).$
	
	In addition, there exists an open neighborhood $\mathcal{O}$ of
	$(\mathring{g}_{\mu \nu}, \mathring{K}_{\mu \nu}, \mathring{P}, \mathring{u}^{j})$
	such that all data belonging to $\mathcal{O}$ launch solutions that also exist on the slab 
	$[T_-, T_+] \times \mathbb{T}^3$ and that have the 
	same regularity properties as $(g_{\mu \nu},P,u^{\mu}).$ Furthermore, on $\mathcal{O},$ the map from the initial data to the 
	solution is continuous.\footnote{By continuous, we mean continuous relative to the norms on the data and the norms on the 
	solution that are stated in the hypotheses and above conclusions of the theorem.}
	
	Finally, if, as described in Section \ref{SS:IDREDUCED}, the data for the modified system are constructed from data for the 
	Einstein-Euler system satisfying the constraints \eqref{E:Gauss} - \eqref{E:Codazzi} on an open subset $S \in 
	\mathbb{T}^3,$ and if the wave coordinate condition $Q_{\mu}|_{S} = 0$ holds, then 
	$(g_{\mu \nu}, p ,u^{\mu})$ is also a solution to the unmodified equations \eqref{E:EulerEinstein} - \eqref{E:fluid2}  on 
	$\mathcal{D}(S),$ the Cauchy development of $S.$

\end{theorem}

\begin{remark}
	The hypotheses in Theorem \ref{T:LocalExistence} have been stated in a manner that allows us to apply to it initial data 
	near that of the background solution of Section \ref{S:backgroundsolution}. Furthermore, we remark that
	the assumptions and conclusions concerning the metric components $g_{jk}$ would appear more natural if expressed in terms of
	the variables $h_{jk} \eqdef e^{-2 \Omega} g_{jk}$; these rescaled quantities are the ones that we use in our global 
	existence proof.
\end{remark}

\begin{proof}
	The existence aspect of Theorem \ref{T:LocalExistence} can be proved using standard methods that follow from energy estimates 
	in the spirit of those proved below in Sections \ref{SS:GENERGIES}, \ref{SS:fluidEnergies}, and 
	\ref{S:EnergyNormEquivalence}. See e.g. \cite[Ch. VI]{lH1997}, \cite[Ch. 2]{aM1984}, \cite[Ch. 5]{jSmS1998}, 
	\cite[Ch. 1]{cS2008}, \cite{jS2008a}, and \cite[Ch. 16]{mT1997III} for details on how to prove local existence as a 
	consequence of the availability of these kinds of energy estimates. Also see \cite[Proposition 1]{hR2008}. 
	The fact that $(g_{\mu \nu},p,u^{\mu}),$ where $p = e^{-3(1 + \speed^2)\Omega}P,$ 
	is also a solution to the modified equations if the constraints and the 
	wave coordinate condition $Q_{\mu}|_{S} = 0$ are satisfied is discussed 
	more fully in Section \ref{SS:PreservationofHarmonicGauge}.

	\begin{remark}
		Our use of energy currents in Sections \ref{SS:fluidEnergies} and \ref{S:EnergyNormEquivalence} to derive energy estimates
		for the relativistic Euler equations may be unfamiliar to some readers. However, once one has such estimates, local 
		existence can be proved using the standard arguments mentioned above.
	\end{remark}
	
\end{proof}

In our proof of Theorem \ref{T:GlobalExistence}, we will use the following continuation principle, which provides
standard criteria that are sufficient to ensure that a solution to the modified equations exists globally in time.

\begin{theorem}[\textbf{Continuation Principle}] \label{T:ContinuationCriterion}
	Assume the hypotheses of Theorem \ref{T:LocalExistence}. Let $T_{max}$ be the supremum over all times $T_+$ such that the 
	solution $(g_{\mu \nu},p,u^{\mu})$ exists on the interval $[0,T_+)$ and has the properties stated in the conclusions of 
	Theorem \ref{T:LocalExistence}. Then if $T_{max} < \infty,$ one of the following four possibilities must occur:
	
	\begin{enumerate}
		\item There is a sequence $(t_n,x_n) \in [0,T_{max}) \times \mathbb{T}^3$ such that
			$\lim_{n \to \infty} g_{00}(t_n,x_n) = 0.$
		\item There is a sequence $(t_n,x_n) \in [0,T_{max}) \times \mathbb{T}^3$ such that the smallest eigenvalue of
			$g_{jk}(t_n,x_n)$ converges to $0$ as $n \to \infty$. 
		\item There is a sequence $(t_n,x_n) \in [0,T_{max}) \times \mathbb{T}^3$ such that
			$\lim_{n \to \infty} P(t_n,x_n) = 0.$
		\item $\lim_{t \to T_{max}^-} \sup_{0 \leq \tau \leq t} \Bigg\lbrace 
			\sum_{\mu,\nu =0}^3 \Big(\| g_{\mu \nu}(\tau,\cdot) \|_{C_b^2}
			+ \| \partial_t g_{\mu \nu}(\tau,\cdot) \|_{C_b^1} \Big) \\
				+ \| P(\tau,\cdot) \|_{C_b^1} + \| \partial_t P(\tau,\cdot) \|_{L^{\infty}}
				+ \sum_{j=1}^3 \Big( \| u^j(\tau,\cdot) \|_{C_b^1} + \| \partial_t u^j(\tau,\cdot) \|_{L^{\infty}} \Big) \Bigg\rbrace = 
				\infty.$
		\end{enumerate}
		Similar results hold for an interval of the form $(T_{min},0].$
\end{theorem}

\begin{proof}
	See e.g. \cite[Ch. VI]{lH1997}, \cite[Ch. 1]{cS2008}, \cite{jS2008b} for the ideas behind a proof. 
	If either $(1)$ or $(2)$ occurs, then the hyperbolicity of the operator $\hat{\Square}_g$ can break down.
	Condition $(3)$ is connected to the fact that the Euler equations can degenerate when the pressure vanishes.
\end{proof}

\subsection{Preservation of the wave coordinate condition} \label{SS:PreservationofHarmonicGauge}

In Section \ref{SS:IDREDUCED}, from given initial data for the Einstein equations, we constructed initial data for the modified equations that in particular satisfy the wave coordinate condition along the Cauchy hypersurface $\mathring{\Sigma};$ i.e., $Q_{\mu}|_{t=0} = 0.$ As mentioned in the statement of Theorem \ref{T:LocalExistence}, these data launch a solution of both the modified equations and the Einstein equations. As we have discussed previously, this fact follows from the fact that the condition $Q_{\mu} \equiv 0$ holds in $\mathcal{D}(\mathring{\Sigma}).$ We now formulate this result as a proposition.

\begin{proposition}[\textbf{Preservation of the Wave Coordinate Condition}] \label{P:Preservationofgauge}
	Let \\ 
	$\big(\mathring{\Sigma} = \lbrace x \in \mathcal{M} \ | \ t=0 \rbrace, \mathring{\underline{g}}_{jk}, 
	\mathring{\underline{K}}_{jk}, \mathring{p}, \underline{\mathring{u}}^j),$ 
	$(j,k=1,2,3),$ be initial data for the Euler-Einstein system \eqref{E:EinsteinFieldsummary} 
	- \eqref{E:ConservationofParticleSummary} that satisfy the constraints \eqref{E:Gauss} - \eqref{E:Codazzi}. Let $\big(g_{\mu 
	\nu}|_{t=0}, \partial_t g_{\mu \nu}|_{t=0}, p|_{t=0}, u^{\mu}|_{t=0}\big),$ $(\mu,\nu=0,1,2,3),$ be initial data for the 
	modified equations \eqref{E:FirstmodifiedRicci} - \eqref{E:Eulerjmodified} that are constructed from the data for 
	the Einstein equations as described in Section \ref{SS:IDREDUCED}. In particular, we recall that the construction of Section 
	\ref{SS:IDREDUCED} leads the fact that $Q_{\mu}|_{t=0} = 0$ where $Q_{\mu}$
	is defined in \eqref{E:Qdef}. Let $(\mathcal{M},g_{\mu \nu}, p, u^{\mu})$ be the maximal 
	globally hyperbolic development of the data for the modified equations. Then $Q_{\mu} = 0$ in 
	$\mathcal{D}(\mathring{\Sigma}) = \mathcal{M}.$
\end{proposition}

\begin{remark}
	To prove Proposition \ref{P:Preservationofgauge}, we do not need the full assumption that the equation of state is of the 
	form $p= \speed^2 \rho;$ we only need the assumption $p = p(\rho)$ and the assumptions described in Section \ref{S:EE}.
\end{remark}

\begin{proof}
	This is a rather standard result whose main ideas can be traced back to \cite{CB1952}. 
	The sketch of a proof given in \cite[Proposition 5.6.1]{iRjS2009}
	can be easily adapted to apply to the system we are studying here. The basic idea is to show that for any
	solution to the modified equations, the quantities $Q_{\mu} = \Gamma_{\mu} - \widetilde{\Gamma}_{\mu}$ 
	are solutions to a homogeneous system of wave equations with trivial initial data, i.e., that 
	$Q_{\mu}|_{t=0} = \partial_t Q_{\mu}|_{t=0} = 0.$ The fact that $Q_{\mu} \equiv 0$ then follows from 
	the standard uniqueness theorem for such systems.
\end{proof}

\subsection{Matrix-vector notation for the fluid variables}
To streamline our notation, it will often be useful to write the fluid equations \eqref{E:finalfirstEuler} - \eqref{E:finalEulerj} using PDE matrix-vector notation. To this end, we introduce the array of fluid variables $\mathbf{W},$ the array of decay-inducing inhomogeneous terms $\mathbf{b},$ and the array of error inhomogeneous terms $\mathbf{b}_{\triangle},$ which are defined by

\begin{align} \label{E:arraydefs}
	\mathbf{W} \eqdef 
		\left( \begin{array}{c}
			P - \bar{p} \\
			u^1 \\
			u^2 \\
			u^3
		\end{array} \right), \qquad 
		\mathbf{b} \eqdef 	\left( \begin{array}{c}
			0 \\
			(\decayparameter - 2) \omega u^{1} \\
			(\decayparameter - 2) \omega u^{2} \\
			(\decayparameter - 2) \omega u^{3}
		\end{array} \right), \qquad
		\mathbf{b}_{\triangle} \eqdef 	\left( \begin{array}{c}
			\triangle \\
			\triangle^1 \\
			\triangle^2 \\
			\triangle^3
		\end{array} \right).
\end{align}	
In the above formulas, $\decayparameter = 3 \speed^2,$ and $\triangle,$ $\triangle^j$ are defined in \eqref{E:triangledef} - \eqref{E:trianglejdef}.

We also introduce the $4 \times 4$ matrices $A^{\mu},$ which are defined by

\begin{subequations}
\begin{align}
	A^0 & =   \begin{pmatrix}
                        u^0 & -(1 + \speed^2)P \frac{u_1}{u_0}  & -(1 + \speed^2)P\frac{u_2}{u_0}  
                        	& -(1 + \speed^2) P \frac{u_3}{u_0}  \\
                       	\Big( \frac{\speed^2}{(1 + \speed^2)P} \Big) \Pi^{10} & u^0 & 0 & 0 \\
                        \Big( \frac{\speed^2}{(1 + \speed^2)P} \Big) \Pi^{20} & 0 & u^0 & 0\\
                        \Big( \frac{\speed^2}{(1 + \speed^2)P} \Big) \Pi^{30} & 0 & 0 & u^0 \\
                    \end{pmatrix}, \label{E:A0def} \\
 A^1 & =   \begin{pmatrix} \label{E:A1def}
                        u^1 & (1 + \speed^2)P & 0  & 0  \\
                       	\frac{\speed^2}{(1 + \speed^2)P} \Pi^{11} & u^1 & 0 & 0 \\
                        \frac{\speed^2}{(1 + \speed^2)P} \Pi^{21} & 0 & u^1 & 0\\
                        \frac{\speed^2}{(1 + \speed^2)P} \Pi^{31} & 0 & 0 & u^1 \\
                    \end{pmatrix},                   
\end{align}
\end{subequations}
and analogously for for $A^2,$ $A^3.$ For future use, we also calculate that the matrix $A^0$ satisfies
$\mbox{det}(A^0) = (u^0)^2 \Big\lbrace \overbrace{(u^0)^2(1-\speed^2) - \speed^2 g^{00}}^{(u^0)^2 - \speed^2 \Pi^{00}} \Big\rbrace,$ and its inverse is given by

\begin{align} \label{E:A0inverse}
	(&A^0)^{-1} = \big\lbrace (u^0)^2 - \speed^2 \Pi^{00} \big\rbrace^{-1} \\ 
										& \times \begin{pmatrix}
                       	u^0 & (1 + \speed^2)P \frac{u_1}{u_0} & (1 + \speed^2)P \frac{u_2}{u_0} 
                       		& (1 + \speed^2)P \frac{u_3}{u_0}  \\
                       	- \frac{\speed^2}{(1 + \speed^2)P} \Pi^{10} & 
                       		u^0 + \frac{\speed^2}{u_0 u^0}\Big(\Pi^{20}u_2 + \Pi^{30}u_3\Big) 
                       		& - \frac{\speed^2}{u_0 u^0} \Pi^{10}u_2  & -\frac{\speed^2}{u_0 u^0} \Pi^{10}u_3 \\
                        - \frac{\speed^2}{(1 + \speed^2)P} \Pi^{20} & - \frac{\speed^2}{u_0 u^0} \Pi^{20}u_1 
                        	& u^0 + \frac{\speed^2}{u_0 u^0}\Big(\Pi^{10}u_1 + \Pi^{30}u_3\Big) & 
                        		- \frac{\speed^2}{u_0 u^0} \Pi^{20}u_3 \\
                        - \frac{\speed^2}{(1 + \speed^2)P} \Pi^{30} & 
                        	- \frac{\speed^2}{u_0 u^0} \Pi^{30}u_1 & - \frac{\speed^2}{u_0 u^0} \Pi^{30}u_2 
                        	& u^0 + \frac{\speed^2}{u_0 u^0}\Big(\Pi^{10}u_1 + \Pi^{20}u_2\Big) \\
                    \end{pmatrix}. \notag
\end{align}

Using this notation, the fluid equations \eqref{E:finalfirstEuler} - \eqref{E:finalEulerj} can be written in
the following compact form:

\begin{align} \label{E:matrixvectorfluidequations}
	A^{\beta} \partial_{\beta} \mathbf{W} = \mathbf{b} + \mathbf{b}_{\triangle}.
\end{align}
We remark that we have split the inhomogeneous term into two pieces to facilitate our analysis in later sections.

In Section \ref{SS:Nonlinearities}, we will provide estimates for the time derivatives of the fluid quantities. Therefore, as a preliminary step, we isolate them in the next corollary.

\begin{corollary} \label{C:isolatepartialtPandpartialtuj}
Let $(P,u^1,u^2,u^3)$ be a solution to the modified relativistic Euler equations \eqref{E:finalfirstEuler} - \eqref{E:finalEulerj}, where $u^0$ is defined by \eqref{E:U0UPPERISOLATED}. Then the time derivatives of the fluid quantities can be expressed as follows:

\begin{subequations}
\begin{align}
	\partial_t (P - \bar{p}) & = \triangle', \label{E:partialtP} \\  
	\partial_t u^0 & = \triangle'^0, \label{E:partialtu0} \\
	\partial_t u^j & = (\decayparameter - 2)\omega u^j + \triangle'^j, \label{E:partialtuj}
\end{align}
\end{subequations}
where

\begin{subequations}
\begin{align}
	\triangle' & = \frac{1}{u^0} \Big\lbrace 1 -\frac{1}{(u^0)^2} \speed^2 \Pi^{00} \Big\rbrace^{-1} \label{E:triangleprimedef} \\
		& \times \Big\lbrace -u^a \partial_a P - (1 + \speed^2)P \partial_a u^a - \frac{(1 + \speed^2)P}{u_0u^0}u^a g_{b \beta}u^{\beta} 	
		\partial_a u^b \notag \\ 
	& \ \ \ \ \ + \frac{\speed^2}{u^0}\Pi^{a0}\partial_a P + \frac{(1 + \speed^2)P}{u_0u^0}(\decayparameter - 2)\omega g_{a 
		\alpha}u^{\alpha} u^a
		+ \triangle + \frac{(1 + \speed^2)P}{u_0u^0} g_{a \alpha} u^{\alpha} \triangle^a \Big\rbrace, \notag \\
	\triangle'^0 & = \frac{1}{u^0} \Big\lbrace \triangle^0 - \frac{\speed^2}{(1 + \speed^2)P}\Pi^{00} \triangle'
		- \frac{\speed^2}{(1 + \speed^2)P}\Pi^{0a} \partial_a P \Big\rbrace, \label{E:triangleprime0def} \\
	\triangle'^j & = \frac{1}{u^0} \Big\lbrace - \triangle_{\alpha \ \beta}^{\ j} u^{\alpha} u^{\beta}
		+ \decayparameter \omega g^{0j} - \frac{\speed^2}{(1 + \speed^2)P}\Pi^{0j}\triangle' 
		- \frac{\speed^2}{(1 + \speed^2)P}\Pi^{aj}\partial_a P - u^a \partial_a u^j \Big\rbrace. \label{E:triangleprimejdef}
\end{align}
\end{subequations}

\end{corollary}

\begin{proof}

To obtain \eqref{E:partialtP}, simply multiply both sides of \eqref{E:matrixvectorfluidequations} by the first row of $(A^0)^{-1},$ use \eqref{E:arraydefs}, \eqref{E:A1def} (plus the analogous, but unwritten, formulas for $A^2,A^3$) and \eqref{E:A0inverse}, and consider the first component of the resulting expression. Equation \eqref{E:partialtuj} then follows from isolating $\partial_t u^j$ in equation \eqref{E:finalEulerj} and using \eqref{E:partialtP} to substitute $\triangle'$ for $\partial_t(P - \bar{p}).$ Equation \eqref{E:partialtu0} similarly follows from \eqref{E:u0equation} and \eqref{E:partialtP}.

\end{proof}

\section{Norms and Energies} \label{S:NormsandEnergies}

In this section, we define the Sobolev norms\footnote{Technically, $\fluidnorm{N}$ is a norm of the difference between the perturbed fluid variables and the background FLRW fluid solution $(\widetilde{P},\widetilde{u}^1,\widetilde{u}^2,\widetilde{u}^3)
= (\bar{p},0,0,0).$} and energies that will play a central role in our global existence theorem of Section \ref{S:GlobalExistence}. They are designed with equations \eqref{E:finalg00equation} - \eqref{E:finalEulerj} and Theorem
\ref{T:ContinuationCriterion} in mind. Let us make a few comments on these quantities. First, we remark that in Section \ref{S:EnergyNormEquivalence}, we will show that if the norms are sufficiently small, then they are equivalent to the energies; i.e., the energies can be used to control Sobolev norms of solutions. The reason that we introduce the energies is that their time derivatives can be estimated via integration by parts/ the divergence theorem. Next, we recall that the background solution variables $(\widetilde{P},\widetilde{u}^1,\widetilde{u}^2,\widetilde{u}^3)$ satisfy $(\widetilde{P},\widetilde{u}^1,\widetilde{u}^2,\widetilde{u}^3) = (\bar{p},0,0,0)$ where $\bar{p} > 0$ is the initial (constant) pressure, and the rescaled pressure $P$ is defined in \eqref{E:RESCALEDPRESSURE}. The quantity 
$\fluidnorm{N}$ introduced below in \eqref{E:fluidnorm} measures the difference of the perturbed fluid variables $(P,u^1,u^2,u^3)$ from the background $(\bar{p},0,0,0).$ We also follow Ringstr\"{o}m \cite{hR2008}
by introducing scalings by $e^{\upalpha \Omega},$ where $\upalpha$ is a number, in the definitions of the norms and energies. The effect of these scalings is that in our proof of global existence, a convenient and viable bootstrap assumption to make for these quantities is that they are of size $\epsilon,$ where $\epsilon$ is sufficiently small. Finally, we remark that the small positive number $q$ that appears in this section and throughout this article is defined in \eqref{E:qdef} below, and we remind the reader that $h_{jk} \eqdef e^{-2\Omega} g_{jk}, P \eqdef e^{3(1 + \speed^2)\Omega}p.$

\subsection{Norms for \texorpdfstring{$g_{\mu \nu}, P - \bar{p}, u^j$}{the spacetime metric and the fluid quantities}} \label{SS:Normsforgandfluidvariables}

\begin{definition} \label{D:Norms}
Let $N$ be a positive integer. We define the norms $\gzerozeronorm{N}(t),$ $\gzerostarnorm{N}(t),$ $\hstarstarnorm{N}(t),$
$\gnorm{N}(t),$ $U_{N-1}(t),$ $\fluidnorm{N}(t),$ and $\totalnorm{N}(t)$ as follows:

\begin{subequations}
\begin{align}
	\gzerozeronorm{N} & \eqdef e^{q \Omega} 
			\| \partial_t g_{00} \|_{H^N} +  e^{q \Omega} \| g_{00} + 1 \|_{H^N}
			+ e^{(q - 1)\Omega} \| \underpartial g_{00} \|_{H^N}, \label{E:mathfrakSNg00} \\
	\gzerostarnorm{N} & \eqdef 
		\sum_{j=1}^3 \Big( e^{(q - 1) \Omega} \| \partial_t g_{0j} \|_{H^N} 
			+  e^{(q - 1) \Omega} \| g_{0j} \|_{H^N} 
			+ e^{(q - 2)\Omega} \| \underpartial g_{0j} \|_{H^N} \Big),  	
			\label{E:mathfrakSNg0*} \\
	\hstarstarnorm{N} & \eqdef \sum_{j,k=1}^3 \Big( e^{q \Omega} 
		\| \partial_t h_{jk} \|_{H^N} + \| \underpartial h_{jk} \|_{H^{N-1}}
		+ e^{(q - 1)\Omega} \| \underpartial h_{jk} \|_{H^N} \Big), \label{E:mathfrakSNh**} \\
	\gnorm{N} & \eqdef \gzerozeronorm{N} + \gzerostarnorm{N} + \hstarstarnorm{N},		
		\label{E:supgnorm} 
\end{align}
\end{subequations}

\begin{subequations}
\begin{align}	
U_{N-1} & \eqdef e^{(1 + q) \Omega} \Big(\sum_{j=1}^3 \| u^{j} \|_{H^{N-1}}^2 \Big)^{1/2}, \label{E:ujnorm} \\
	\fluidnorm{N} & \eqdef e^{\Omega} \| u^{j} \|_{H^N} + \| P - \bar{p} \|_{H^N}, \label{E:fluidnorm} 
\end{align}
\end{subequations}

\begin{align}
	\totalnorm{N} & \eqdef \gnorm{N} + \fluidnorm{N} + U_{N-1}. \label{E:totalnorm}
\end{align}

Note that as discussed in Section \ref{SS:Commentsonanalysis}, the norm $U_{N-1}$ in \eqref{E:totalnorm}
controls the lower-order derivatives of the $u^j$ with \emph{larger weights} than the norm $\fluidnorm{N}.$

\end{definition}

\subsection{Energies for the metric \texorpdfstring{$g$}{the metric}} \label{SS:GENERGIES}

\subsubsection{The building block energy for \texorpdfstring{$g_{\mu \nu}$}{the metric components}}

The energies for the metric components will be built from the quantities defined in the following lemma. They are designed 
with equations \eqref{E:finalg00equation} - \eqref{E:finalhjkequation} in mind.

\begin{lemma} \cite[Lemma 15]{hR2008} \label{L:buildingblockmetricenergy}
Let $v$ be a solution to the scalar equation

\begin{align} \label{E:vscalar}
	\hat{\square}_g v = \upalpha H \partial_t v + \upbeta H^2 v + F,
\end{align}

\noindent where $	\hat{\square}_g = g^{\lambda \kappa} \partial_{\lambda} \partial_{\kappa},$ 
$\upalpha > 0$ and $\upbeta \geq 0,$ and define $\mathcal{E}_{(\upgamma,\updelta)}[v, \partial v] \geq 0$ by

\begin{align} \label{E:mathcalEdef}
		\mathcal{E}_{(\upgamma,\updelta)}^2[v,\partial v] \eqdef \frac{1}{2} \int_{\mathbb{T}^3} 
		\big\lbrace -g^{00} (\partial_t v)^2 + g^{ab}(\partial_a v)(\partial_b v) - 2 \upgamma H g^{00} v \partial_t v
		+ \updelta H^2 v^2 \big\rbrace \, d^3 x.
\end{align}

Then there are constants 
$\upeta > 0,$ $C >0,$ $\updelta \geq 0,$ and $\upgamma \geq 0,$ with $\upeta$ and $C$ depending 
on $\upalpha,$ $\upbeta,$ $\upgamma$ and $\updelta,$ such that 

\begin{align}
	|g^{00} + 1| \leq \upeta
\end{align}
implies that

\begin{align} \label{E:mathcalEfirstlowerbound}
	\mathcal{E}_{(\upgamma,\updelta)}^2[v,\partial v] \geq C \int_{\mathbb{T}^3} (\partial_t v)^2 + g^{ab}(\partial_a 
		v)(\partial_b v) + C_{(\upgamma)} v^2 \, d^3x,
\end{align} where $C_{(\upgamma)} = 0$ if $\upgamma = 0$ and $C_{(\upgamma)} = 1$ if $\upgamma > 0 .$ Furthermore, if
$\upbeta = 0,$ then $\upgamma = \updelta = 0.$ Finally, we have that

\begin{align} \label{E:mathcalEtimederivativebound}
	\frac{d}{dt} (\mathcal{E}_{(\upgamma,\updelta)}^2[v,\partial v]) & \leq - \upeta H \mathcal{E}_{(\upgamma,\updelta)}^2[v,\partial v] 
		+ \int_{\mathbb{T}^3} \Big\lbrace - (\partial_t v + \upgamma H v)F + \triangle_{\mathcal{E};(\upgamma, \updelta)}[v,\partial v] \Big\rbrace \, d^3 x,
\end{align}
where

\begin{align} \label{E:trianglemathscrEdef}
	\triangle_{\mathcal{E};(\upgamma, \updelta)}[v,\partial v] & = - \upgamma H (\partial_a g^{ab}) v \partial_b v
		- 2 \upgamma H (\partial_a g^{0a}) v \partial_t v - 2 \upgamma H g^{0a}(\partial_a v)(\partial_t v) \\
	& \ \ - (\partial_a g^{0a})(\partial_t v)^2 - (\partial_a g^{ab})(\partial_b v)(\partial_t v)
		- \frac{1}{2}(\partial_t g^{00})(\partial_t v)^2 \notag \\
	& \ \ + \bigg(\frac{1}{2} \partial_t g^{ab} + H g^{ab} \bigg) (\partial_a v) (\partial_b v)
		- \upgamma H (\partial_t g^{00}) v \partial_t v - \upgamma H (g^{00} + 1)(\partial_t v)^2. \notag
\end{align}

\end{lemma}

\begin{proof}
	A proof based on a standard integration by parts argument (multiply both sides of equation \eqref{E:vscalar} by
	$- (\partial_t v + \upgamma H v)$ before integrating by parts over $\mathbb{T}^3$)
	is given in Lemma 15 of \cite{hR2008}. In particular, we quote the following identity:	
	
	\begin{align} \label{E:mathcalEtimederivativeformula}
		\frac{d}{dt} (\mathcal{E}_{(\upgamma,\updelta)}^2[v,\partial v]) & = \int_{\mathbb{T}^3} 
		\Big\lbrace -(\upalpha - \upgamma) H (\partial_t v)^2 + (\updelta - \upbeta - \upgamma \upalpha) H^2 v \partial_t v 
		- \upbeta \upgamma H^3 v^2 \\ 
	& \ \ \ \ \ - (1 + \upgamma) H g^{ab}(\partial_a v)(\partial_b v) - (\partial_t v + \upgamma H v)F 
		+ \triangle_{\mathcal{E};(\upgamma, \updelta)}[v,\partial v] \Big\rbrace \, d^3 x. \notag
	\end{align}
\end{proof}

\subsubsection{Energies for the components of $g$}

\begin{definition}  \label{D:energiesforg}
We define the \emph{non-negative} energies $\gzerozeroenergy{N}(t),$ 
$\gzerostarenergy{N}(t),$ $\hstarstarenergy{N}(t),$ and $\genergy{N}(t)$ as follows:

 \begin{subequations}
\begin{align}
	\gzerozeroenergy{N}^2 & \eqdef \sum_{|\vec{\alpha}| \leq N}
		e^{2q \Omega} \mathcal{E}_{(\upgamma_{00},\updelta_{00})}^2[\partial_{\vec{\alpha}} (g_{00} + 1),
		\partial (\partial_{\vec{\alpha}} g_{00})], \label{E:g00energydef} \\
	\gzerostarenergy{N}^2 & \eqdef \sum_{|\vec{\alpha}| \leq N} \sum_{j=1}^3 
		e^{2(q-1) \Omega}\mathcal{E}_{(\upgamma_{0*},\updelta_{0*})}^2[\partial_{\vec{\alpha}} g_{0j},
			\partial (\partial_{\vec{\alpha}} g_{0j})], \label{E:g0*energydef} \\
	\hstarstarenergy{N}^2 & \eqdef \sum_{|\vec{\alpha}| \leq N} 
		\Big\lbrace \sum_{j,k=1}^3 e^{2q \Omega} 
		\mathcal{E}_{(0,0)}^2[0,\partial (\partial_{\vec{\alpha}} h_{jk})] 
		+ \frac{1}{2} \int_{\mathbb{T}^3} c_{\vec{\alpha}} H^2 \big(\partial_{\vec{\alpha}} h_{jk} \big)^2 \, d^3 x 
		\Big\rbrace,	 \label{E:h**energydef} \\
	\genergy{N} & \eqdef \gzerozeroenergy{N} + \gzerostarenergy{N} + 	
		\hstarstarenergy{N}, \label{E:gtotalenergydef}
\end{align}
\end{subequations}
where

\begin{subequations}
\begin{align}
	h_{jk} & \eqdef e^{-2 \Omega} g_{jk}, && (j,k = 1,2,3), \\
	c_{\vec{\alpha}} & \eqdef 0, && \mbox{if} \ |\vec{\alpha}| = 0, \\
	c_{\vec{\alpha}} & \eqdef 1, && \mbox{if} \ |\vec{\alpha}| > 0, 
\end{align}
\end{subequations}
and $(\upgamma_{00}, \updelta_{00}), (\upgamma_{0*}, \updelta_{0*}),$ and $(\upgamma_{**}, \updelta_{**})=(0,0)$ are the constants generated by applying Lemma \ref{L:buildingblockmetricenergy} to equations \eqref{E:finalg00equation} - \eqref{E:finalhjkequation} respectively.

\end{definition}

In the next lemma, we provide a preliminary estimate of the time derivative of these energies.

\begin{lemma} \cite[Lemma 6.2.2]{iRjS2009} \label{L:metricfirstdiferentialenergyinequality}
		Assume that $g_{\mu \nu},$ $(\mu,\nu=0,1,2,3),$ is a solution to the modified equations 
	\eqref{E:finalg00equation} - \eqref{E:finalhjkequation}, and let $\gzerozeroenergy{N},$ $\gzerostarenergy{N},$ and
	$\hstarstarenergy{N}$ be as in Definition \ref{D:energiesforg}. Let 
	$[\hat{\Square}_g ,\partial_{\vec{\alpha}}]$ denote the commutator of the operators $\hat{\Square}_g$ and
	$\partial_{\vec{\alpha}}.$ Then under the assumptions of Lemma \ref{L:buildingblockmetricenergy},
	the following differential inequalities are satisfied, where 
	$\triangle_{\mathcal{E};(\gamma, \delta)}[\cdot, \partial(\cdot )]$ is defined in \eqref{E:trianglemathscrEdef}, and the
	constants $(\upgamma_{00}, \updelta_{00}), (\upgamma_{0*}, \updelta_{0*}),$ and $(\upgamma_{**}, \updelta_{**})=(0,0)$ 
	are defined in Definition \ref{D:energiesforg}:
	
	\begin{subequations}
	\begin{align}
		\frac{d}{dt}(\gzerozeroenergy{N}^2) & \leq (2q - \upeta_{00})H\gzerozeroenergy{N}^2 
			+ 2q(\omega - H) \gzerozeroenergy{N}^2 \\
		& \ \ - \sum_{|\vec{\alpha}| \leq N } \int_{\mathbb{T}^3} e^{2q \Omega} \big\lbrace \partial_t 
			\partial_{\vec{\alpha}}(g_{00}+1)
		  + \upgamma_{00} H \partial_{\vec{\alpha}}(g_{00}+1) \big\rbrace \notag \\
		 & \hspace{1in} \times \big\lbrace \partial_{\vec{\alpha}} \triangle_{00} 
		 	+ [\hat{\Square}_g ,\partial_{\vec{\alpha}}](g_{00}+1) \big\rbrace \, d^3 x \notag \\
		& \ \ + \sum_{|\vec{\alpha}| \leq N } \int_{\mathbb{T}^3} 
			e^{2q \Omega} \triangle_{\mathcal{E};(\upgamma_{00}, \updelta_{00})}[\partial_{\vec{\alpha}}(g_{00}+1)
				,\partial (\partial_{\vec{\alpha}} g_{00})]   \,d^3 x, \notag 
	\end{align}
	
	\begin{align}
		\frac{d}{dt}(\gzerostarenergy{N}^2) & \leq [2(q-1) - \upeta_{0*}]H\gzerostarenergy{N}^2 
			+ 2(q-1)(\omega - H) \gzerostarenergy{N}^2 \\
		& \ \ - \sum_{|\vec{\alpha}| \leq N } \sum_{j=1}^3 \int_{\mathbb{T}^3} e^{2(q-1) \Omega} \big\lbrace \partial_t 
			\partial_{\vec{\alpha}} g_{0j}
		  + \upgamma_{0*} H \partial_{\vec{\alpha}} g_{0j} \big\rbrace \notag \\
		 & \hspace{1in} \times \big\lbrace -2H \partial_{\vec{\alpha}} (g^{ab}\Gamma_{a j b})
			+ \partial_{\vec{\alpha}} \triangle_{0j} + [\hat{\Square}_g ,\partial_{\vec{\alpha}}] g_{0j} 
			\big\rbrace \, d^3 x \notag \\
		& \ \ + \sum_{|\vec{\alpha}| \leq N } \sum_{j=1}^3 \int_{\mathbb{T}^3} 
			e^{2(q-1)\Omega} \triangle_{\mathcal{E};(\upgamma_{0*}, \updelta_{0*})}[\partial_{\vec{\alpha}} g_{0j},
			\partial( \partial_{\vec{\alpha}} g_{0j})]   \,d^3 x, \notag 
	\end{align}
	
	\begin{align}			
		\frac{d}{dt}(\hstarstarenergy{N}^2) & \leq (2q - \upeta_{**})H\hstarstarenergy{N}^2 
			+ 2q(\omega - H) \hstarstarenergy{N}^2 \\
		& \ \ - \sum_{|\vec{\alpha}| \leq N } \sum_{j,k=1}^3 \int_{\mathbb{T}^3} e^{2q \Omega} \big\lbrace \partial_t 
			\partial_{\vec{\alpha}} h_{jk}
		  + \underbrace{\upgamma_{**}}_0 H \partial_{\vec{\alpha}} h_{jk} \big\rbrace \big\lbrace \partial_{\vec{\alpha}} 
		  	\triangle_{jk} 
		  + [\hat{\Square}_g ,\partial_{\vec{\alpha}}] h_{jk} \big\rbrace \, d^3 x \notag \\
		& \ \ + \sum_{|\vec{\alpha}| \leq N } \sum_{j,k=1}^3 \int_{\mathbb{T}^3} 
			e^{2q \Omega} \triangle_{\mathcal{E};(0, 0)}[0,\partial (\partial_{\vec{\alpha}} h_{jk})]  \,d^3 x  \notag \\ 
		& \ \ + \sum_{1 \leq |\vec{\alpha}| \leq N } \int_{\mathbb{T}^3}  
			H^2 (\partial_{\vec{\alpha}} \partial_t h_{jk})(\partial_{\vec{\alpha}} h_{jk}) \, d^3 x. \notag
	\end{align}	
	\end{subequations}
\end{lemma}

\begin{proof}
	Lemma \ref{L:metricfirstdiferentialenergyinequality} follows easily from definitions \eqref{E:g00energydef} - \eqref{E:gtotalenergydef}, and from \eqref{E:mathcalEtimederivativebound}.
\end{proof}

The following corollary follows easily from Lemma \ref{L:metricfirstdiferentialenergyinequality}, 
definitions \eqref{E:g00energydef} - \eqref{E:h**energydef}, and the Cauchy-Schwarz inequality for integrals.

\begin{corollary} \cite[Corollary 6.2.3]{iRjS2009} \label{C:metricfirstdiferentialenergyinequality}
	Under the assumptions of Lemma \ref{L:metricfirstdiferentialenergyinequality}, we have that
	
	\begin{subequations}
	\begin{align}
		\frac{d}{dt}(\gzerozeroenergy{N}^2) & \leq (2q - \upeta_{00})H\gzerozeroenergy{N}^2 
			+ 2q(\omega - H) \gzerozeroenergy{N}^2 
			+ \gzerozeronorm{N} e^{q \Omega}  \| \triangle_{00} \|_{H^N} \\
		& \ \ + \gzerozeronorm{N} \sum_{|\vec{\alpha}| \leq N} 
		 	e^{q \Omega} \| [\hat{\Square}_g ,\partial_{\vec{\alpha}}](g_{00}+1) \|_{L^2} \notag \\
		& \ \ + \sum_{|\vec{\alpha}| \leq N } e^{2q \Omega} \| \triangle_{\mathcal{E};(\upgamma_{00}, 
			\updelta_{00})}[\partial_{\vec{\alpha}}(g_{00}+1),\partial (\partial_{\vec{\alpha}} g_{00})] \|_{L^1}, \notag 
	\end{align}
	
	\begin{align}
	 	\frac{d}{dt}(\gzerostarenergy{N}^2) & \leq [2(q-1) - \upeta_{0*}]H\gzerostarenergy{N}^2 
			+ 2(q-1)(\omega - H) \gzerostarenergy{N}^2 	\label{E:underlinemathfrakEg0*firstdifferential} \\
		& \ \ + 2H \gzerostarnorm{N} \sum_{j=1}^3 e^{(q-1) \Omega} \| g^{ab} \Gamma_{ajb} \|_{H^N}
		  + \gzerostarnorm{N} \sum_{j=1}^3 e^{(q-1) \Omega} \| \triangle_{0j} \|_{H^N}
			\notag \\
		&	\ \ + \gzerostarnorm{N} 
			\sum_{|\vec{\alpha}| \leq N } \sum_{j=1}^3 e^{(q-1) \Omega} \| [\hat{\Square}_g ,\partial_{\vec{\alpha}}] g_{0j} \|_{L^2}
			\notag \\
		& \ \ + \sum_{|\vec{\alpha}| \leq N } \sum_{j=1}^3 e^{2(q-1) \Omega} 
			\| \triangle_{\mathcal{E};(\upgamma_{0*}, \updelta_{0*})}[\partial_{\vec{\alpha}} g_{0j}
			,\partial(\partial_{\vec{\alpha}} g_{0j})] \|_{L^1}, \notag 
	\end{align}		

	\begin{align}			
		\frac{d}{dt}(\hstarstarenergy{N}^2) & \leq (2q - \upeta_{**})H\hstarstarenergy{N}^2 
			+ 2q(\omega - H) \hstarstarenergy{N}^2 + \hstarstarnorm{N} \sum_{j,k=1}^3 \| \triangle_{jk} \|_{H^N} \\
		& \ \ + \hstarstarnorm{N} \sum_{|\vec{\alpha}| \leq N } \sum_{j,k=1}^3 
				\| [\hat{\Square}_g ,\partial_{\vec{\alpha}}] h_{jk} \|_{L^2} \notag \\
		& \ \ + \sum_{|\vec{\alpha}| \leq N } \sum_{j,k=1}^3 
			\| \triangle_{\mathcal{E};(0, 0)}[0,\partial(\partial_{\vec{\alpha}} h_{jk}) \|_{L^1}
			+ H^2 e^{-q \Omega}\hstarstarnorm{N}^2, \notag
	\end{align}	
	\end{subequations}
	where the norms $\gzerozeronorm{N},$ $\gzerostarnorm{N},$ $\hstarstarnorm{N}$ are defined in Definition \ref{D:Norms}.
\end{corollary}

\subsection{The equations of variation and energies for \texorpdfstring{$P - \bar{p}, u^j$}{the fluid variables}} \label{SS:fluidEnergies}
In this section, we define energies that are useful for studying the Euler equations \eqref{E:finalfirstEuler} - 
\eqref{E:finalEulerj}. Now in order to estimate the spatial derivatives of the fluid variables, we will of course have to differentiate the equations \eqref{E:finalfirstEuler} - \eqref{E:finalEulerj}, or equivalently, equation \eqref{E:matrixvectorfluidequations}. The derivatives are solutions to the linearization of \eqref{E:finalfirstEuler} - \eqref{E:finalEulerj} around the background variables $(P,u^1,u^2,u^3).$ We refer to the linearized system as the \emph{equations of variation}, while the unknowns $(\dot{P}, \dot{u}^1, \dot{u}^2, \dot{u}^3)$ are the \emph{variations}. More specifically, the equations of variation in the unknowns $(\dot{P}, \dot{u}^1, \dot{u}^2, \dot{u}^3)$ corresponding to the background $(P,u^1,u^2,u^3)$ are defined by

\begin{subequations}
\begin{align}
	u^{\alpha} \partial_{\alpha} \dot{P} + (1 + \speed^2) \Big(\frac{-1}{u_0}\Big) P u_a \partial_t \dot{u}^a + (1 + \speed^2) P 
		\partial_a \dot{u}^a & = \mathfrak{F}, \label{E:EOV1} \\
	u^{\alpha} \partial_{\alpha} \dot{u}^j + \frac{\speed^2}{(1 + \speed^2)P}\Pi^{j \alpha} \partial_{\alpha} \dot{P}
		& = (\decayparameter - 2)\omega \dot{u}^j + \mathfrak{G}^j, \label{E:EOV2}
\end{align}
\end{subequations}
where $u^0 > 0$ is such that $g_{\alpha \beta} u^{\alpha} u^{\beta} = -1.$ In our applications below $(\dot{P}, \dot{u}^1, \dot{u}^2, \dot{u}^3)$ will be equal to $\big(\partial_{\vec{\alpha}}(P - \bar{p}),\partial_{\vec{\alpha}} u^1, \partial_{\vec{\alpha}} u^2, \partial_{\vec{\alpha}} u^3 \big),$ where $\vec{\alpha}$ is a spatial derivative multi-index. The terms $\mathfrak{F}, (\decayparameter - 2)\omega \dot{u}^j,$  $\mathfrak{G}^j$ denote the inhomogeneous terms that arise from differentiating \eqref{E:finalfirstEuler} - \eqref{E:finalEulerj}. Note that we have split the inhomogeneous term in \eqref{E:EOV2} into two pieces; the $(\decayparameter - 2)\omega \dot{u}^j$ piece is responsible for creating decay in the $\dot{u}^j$ variable, while in our applications, $\mathfrak{G}^j$ will be an error term.

Using matrix-vector notation, we can abbreviate the equations \eqref{E:EOV1}- \eqref{E:EOV2} as  

\begin{align} \label{E:EOVmatrixvector}
	A^{\beta} \partial_{\beta} \dot{\mathbf{W}} = \mathbf{I},
\end{align}
where the $A^{\mu}$ are defined in \eqref{E:A0def} - \eqref{E:A1def}, and

\begin{align}
	\dot{\mathbf{W}} & \eqdef (\dot{P}, \dot{u}^1, \dot{u}^2, \dot{u}^3)^T, \\
	\mathbf{I} & \eqdef (\mathfrak{F}, \mathfrak{G}^1, \mathfrak{G}^2, \mathfrak{G}^3)^T.
\end{align}

To each variation $(\dot{u}^1, \dot{u}^2, \dot{u}^3),$ we associate a quantity $\dot{u}^0$ defined by

\begin{align} \label{E:dot0intermsofdotj}
	\dot{u}^0 \eqdef  -\frac{1}{u_0}u_a \dot{u}^a.
\end{align}
This quantity appears below in the expression \eqref{E:EnergyCurrent}, which defines our fluid energy current. In our analysis, we will need the following lemma, which essentially states that $\dot{u}^0$ is a solution to the linearization of \eqref{E:u0equation} around the background $(P,u^1,u^2,u^3).$

\begin{lemma} \label{L:dot0equation}
Assume that $(\dot{P},\dot{u}^1, \dot{u}^2, \dot{u}^3)$ is a solution to the equations of variation \eqref{E:EOV1} - \eqref{E:EOV2} corresponding to the background $(P,u^1,u^2,u^3),$ and let $\dot{u}^0 \eqdef  -\frac{1}{u_0}u_a \dot{u}^a$ be as defined in \eqref{E:dot0intermsofdotj}. Then $\dot{u}^0$ is a solution to the following equation:

\begin{align}
	u^{\alpha} \partial_{\alpha} \dot{u}^0 + \frac{\speed^2}{(1 + \speed^2)P}\Pi^{0 \alpha} \partial_{\alpha} \dot{P}
		& = \mathfrak{G}^0, \label{E:dotu0equation} 
\end{align}
where

\begin{subequations}
\begin{align}
	\Pi^{0 \mu} & = u^0 u^{\mu} + g^{0 \mu},&& (\mu = 0,1,2,3), \\
	\mathfrak{G}^0 & = -\bigg[u^{\alpha} \partial_{\alpha} \Big(\frac{u_a}{u_0}\Big)\bigg] \dot{u}^a 
		- (\decayparameter - 2)\omega\Big(\frac{1}{u_0}\Big)u_a \dot{u}^a - \Big(\frac{1}{u_0}\Big)u_a \mathfrak{G}^a.
		\label{E:dotu0G0def}
\end{align}
\end{subequations}
\end{lemma}

\begin{proof}
	By the definition of $\dot{u}^0,$ the left-hand side of \eqref{E:dotu0equation} is equal to
	
	\begin{align} \label{E:dotu0equationidentity1}
		-\frac{u_a}{u_0} u^{\alpha} \partial_{\alpha} \dot{u}^a + \frac{\speed^2}{(1 + \speed^2)P}\Pi^{0 \alpha} \partial_{\alpha} \dot{P}
		-\bigg[u^{\alpha} \partial_{\alpha} \Big(\frac{u_a}{u_0}\Big)\bigg] \dot{u}^a.
	\end{align}
	
	On the other hand, contracting equation \eqref{E:EOV2} against $u_j$ and using the identity 
	$u_a \Pi^{a \alpha} = - u_0 \Pi^{0 \alpha},$ we conclude that
	
	\begin{align} \label{E:dotu0equationidentity2}
		u_a u^{\alpha} \partial_{\alpha} \dot{u}^a - u_0 \frac{\speed^2}{(1 + \speed^2)P}\Pi^{0 \alpha} \partial_{\alpha} \dot{P}
		= (\decayparameter - 2)\omega u_a \dot{u}^a + u_a \mathfrak{G}^a.
	\end{align}
	Multiplying \eqref{E:dotu0equationidentity2} by $-\frac{1}{u_0}$ and using \eqref{E:dotu0equationidentity1},
	we arrive at \eqref{E:dotu0equation}.
\end{proof}

In the next lemma, we provide detailed information about the structure of the inhomogeneous terms appearing in the equations of variation and in Lemma \ref{L:dot0equation}. We again split the terms into two pieces to facilitate our analysis in later sections.

\begin{lemma}
Let $\mathbf{W} \eqdef (P - \bar{p}, u^1,u^2,u^3)^T$ be a solution to the relativistic Euler equations \eqref{E:matrixvectorfluidequations}. Then $(\dot{P},\dot{u}^1,\dot{u}^2,\dot{u}^3)^T \eqdef$ 
$(\partial_{\vec{\alpha}} (P - \bar{p}), \partial_{\vec{\alpha}} u^1, \partial_{\vec{\alpha}}u^2, \partial_{\vec{\alpha}}u^3)^T$ is a solution to the equations of variation \eqref{E:EOVmatrixvector} with inhomogeneous term $\mathbf{I}$ that can be expressed as follows:

\begin{align} \label{E:Ialphadecomp}
	\mathbf{I} \eqdef \partial_{\vec{\alpha}} \mathbf{b} + \mathbf{b}_{\triangle \vec{\alpha}},
\end{align}
where

\begin{subequations}
\begin{align} \label{E:bdef}
	\mathbf{b} & \eqdef \big(0, (\decayparameter - 2)u^1, (\decayparameter - 2)u^2, (\decayparameter - 2)u^3 \big)^T, 
					 \\ 
	\mathbf{b}_{\triangle \vec{\alpha}} & \eqdef \big(\mathfrak{F}_{\vec{\alpha}}, \mathfrak{G}_{\vec{\alpha}}^1, 
		\mathfrak{G}_{\vec{\alpha}}^2, \mathfrak{G}_{\vec{\alpha}}^3 \big)^T
		= \Big\lbrace A^0 \partial_{\vec{\alpha}} \big[ (A^0)^{-1}\mathbf{b} \big] 
		- \partial_{\vec{\alpha}} \mathbf{b} \Big\rbrace +
		A^0 \partial_{\vec{\alpha}} \big[ (A^0)^{-1} \mathbf{b}_{\triangle} \big] 
			\label{E:FalphamathfrakGjalphainhomogeneousterms}  \\
		& \ \ + A^0 \Big\lbrace(A^0)^{-1}A^a \partial_a \partial_{\vec{\alpha}} \mathbf{W}
 			- \partial_{\vec{\alpha}} \big[(A^0)^{-1} A^a \partial_a \mathbf{W}\big] \Big\rbrace. \notag
\end{align}

\end{subequations}
Furthermore, $\dot{u}^0 \eqdef -\frac{1}{u_0}u_a \dot{u}^a$ is a solution to equation 
\eqref{E:dotu0equation} with inhomogeneous term $\mathfrak{G}_{\vec{\alpha}}^0$ defined by

\begin{align}
 	\mathfrak{G}_{\vec{\alpha}}^0 & = -\bigg[u^{\nu} \partial_{\nu} \Big(\frac{u_a}{u_0}\Big)\bigg] 
 		\partial_{\vec{\alpha}} u^a 
		- (\decayparameter - 2)\omega\Big(\frac{1}{u_0}\Big)u_a \partial_{\vec{\alpha}} u^a
		- \Big(\frac{1}{u_0}\Big)u_a \mathfrak{G}_{\vec{\alpha}}^a. \label{E:mathfrakG0alphainhomogeneousterm}
\end{align}

\end{lemma}

\begin{proof}
	Equations \eqref{E:Ialphadecomp} - \eqref{E:FalphamathfrakGjalphainhomogeneousterms} 
	are a straightforward decomposition of the inhomogeneous term $\mathbf{I} = A^0 \partial_{\vec{\alpha}} \big\lbrace 
	(A^0)^{-1} A^{\mu} \partial_{\mu} \mathbf{W} \big\rbrace + A^0 [(A^0)^{-1}A^{\mu} \partial_{\mu}, \partial_{\vec{\alpha}}] 
	\mathbf{W}$ \\
	$= A^0 \partial_{\vec{\alpha}} \big\lbrace (A^0)^{-1} (\mathbf{b} + \mathbf{b}_{\triangle}) \big\rbrace
	+ A^0 [(A^0)^{-1}A^a \partial_a, \partial_{\vec{\alpha}}] \mathbf{W},$ where $[\cdot,\cdot]$
	denotes the commutator. The relation \eqref{E:mathfrakG0alphainhomogeneousterm} follows directly from \eqref{E:dotu0G0def}.
	
\end{proof}

\subsubsection{The fluid energy currents}

To each variation $\dot{\mathbf{W}}=(\dot{P},\dot{u}^1,\dot{u}^2,\dot{u}^3)^T,$ we associate the
following \emph{energy current}, where $\dot{u}^0 \eqdef -\frac{1}{u_0}u_a \dot{u}^a:$ 

\begin{align} \label{E:EnergyCurrent}
	\dot{J}^{\mu} & \eqdef \frac{u^{\mu}}{(1 + \speed^2)P}\dot{P}^2 + 2 \dot{u}^{\mu} \dot{P} 
		+ \frac{(1 + \speed^2)P u^{\mu}}{\speed^2} g_{\alpha \beta} \dot{u}^{\alpha} \dot{u}^{\beta}.
\end{align}
Currents of the form \eqref{E:EnergyCurrent} are the building blocks for some of our fluid energies (see \eqref{E:fluidenergydef}). We remark that similar currents were used in \cite{dC2007}, \cite{jS2008b}, \cite{jS2008a}.

\begin{remark}
	We sometimes write $\dot{J}^{\mu}[\dot{\mathbf{W}}, \dot{\mathbf{W}}]$ to emphasize that
	$\dot{J}^{\mu}$ depends quadratically on the variations.
\end{remark}

\subsubsection{The divergence of the fluid energy current}

Let $\dot{\mathbf{W}}=(\dot{P},\dot{u}^1,\dot{u}^2,\dot{u}^3)^T$ be a solution to the equations of variation 
\eqref{E:EOV1} - \eqref{E:EOV2}. Then using the equations \eqref{E:EOV1} - \eqref{E:EOV2} and \eqref{E:dotu0equation} to replace derivatives of $\dot{\mathbf{W}}$ with the inhomogeneous terms, an omitted, tedious computation gives that the divergence of $\dot{J}^{\mu}$ can be expressed as follows:

\begin{align} \label{E:divergenceofdotJ}
	\partial_{\mu} \big( \dot{J}^{\mu}[\dot{\mathbf{W}}, \dot{\mathbf{W}}] \big) 
		& = \bigg(\partial_t \Big[\frac{u^0}{(1 + \speed^2)P} \Big] \bigg)\dot{P}^2 
		+	\bigg(\partial_a \Big[\frac{u^a}{(1 + \speed^2)P} \Big] \bigg)\dot{P}^2  \\
	& \ \ + \bigg( \partial_t \Big[\frac{(1 + \speed^2)P u^0}{\speed^2} \Big] \bigg) 
		\bigg(g_{00} (\dot{u}^{0})^2 + 2 g_{0a} \dot{u}^{0} \dot{u}^{a} 
		+ g_{ab} \dot{u}^{a} \dot{u}^{b} \bigg)		\notag \\
	& \ \ + \bigg( \partial_a \Big[\frac{(1 + \speed^2)P u^a}{\speed^2} \Big] \bigg) 
		\bigg(g_{00} (\dot{u}^{0})^2 + 2 g_{0a} \dot{u}^{0} \dot{u}^{a} 
		+ g_{ab} \dot{u}^{a} \dot{u}^{b} \bigg) \notag  \\
	& \ \ + \frac{(1 + \speed^2)P u^a}{\speed^2} (\partial_a g_{00}) (\dot{u}^{0})^2
		+ 2 \frac{(1 + \speed^2)P u^a}{\speed^2} (\partial_a g_{0b}) \dot{u}^{0} \dot{u}^{b} 
		\notag \\
	& \ \ + \frac{(1 + \speed^2)P u^a}{\speed^2} (\partial_a g_{lm}) \dot{u}^{l} \dot{u}^{m}
		+ \frac{(1 + \speed^2)P u^0}{\speed^2} (\partial_t g_{00}) (\dot{u}^0) ^2 \notag \\
	& \ \ + \frac{2(1 + \speed^2)P u^0}{\speed^2} (\partial_t g_{0 a}) \dot{u}^0 \dot{u}^a
		+ \frac{(1 + \speed^2)P (u^0 - 1)}{\speed^2} (\partial_t g_{ab}) \dot{u}^a \dot{u}^b \notag \\
	& \ \ - 2 \bigg(\partial_t \Big[ \frac{u_a}{u_0} \Big] \bigg)\dot{u}^a \dot{P}
		+ \frac{(1+\speed^2)P}{\speed^2} (\partial_t g_{ab} - 2 \omega g_{ab})\dot{u}^a \dot{u}^b 
		 \notag \\
	& \ \ + \underbrace{\frac{2(1+\speed^2)P}{\speed^2}(\decayparameter - 1 ) \omega g_{ab}\dot{u}^a \dot{u}^b}_{< 0} 
		+ \frac{2\mathfrak{F}}{(1 + \speed^2)P} \dot{P} \notag \\ 
	& \ \ + \frac{2(1+\speed^2)P}{\speed^2} g_{00} \mathfrak{G}^0 \dot{u}^0
		+ \frac{4(1+\speed^2)P}{\speed^2} g_{0a} \mathfrak{G}^0 \dot{u}^a 
		+ \frac{2(1+\speed^2)P}{\speed^2} g_{ab} \mathfrak{G}^a \dot{u}^b. \notag
\end{align}
\emph{Note that the right-hand side of \eqref{E:divergenceofdotJ} does not depend on the derivatives of $\dot{\mathbf{W}};$}
this property is essential for closing our energy estimates for the top derivatives of the fluid variables. We remark that we have organized the terms on the terms on the right-hand side of \eqref{E:divergenceofdotJ} in a way that will be useful for proving inequality \eqref{E:divergenceofdotJL1estimate}.

\subsubsection{The definition of $\fluidenergy{N}$}

In our analysis, we will make use of several different norms and energies for the fluid. The energy $\fluidenergy{N}$ 
will control all spatial derivatives of all of $\mathbf{W} \eqdef (P - \bar{p},u^1,u^2,u^3)^T,$ while the norm $U_{N-1}$ defined in \eqref{E:ujnorm} will control the lower derivatives of $(u^1,u^2,u^3)$ with \emph{larger} weights. These larger weights lead to better $L^{\infty}$ decay for the lower-order derivatives of $(u^1,u^2,u^3)$ than for the top order derivatives; this improved decay plays an essential role in our analysis. We now proceed to the definition of the \emph{non-negative} (see inequality \eqref{E:ENSNcomparison}) quantity $\fluidenergy{N}.$ 

\begin{definition} \label{D:fluidenergydef}
	Let $N$ be a positive integer, and let $\mathbf{W} \eqdef (P - \bar{p},u^1,u^2,u^3)^T$ be the array of fluid variables. We 
	define the fluid energy $\fluidenergy{N}(t) \geq 0$ by
	
	\begin{align} \label{E:fluidenergydef}
		\fluidenergy{N}^2 & \eqdef \sum_{|\vec{\alpha}| \leq N} \int_{\mathbb{T}^3} \dot{J}^0[\partial_{\vec{\alpha}} 
			\mathbf{W},\partial_{\vec{\alpha}} \mathbf{W}] \, d^3 x,
	\end{align}
where $\dot{J}^0[\partial_{\vec{\alpha}} \mathbf{W},\partial_{\vec{\alpha}} \mathbf{W}]$ is defined
in \eqref{E:EnergyCurrent}.

\end{definition}

In the next corollary, we provide a preliminary estimate for the time derivatives of $U_{N-1}$ and 
$\fluidenergy{N}.$

\begin{corollary} \label{C:fluidenergytimederivative}
	Let $\mathbf{W} \eqdef (P - \bar{p}, u^1,u^2,u^3)^T$ be a solution to the relativistic Euler equations 
	\eqref{E:matrixvectorfluidequations}.	Let $U_{N-1}(t)$ and $\fluidenergy{N}(t)$ be the fluid norm and
	energy defined in \eqref{E:ujnorm} and \eqref{E:fluidenergydef} respectively. Then the following differential
	inequalities are satisfied:
	
	\begin{align} \label{E:Unormtimederivative}
		\frac{d}{dt}\big(U_{N-1}^2 \big) & \leq 2 \overbrace{(\decayparameter - 1 + q)}^{< 0}\omega  e^{(1 + q) \Omega} 
			U_{N-1}^2 + 2 e^{(1 + q) \Omega} U_{N-1} \sum_{a=1}^3 \| \triangle'^a \|_{H^{N-1}}, 
	\end{align}
	
	\begin{align} \label{E:fluidenergytimederivative}
		\frac{d}{dt}\big(\fluidenergy{N}^2 \big) & \leq 
			\sum_{|\vec{\alpha}| \leq N} \big\| \partial_{\mu} \big(\dot{J}^{\mu}[\partial_{\vec{\alpha}} 
			\mathbf{W},\partial_{\vec{\alpha}} \mathbf{W}] \big) \big\|_{L^1}.
	\end{align}

\end{corollary}

\begin{proof}
	To prove \eqref{E:Unormtimederivative}, we use the definition \eqref{E:ujnorm} of $U_{N-1}$
	and equation \eqref{E:partialtuj} (differentiated with $\partial_{\vec{\alpha}}$) to conclude that
	
	\begin{align} \label{E:partialtunderlineUNsquaredexpression}
		\frac{d}{dt}\big(U_{N-1}^2 \big) = 2(1 + q)\omega U_{N-1}^2
			+ 2 e^{2(1 + q)\Omega} \sum_{|\vec{\alpha}| \leq N} \sum_{a=1}^3 \int_{\mathbb{T}^3}
			(\partial_{\vec{\alpha}} u^a) \big\lbrace (\decayparameter  - 2)\omega \partial_{\vec{\alpha}}u^a 
				+ \partial_{\vec{\alpha}}\triangle'^j \big\rbrace \, d^3 x.
	\end{align}
	Inequality \eqref{E:Unormtimederivative} now follows from \eqref{E:partialtunderlineUNsquaredexpression},
	the definition of $U_{N-1},$ and the Cauchy-Schwarz inequality for integrals.
	
	To prove \eqref{E:fluidenergytimederivative}, we use the definition \eqref{E:fluidenergydef} of $\fluidenergy{N}^2$
	and the divergence theorem to conclude that
	
	\begin{align}
		\frac{d}{dt}\big(\fluidenergy{N}^2 \big) & = \sum_{|\vec{\alpha}| \leq N} \int_{\mathbb{T}^3} 
			\partial_t\big(\dot{J}^0 [\partial_{\vec{\alpha}} \mathbf{W},\partial_{\vec{\alpha}} 
			\mathbf{W}]\big) \, d^3 x \\
		& = \sum_{|\vec{\alpha}| \leq N} \int_{\mathbb{T}^3} 
			\partial_{\mu}\big(\dot{J}^{\mu} \big[\partial_{\vec{\alpha}} \mathbf{W}, \partial_{\vec{\alpha}} 
			\mathbf{W}]\big) \, d^3 x, \notag 
	\end{align}
	from which \eqref{E:fluidenergytimederivative} easily follows.

\end{proof}

\subsection{The total energy \texorpdfstring{$\totalenergy{N}$}{}} \label{SS:TotalEnergy}

\begin{definition}
Using definitions \eqref{E:ujnorm}, \eqref{E:gtotalenergydef}, and \eqref{E:fluidenergydef}, we define $\totalenergy{N},$ the total energy associated to $g,$ $P,$ $u,$ as follows:
	
	\begin{align} \label{E:totalenergy}
		\totalenergy{N} & \eqdef \genergy{N} + \fluidenergy{N} + U_{N-1}.
	\end{align}
\end{definition}

\section{Linear-Algebraic Estimates of \texorpdfstring{$g_{\mu \nu}$ and $g^{\mu \nu},$ $(\mu,\nu=0,1,2,3)$}{the metric}} \label{S:LinearAlgebra}
\setcounter{equation}{0}

In this section, we recall some linear-algebraic estimates of $g_{\mu \nu}$ and $g^{\mu \nu}$ that were proved by
Ringstr\"{o}m. In addition to providing some rough $L^{\infty}$ estimates that we will use in Sections \ref{S:BootstrapConsequences} and \ref{S:EnergyNormEquivalence}, the lemmas will guarantee that $g_{\mu \nu}$ is a Lorentzian metric. This latter fact has already been used in our statement of the conclusions of Theorem \ref{T:LocalExistence}.

\begin{lemma} \cite[Lemmas 1 and 2]{hR2008}\label{L:ginverseformluas}
	Let $g_{\mu \nu}$ be a symmetric $4 \times 4$ matrix of real numbers. 
	Let $(g_{\flat})_{jk}$ be the $3 \times 3$ matrix defined by $(g_{\flat})_{jk} = g_{jk},$ and let $(g_{\flat}^{-1})^{jk}$
	be the $3 \times 3$ inverse of $(g_{\flat})_{jk}.$ Assume that $g_{00} < 0$ and that $(g_{\flat})_{jk}$ is positive definite. 
	Then $g_{\mu \nu}$ is a Lorentzian metric with inverse $g^{\mu \nu},$ $g^{00} < 0,$ and the $3 \times 3$ matrix 
	$(g^{\#})^{jk}$ defined by $(g^{\#})^{jk} \eqdef g^{jk}$ is positive definite. Furthermore, the following relations hold:
	
	\begin{subequations}
	\begin{align}
		g^{00} & = \frac{1}{g_{00} - d^2}, \\
		\frac{g_{00}}{g_{00} - d^2} (g_{\flat}^{-1})^{ab} X_{a}X_{b} 
			& \leq (g^{\#})^{ab} X_{a}X_{b} \leq (g_{\flat}^{-1})^{ab} X_{a}X_{b}, && \forall (X_1,X_2,X_3) \in \mathbb{R}^3, \\
		g^{0j} & = \frac{1}{d^2 - g_{00}} (g_{\flat}^{-1})^{aj} g_{0a}, && (j=1,2,3),
	\end{align}
	\end{subequations}
	where 
	
	\begin{align}
		d^2 = (g_{\flat}^{-1})^{ab} g_{0a} g_{0b}.
	\end{align}
	
\end{lemma}

\hfill $\qed$

The estimates in the next lemma are based on the following rough assumptions, which we will upgrade during our global existence argument.

\begin{center}
	\large{Rough Bootstrap Assumptions for $g_{\mu \nu}:$}
\end{center}

We assume that there are constants $\upeta > 0$ and $c_1 \geq 1$ such that

\begin{subequations}
\begin{align}
	|g_{00} + 1| & \leq \upeta, && \label{E:metricBAeta} \\
	c_1^{-1} \delta_{ab} X^a X^b 
		& \leq e^{-2 \Omega} g_{ab}X^{a}X^{b} 
		\leq c_1 \delta_{ab} X^a X^b, && \forall (X^1,X^2,X^3) \in \mathbb{R}^3, \label{E:gjkBAvsstandardmetric}   \\
	\sum_{a=1}^3 |g_{0a}|^2 & \leq \upeta c_1^{-1} e^{2(1 - q) \Omega}. && \label{E:g0jBALinfinity}
\end{align}
\end{subequations}
For our global existence argument, we will assume that $\upeta = \upeta_{min},$ where $\upeta_{min}$ is defined in Section \ref{SS:etamin}.

\begin{lemma} \cite[Lemma 7]{hR2008} \label{L:ginverseestimates}
	Let $g_{\mu \nu}$ be a symmetric $4 \times 4$ matrix of real numbers satisfying \eqref{E:metricBAeta} - 
	\eqref{E:g0jBALinfinity}, where $\Omega \geq 0$ and $0 \leq q < 1.$ 
	Then $g$ is a Lorentzian metric, and there exists a constant $\upeta_0 > 0$ 
	such that $0 \leq \upeta \leq \upeta_0$ implies that the following estimates hold for the 
	its inverse $g^{\mu \nu}:$ 
	
	\begin{subequations}
	\begin{align}
		|g^{00} + 1| & \leq 4 \upeta, && \label{E:g00upperplusoneroughestimate} \\
		\sqrt{\sum_{a=1}^3 |g^{0a}|^2} & \leq \upeta c_1^{-1} e^{-2 \Omega} \sqrt{\sum_{a=1}^3 |g_{0a}|^2}, \\
		|g^{0a} g_{0a}| & \leq 2 c_1 e^{-2 \Omega} \sum_{a=1}^3 |g_{0a}|^2, && \\
		\frac{2}{3c_1} \delta^{ab}X_{a} X_{b}
		& \leq e^{2 \Omega} g^{ab}X_{a}X_{b} 
			\leq \frac{3c_1}{2}\delta^{ab}X_{a}X_{b}, && \forall (X_1,X_2,X_3) \in \mathbb{R}^3. 
				\label{E:gjkuppercomparetostandard}  
	\end{align}
	\end{subequations}
\end{lemma}

\hfill $\qed$

\section{The Bootstrap Assumption for \texorpdfstring{$\totalnorm{N}$}{the Solution Norm} and the Definition of $N, \upeta_{min},$ and $q$} \label{S:BootstrapAssumptions}

For the remainder of the article, $N$ denotes a fixed integer subject to the requirement

\begin{align}
	N &\geq 3 \qquad (\mbox{This is large enough for all of our results except some of the conclusions of 
		Theorem \ref{T:Asymptotics}}), \\ \label{E:Ndef}
	N &\geq 5 \qquad	(\mbox{This is large enough for the full results of Theorem \ref{T:Asymptotics} to be valid}).
\end{align}
We require $N$ to be of this size to ensure that various Sobolev embedding results and the conclusions of the propositions in Appendix \ref{A:SobolevMoser} are valid.

In our global existence argument, we will make the following bootstrap assumption:

\begin{align} \label{E:fundamentalbootstrap}
	\totalnorm{N} \leq \epsilon,
\end{align}
where $\totalnorm{N}$ is defined in \eqref{E:totalnorm}, and $\epsilon$ is a sufficiently 
small positive number. Observe that $\totalnorm{N}$ measures how much $(g, p, u)$ differs from the FLRW background solution
$(\widetilde{g}, \widetilde{p}, \widetilde{u})$ derived in Section \ref{S:backgroundsolution}. In particular, $\totalnorm{N} \equiv 0$ for the background solution.

\subsection{The definitions of \texorpdfstring{$\upeta_{min}$ and $q$}{}} \label{SS:etamin}

\begin{definition}
	Let us apply Lemma \ref{L:buildingblockmetricenergy} to each of the equations \eqref{E:finalg00equation} - 
	\eqref{E:finalhjkequation}, denoting the constant $\upeta$ produced by the lemma in each case by
	 $\upeta_{00}, \upeta_{0*},$ and $\upeta_{**}$ respectively. Furthermore, let $\upeta_0$ be the constant from
	 Lemma \ref{L:ginverseestimates}. We now define the positive quantities (recalling that $0 < \decayparameter < 1$ when
	 $0 < \speed < \sqrt{1/3}$) $\upeta_{min}$ and $q$ by
	 \begin{align}
	 	\upeta_{min} & \eqdef \frac{1}{8} \mbox{min} \big\lbrace 1, \upeta_0, \upeta_{00}, \upeta_{0*}, \upeta_{**} \big\rbrace, 
	 		\label{E:etamindef} \\
	 	q & \eqdef \frac{2}{3} \mbox{min}\big\lbrace \upeta_{min}, \decayparameter, 1 - \decayparameter \big\rbrace. 
	 		\label{E:qdef}
	 \end{align}
\end{definition}
The constants $\upeta_{min}$ and $q$ have been chosen to be small enough to close the bootstrap argument for global existence given in Section \ref{SS:globalexistencetheorem}. In particular, inequality \eqref{E:g00upperplusoneroughestimate}, with $\upeta \leq \upeta_{min},$ guarantees that the energies $\mathcal{E}_{(\upgamma,\updelta)}[\cdot,\partial (\cdot)]$ 
for solutions to \eqref{E:finalg00equation} - \eqref{E:finalhjkequation} have the coercive property \eqref{E:mathcalEfirstlowerbound}.

\begin{remark} \label{R:RoughBootstrapAutomatic}

By Sobolev embedding and Lemma \ref{L:ginverseestimates}, if $\epsilon$ is small enough, then the assumption $\totalnorm{N} \leq \epsilon$ implies that there exists a constant $C > 0$ such that

\begin{subequations}
\begin{align} \label{E:g00plusonesmall}
	|g_{00} + 1| & \leq C \epsilon, \\
	\sum_{j=1}^3 |g_{0j}| & \leq C \epsilon e^{(1 - q) \Omega}. \label{E:g0jremarkbound}
\end{align}
\end{subequations}

Therefore, if $\epsilon$ is sufficiently small, the inequalities \eqref{E:metricBAeta} and \eqref{E:g0jBALinfinity} 
are an automatic consequence of \eqref{E:gjkBAvsstandardmetric} and the assumption $\totalnorm{N} \leq \epsilon.$

\end{remark}

\section{Sobolev Estimates} \label{S:BootstrapConsequences}

In this section, we use the bootstrap assumptions of Sections \ref{S:LinearAlgebra} and \ref{S:BootstrapAssumptions} to deduce estimates of $g_{\mu \nu},$ $g^{\mu \nu},$ $P,$ $u^{\mu},$ and the nonlinear terms in the modified equations \eqref{E:finalg00equation} - \eqref{E:finalEulerj}, \eqref{E:u0equation}. The main goal is to show that the error terms are small compared to the principal terms, which is the main step in closing the bootstrap argument in the proof of Theorem \ref{T:GlobalExistence}. The primary tools for estimating these terms are standard Sobolev-Moser estimates, which we have collected together in the Appendix for convenience.

\subsection{Estimates of \texorpdfstring{$g_{\mu \nu},$ $g^{\mu \nu}$}{the metric components}} \label{SS:Bootstrapconsequencesg}

In this section, we state the first proposition that will be used to deduce the energy inequalities of Section
\ref{SS:IntegralInequalities}. The proposition was essentially proved in \cite[Proposition 9.1.1]{iRjS2009}, using the ideas
of \cite[Lemmas 9,11,18,20]{hR2008}. We don't bother to repeat the proof here, since similar arguments are used below in the proof of Proposition \ref{P:Nonlinearities}. We remark that the proof makes use of the lemmas stated in Section \ref{S:LinearAlgebra}.

\begin{proposition} \cite[Proposition 9.1.1]{iRjS2009} \label{P:BoostrapConsequences}
Let $N \geq 3$ be an integer, and assume that the bootstrap assumptions \eqref{E:metricBAeta} - \eqref{E:g0jBALinfinity} hold 
on the spacetime slab $[0,T) \times \mathbb{T}^3$ for some constant $c_1 \geq 1$ and for $\upeta = \upeta_{min}.$
Then there exists a constant $\epsilon' > 0$ and a constant $C > 0,$ where $C$ depends on $N, c_1,$ and $\upeta_{min},$ such that if $\totalnorm{N}(t) \leq \epsilon'$ on $[0,T),$ then the following estimates also hold on $[0,T),$ where $h_{jk} = e^{-2 \Omega} g_{jk}:$

\begin{subequations}
\begin{align}
	\| g_{00} \|_{L^{\infty}} & \leq 2, \label{E:g00lowerLinfinity} \\
	\| g_{0j} \|_{L^{\infty}} & \leq C e^{(1 - q) \Omega} \gnorm{N}, \label{E:g0jlowerLinfinity} \\
	\| g_{jk} \|_{L^{\infty}} & \leq C e^{2 \Omega},  \label{E:gjklowerLinfinity}
\end{align}
\end{subequations}

\begin{subequations}
\begin{align}
	\| \underpartial g^{00} \|_{H^{N-1}} & \leq C e^{-q \Omega}  \gnorm{N}, \label{E:partialg00upperHNminusone} \\
	\|g^{00} \|_{L^{\infty}} & \leq 5, \label{E:g00upperLinfinity} \\
	\| g^{00} + 1 \|_{H^N} & \leq C e^{-q \Omega} \gnorm{N}, \label{E:g00upperplusoneHN} \\
	\| g^{00} + 1 \|_{L^{\infty}} & \leq C e^{-q \Omega} \gnorm{N}, \label{E:g00upperplusoneLinfinity} \\
	\| g^{jk} \|_{L^{\infty}} & \leq C e^{-2 \Omega},  \label{E:gjkupperLinfinity} \\
	\| \underpartial g^{jk} \|_{H^{N-1}} & \leq C e^{-2 \Omega}  \gnorm{N}, \label{E:partialgjkupperHNminusone} \\
	\| \underpartial g^{jk} \|_{L^{\infty}} & \leq C e^{-2 \Omega}  \gnorm{N}, \label{E:partialgjkupperLinfinity} \\
	\| g^{0j} \|_{H^N} & \leq C e^{-(1 + q)\Omega}  \gnorm{N}, \label{E:g0jupperHN} \\
	\| g^{0j} \|_{C_b^1} & \leq C e^{-(1 + q)\Omega} \gnorm{N}, \label{E:g0jupperCb1} 
\end{align}
\end{subequations}

\begin{subequations}
\begin{align}
	\| \partial_t g_{jk} - 2 \omega g_{jk} \|_{H^N} & \leq C e^{(2 - q) \Omega} \hstarstarnorm{N}, 
		\label{E:partialtgjkminusomegagjklowerHN} \\
	\| \partial_t g_{jk} - 2 \omega g_{jk} \|_{C_b^1} & \leq C e^{(2 - q) \Omega} \hstarstarnorm{N}, 
		\label{E:partialtgjkminusomegagjklowerLinfinity} \\
	\| \partial_t g_{jk} \|_{C_b^1} & \leq C e^{2 \Omega}, \label{E:partialtgjklowerC1} 
\end{align}
\end{subequations}

\begin{subequations}
\begin{align}
	\| g^{aj} \partial_t g_{ak} - 2 \omega \delta_k^j \|_{H^N} & \leq C e^{-q \Omega}  \gnorm{N}, 
		\label{E:gjaupperpartialtgaklowerminus2omegadeltakjHN} \\
	\| g^{aj} \partial_t g_{ak} - 2 \omega \delta_k^j \|_{L^{\infty}} & \leq C e^{-q \Omega}  \gnorm{N}, 
		\label{E:gjaupperpartialtgaklowerminus2omegadeltakjLinfinity} 
\end{align}
\end{subequations}
	
\begin{subequations}
\begin{align}
	\| \partial_t g^{jk} + 2 \omega g^{jk} \|_{H^N} & \leq C e^{-(2+q)\Omega} \gnorm{N},  
		\label{E:partialtgjkupperplusomegagjkHN} \\
	\| \partial_t g^{jk} + 2 \omega g^{jk} \|_{L^{\infty}} & \leq C e^{-(2+q)\Omega} \gnorm{N},  
		\label{E:partialtgjkupperplusomegagjkLinfinity} \\
	\| \partial_t g^{00} \|_{L^{\infty}} & \leq C e^{-q \Omega} \gnorm{N},  
		\label{E:partialtg00upperLinfinity} \\
	\| \partial_t g^{0j} \|_{L^{\infty}} & \leq C e^{-(1 + q) \Omega} \gnorm{N}, 
		\label{E:partialtg0jupperLinfinity} \\
	\| \partial_t g^{jk} \|_{L^{\infty}} & \leq C e^{-2 \Omega},  
		\label{E:partialtgjkupperLinfinity} 
\end{align}
\end{subequations}

\begin{subequations}
\begin{align} 
	\| g^{ab}\Gamma_{a j b} \|_{H^N} & \leq C e^{(1-q)\Omega}\hstarstarnorm{N}, 
		\label{E:gabupperGammaajblowerHN} \\
	\| g^{ab}\Gamma_{a j b} \|_{H^{N-1}} & \leq C \hstarstarnorm{N}.
		\label{E:gabupperGammaajblowerHNminusone} 
\end{align}
\end{subequations}

In the above estimates, the norms $\hstarstarnorm{N}$ and $\gnorm{N}$ are defined in Definition \ref{D:Norms}.

\end{proposition}

\hfill $\qed$

\subsection{Estimates of \texorpdfstring{$P,$ $u^{\mu},$}{the fluid quantities} and the error terms} \label{SS:Nonlinearities}

In this section, we state and prove the second proposition that will be used to deduce the energy inequalities of Section
\ref{SS:IntegralInequalities}.

\begin{proposition} \label{P:Nonlinearities}
	
	Let $N \geq 3$ be an integer, and let $(g_{\mu \nu}, P, u^{\mu}),$ $(\mu,\nu = 0,1,2,3),$ be a solution to the reduced
	equations \eqref{E:finalg00equation} - \eqref{E:finalEulerj} on the spacetime slab $[0,T) \times \mathbb{T}^3$. Assume that 
	the bootstrap assumptions \eqref{E:metricBAeta} - \eqref{E:g0jBALinfinity} hold on the same slab for some constant $c_1 \geq 
	1$ and for $\upeta = \upeta_{min}.$ Then there exists a constant $\epsilon'' > 0$ and a constant $C > 0,$ where $C$ depends 
	on $N, c_1,$ and $\upeta_{min},$ such that if $\totalnorm{N}(t) \leq \epsilon''$ on $[0,T),$ then the following estimates 
	also hold on $[0,T)$ for the quantities $\triangle_{A,\mu \nu}, \triangle_{C,00},$ and $\triangle_{C,0j},$ defined in 
	\eqref{E:triangleA00def} - \eqref{E:triangleAjkdef} and \eqref{E:triangleC00def} - \eqref{E:triangleC0jdef}:
	
	\begin{subequations}
	\begin{align}
		\| \triangle_{A,00} \|_{H^N} & \leq C e^{-2q \Omega} \gnorm{N}^2, \label{E:triangleA00HN} \\
		\| \triangle_{A,0j} \|_{H^N} & \leq C e^{(1-2q) \Omega} \gnorm{N}^2, \label{E:triangleA0jHN} \\
		\| \triangle_{A,jk} \|_{H^N} & \leq C e^{(2-2q) \Omega} \gnorm{N}^2, \label{E:triangleAjkHN} \\
		\| \triangle_{C,00} \|_{H^N} & \leq C e^{-2q \Omega} \gnorm{N}^2, \label{E:triangleC00HN} \\
		\| \triangle_{C,0j} \|_{H^N} & \leq C e^{(1-2q) \Omega} \gnorm{N}^2. \label{E:triangleC0jHN}
	\end{align}
	\end{subequations}
	
	For the fluid quantities, we have the following estimates on $[0,T):$
	
	\begin{subequations}
	\begin{align} 
		\| P \|_{L^{\infty}} & \leq C, \label{E:PLinfinity} \\
		\| u^0 - 1 \|_{H^N} & \leq C e^{- q \Omega} \totalnorm{N}, \label{E:u0upperHN} \\
		\| u^0 - 1 \|_{L^{\infty}} & \leq C e^{- q \Omega} \totalnorm{N}, \label{E:u0upperLinfinity} \\
		\| u_0 + 1 \|_{H^N} & \leq C e^{- q \Omega} \totalnorm{N}, \label{E:u0lowerHN} \\
		\| u_0 + 1 \|_{L^{\infty}} & \leq C e^{- q \Omega} \totalnorm{N}, \label{E:u0lowerLinfinity} \\
		\| u_j \|_{H^{N-1}} & \leq C e^{(1 - q) \Omega} \totalnorm{N}, \label{E:ujlowerHNminusone} \\
		\| u_j \|_{L^{\infty}} & \leq C e^{(1 - q) \Omega} \totalnorm{N}, \label{E:ujlowerLinfinity} \\
		\| u_j \|_{H^N} & \leq C e^{\Omega} \totalnorm{N}. \label{E:ujlowerHN}
	\end{align}
	\end{subequations}

	For the time derivatives of the fluid quantities, we have the following estimates on $[0,T):$
	
	\begin{subequations}
	\begin{align}
		\| \partial_t P \|_{H^{N-1}} & \leq C e^{-q \Omega} \totalnorm{N}, \label{E:partialtPHNminusone} \\
		\| \partial_t P \|_{L^{\infty}} & \leq C e^{-q \Omega} \totalnorm{N}, \label{E:partialtPLinfinity} \\
		\| \partial_t u^0 \|_{H^{N-1}} & \leq C e^{-q \Omega} \totalnorm{N}, \label{E:partialtu0upperHNminusone} \\
		\| \partial_t u^0 \|_{L^{\infty}} & \leq C e^{-q \Omega} \totalnorm{N}, \label{E:partialtu0upperLinfinity} \\
		\| \partial_t u^j \|_{H^{N-1}} & \leq C e^{-(1 + q) \Omega} \totalnorm{N}, \label{E:partialtujupperHNminusone} \\
		\| \partial_t u^j \|_{L^{\infty}} & \leq C e^{-(1 + q) \Omega} \totalnorm{N}, \label{E:partialtujupperLinfinity} \\
		\| \partial_t u_0 \|_{H^{N-1}} & \leq C e^{-q \Omega} \totalnorm{N}, \label{E:partialtu0lowerHNminusone} \\
		\| \partial_t u_0 \|_{L^{\infty}} & \leq C e^{-q \Omega} \totalnorm{N}, \label{E:partialtu0lowerLinfinity} \\
		\| \partial_t u_j \|_{H^{N-1}} & \leq C e^{(1 - q) \Omega} \totalnorm{N}, \label{E:partialtujlowerHNminusone} \\
		\| \partial_t u_j \|_{L^{\infty}} & \leq C e^{(1 - q ) \Omega} \totalnorm{N}. \label{E:partialtujlowerLinfinity}
	\end{align}
	\end{subequations}
	
	For the quantities $\triangle_{\mu \nu}$ defined in \eqref{E:triangle00} - \eqref{E:trianglejk}, we have the
	following estimates on $[0,T):$
	
	\begin{subequations}
	\begin{align}
		\| \triangle_{00} \|_{H^N} & \leq C e^{-2q \Omega} \totalnorm{N}, \label{E:triangle00HN} \\
		\| \triangle_{0j} \|_{H^N} & \leq C  e^{(1 - 2q)\Omega} \totalnorm{N},  \label{E:triangle0jHN} \\
		\| \triangle_{jk} |_{H^N} & \leq C e^{-2q \Omega} \totalnorm{N}. \label{E:trianglejkHN}
	\end{align}
	\end{subequations}
	
	For the commutator terms from Corollary \ref{C:metricfirstdiferentialenergyinequality}, we have the following
	estimates on $[0,T):$
	
	\begin{subequations}	
	\begin{align}
		\| [\hat{\square}_g, \partial_{\vec{\alpha}} ] (g_{00} + 1) \|_{L^2} & \leq C e^{-2q \Omega} \totalnorm{N},  
			 \label{E:g00commutatorL2} \\
		\| [\hat{\square}_g, \partial_{\vec{\alpha}} ] g_{0j} \|_{L^2} & \leq 
			 C e^{(1-2q) \Omega} \totalnorm{N}, \label{E:g0jcommutatorL2} \\
		\| [\hat{\square}_g, \partial_{\vec{\alpha}} ] h_{jk} \|_{L^2} & \leq  
			 C e^{-2q \Omega} \totalnorm{N}.  \label{E:hjkcommutatorL2}
	\end{align}
	\end{subequations}
	
	For the terms from Corollary \ref{C:metricfirstdiferentialenergyinequality},
	where $\triangle_{\mathcal{E};(\gamma, \delta)}[v,\partial v]$ is defined in \eqref{E:trianglemathscrEdef},
	we have the following estimates on $[0,T):$
	
	\begin{subequations}
	\begin{align}
		e^{2q \Omega} \| \triangle_{\mathcal{E};(\upgamma_{00}, \updelta_{00})}[\partial_{\vec{\alpha}} (g_{00} + 1)
			,\partial (\partial_{\vec{\alpha}} g_{00})] \|_{L^1}
			& \leq C e^{-q \Omega} \gzerozeronorm{N} \totalnorm{N},    \label{E:triangleEgamma00delta00L1} \\
		e^{2(q-1) \Omega} \| \triangle_{\mathcal{E};(\upgamma_{0*}, \updelta_{0*})}[\partial_{\vec{\alpha}} (g_{0j}),
			\partial (\partial_{\vec{\alpha}} g_{0j})] \|_{L^1}
			& \leq C e^{-q \Omega} \gzerostarnorm{N} \totalnorm{N}, \label{E:triangleEgamma0jdelta0*L1} \\
		e^{2q \Omega} \| \triangle_{\mathcal{E};(0,0)}[0,\partial (\partial_{\vec{\alpha}} h_{jk})] \|_{L^1}
			& \leq C e^{-q \Omega} \hstarstarnorm{N} \totalnorm{N}.
			\label{E:triangleEgamma**delta**L1}
	\end{align}
	\end{subequations}
	
	For the Christoffel symbol error terms defined in \eqref{E:triangle000} - \eqref{E:triangleikj}, 
	we have the following estimates on $[0,T):$
		
		\begin{subequations}
		\begin{align}
		\| \triangle_{0 \ 0}^{\ 0} \|_{H^N} & \leq C e^{-q \Omega} \gnorm{N},  \label{E:triangleGamma000HN} \\		
		\| \triangle_{j \ 0}^{\ 0} \|_{H^N} & \leq C e^{(1-q) \Omega} \gnorm{N}, \label{E:triangleGammaj00HN} \\
		\| \triangle_{0 \ 0}^{\ j} \|_{H^N} & \leq C e^{-(1+q) \Omega} \gnorm{N}, \label{E:triangleGamma0j0HN} \\
		\| \triangle_{0 \ k}^{\ j} \|_{H^N} & \leq C e^{-q \Omega} \gnorm{N}, \label{E:triangleGamma0jkHN} \\
		\| \triangle_{j \ k}^{\ 0} \|_{H^N} & \leq C e^{(2-q) \Omega} \gnorm{N}, \label{E:triangleGammaj0kHN} \\
		\| \triangle_{i \ j}^{\ k} \|_{H^N} & \leq C e^{(1-q) \Omega} \gnorm{N}. \label{E:triangleGammaikjHN}
	\end{align}
	\end{subequations}
	
	For the error terms $\triangle,$ $\triangle^{\mu},$ $\triangle',$ and $\triangle'^{\mu}$ 
	defined in \eqref{E:triangledef}, \eqref{E:trianglejdef}, \eqref{E:triangle0def}, 
	\eqref{E:triangleprimedef}, \eqref{E:triangleprime0def}, \eqref{E:triangleprimejdef} we have the 
	following estimates on $[0,T):$
	
	\begin{subequations}	
	\begin{align}
		\| \triangle \|_{H^N} & \leq C e^{- q \Omega} \totalnorm{N}, \label{E:triangleHN} \\
		\| \triangle^0 \|_{H^N} & \leq C e^{- q \Omega} \totalnorm{N}, \label{E:triangle0HN} \\
		\| \triangle^j \|_{H^N} & \leq C e^{- (1 + q) \Omega} \totalnorm{N}, \label{E:trianglejHN} \\
		\| \triangle' \|_{H^{N-1}} & \leq C e^{-q \Omega} \totalnorm{N}, \label{E:triangleprimeHNminusone} \\
		\| \triangle'^0 \|_{H^{N-1}} & \leq C e^{-q \Omega} \totalnorm{N}, \label{E:triangleprime0HNminusone} \\
		\| \triangle'^j \|_{H^{N-1}} & \leq C e^{- (1 + q) \Omega} \totalnorm{N}. \label{E:triangleprimejHNminusone}
	\end{align}
	\end{subequations}

	For the $L^2$ norm of the variation $\dot{u}^0$ defined in \eqref{E:dot0intermsofdotj}, we have the following estimate
	on $[0,T):$
	
	\begin{align} \label{E:dotu0L2intermsofdotuaL2}
		\| \dot{u}^0 \|_{L^2} & \leq C e^{(1-q) \Omega}\totalnorm{N} \sum_{a=1}^3 \| \dot{u}^a \|_{L^2}.
	\end{align}

For the $L^2$ norms of the inhomogeneous terms $\mathfrak{F}_{\vec{\alpha}},$ 
$\mathfrak{G}_{\vec{\alpha}}^{\mu}$ defined in \eqref{E:FalphamathfrakGjalphainhomogeneousterms}, \eqref{E:mathfrakG0alphainhomogeneousterm}, we have the following estimates on $[0,T):$

\begin{subequations}
\begin{align} \label{E:mathfrakFalphamathfrakGjalphaL2}
	\left\| \begin{pmatrix}
					\mathfrak{F}_{\vec{\alpha}} \\
					\mathfrak{G}_{\vec{\alpha}}^1 \\
					\mathfrak{G}_{\vec{\alpha}}^2 \\
					\mathfrak{G}_{\vec{\alpha}}^3																			
				\end{pmatrix} \right\|_{L^2}
				& \leq C \totalnorm{N}
				\begin{pmatrix}
                        e^{-q \Omega} \\
                       	e^{-(1 + q) \Omega}   \\
                        e^{-(1 + q) \Omega}   \\
                        e^{-(1 + q) \Omega} 
     		\end{pmatrix}, && (0 \leq |\vec{\alpha}| \leq N), \\
	\| \mathfrak{G}_{\vec{\alpha}}^0 \|_{L^2} & \leq C e^{-q \Omega} \totalnorm{N}, && (0 \leq |\vec{\alpha}| \leq N).  
		\label{E:mathfrakG0alphaL2}
\end{align}
\end{subequations}	
	
In the above estimates, the norms $\gzerozeronorm{N},$ $\gzerostarnorm{N},$ $\hstarstarnorm{N},$ $\gnorm{N},$ and $\totalnorm{N}$ are defined in Definition \ref{D:Norms}.
	
\end{proposition}

\begin{proof} 

\begin{remark} \label{R:ProofsRemark}
Throughout the remaining proofs in this article, we will make use of the results of Lemma \ref{L:backgroundaoftestimate}, the definitions of the norms from Section \ref{S:NormsandEnergies}, the definitions \eqref{E:etamindef}, \eqref{E:qdef} of $\upeta_{min}$ and $q,$ and the Sobolev embedding result $H^{M+2}(\mathbb{T}^3) \hookrightarrow C_b^M(\mathbb{T}^3),$ $(M \geq 0)$. \textbf{We will also use the assumption that $\totalnorm{N},$ which is defined in \eqref{E:totalnorm}, is sufficiently small without explicitly mentioning it every time. Furthermore, the smallness is adjusted as necessary at each step in the proof.} For brevity, we don't explicitly estimate how small $\totalnorm{N}$ must be. We also remark that as discussed in Section \ref{SS:runningconstants}, the constants $c,$ $C,$ $C_*$ that appear throughout the article can be chosen uniformly (however, they may depend on $N$) as long as $\totalnorm{N}$ is sufficiently small. Finally, we prove statements in logical order, rather than the order in which they are stated in the proposition.
\end{remark}

Before beginning the proof, we observe the following updated version of the Counting Principle from \cite{iRjS2009}, which provides useful heuristic guidelines for many of the estimates. This tool is only intended to help guide the reader through the estimates; we provide complete proofs of many of the estimates.

\begin{center}
	\textbf{Counting Principle}
\end{center}	
	
	\emph{Consider a product which contains as factors metric components $g_{\mu \nu},$ inverse metric components
	$g^{\mu \nu},$ the first derivatives of these quantities, and the four-velocity $u^{\mu}.$ If $U$ denotes the total number of 
	\textbf{upstairs spatial indices} among these factors, and $D$ denotes the total number of 
	\textbf{downstairs spatial metric indices}, then the expected contribution to the rate of growth/decay of the $H^N$ norm of 
	the product coming from these terms is no larger than $e^{\Omega(D-U)}$ (i.e., we expect these terms to contribute 
	at least this much decay). For purposes of counting, a spatial derivative $\partial_j$ of a metric or inverse metric 
	component is considered to be a downstairs spatial index, while time derivatives 
	of these quantities are neutral. We remark that by these criteria, $h_{jk} \eqdef e^{-2 \Omega} g_{jk},$ 
	$(j,k = 1,2,3),$ is an order $1$ term that doesn't contribute to the decay rate. Furthermore, each factor in a product, 
	excluding $h_{jk}$ but including $\partial_i h_{jk},$ that is equal to one of quantities under an $H^N$ norm in definitions 
	\eqref{E:mathfrakSNg00} - \eqref{E:ujnorm} contributes an additional decay factor of $e^{-q \Omega}.$ Finally, the rescaled 
	pressure $P$ counts as an order $1$ term.}
	\\
	
	\textbf{In this way, many of the estimates proved below in detail can be ascertained by counting spatial indices and making 
	sure that at least one factor in a product contributes an additional decay factor of $e^{-q \Omega}.$}
	
	\ \\
		
\noindent \emph{Proofs of \eqref{E:triangleA00HN} - \eqref{E:triangleC0jHN}}: To prove \eqref{E:triangleA00HN}, we first recall equation \eqref{E:triangleA00def}:

\begin{align}
	\triangle_{A,00} & = (g^{00})^2 \Big\lbrace(\partial_t g_{00})^2 - (\Gamma_{000})^2 \Big\rbrace 
			+ g^{00}g^{0a} \Big\lbrace 2 (\partial_t g_{00})(\partial_t g_{0a} + \partial_a g_{00}) 
			- 4 \Gamma_{000} \Gamma_{00a} \Big\rbrace \label{E:triangleA00defagain} \\
		& \ \ + g^{00}g^{ab} \Big\lbrace (\partial_t g_{0a})(\partial_t g_{0b}) 
				+ (\partial_a g_{00}) (\partial_b g_{00}) 
				- 2 \Gamma_{00a} \Gamma_{00b} \Big\rbrace \notag \\
		& \ \ + g^{0a} g^{0b} \Big\lbrace 2(\partial_t g_{00})(\partial_a g_{0b}) + 2(\partial_t g_{0b})(\partial_a g_{00}) 
				- 2 \Gamma_{000} \Gamma_{a0b} - 2 \Gamma_{00b} \Gamma_{00a} \Big\rbrace \notag \\
		& \ \ + g^{ab} g^{0l} \Big\lbrace 2(\partial_t g_{0a})(\partial_l g_{0b}) + 2(\partial_b g_{00})(\partial_a g_{0l}) 
				- 4\Gamma_{00a} \Gamma_{l0b}  \Big\rbrace \notag \\
		& \ \ + g^{ab}g^{lm}(\partial_a g_{0l})(\partial_b g_{0m}) 
				+ \frac{1}{2} g^{lm}(\underbrace{g^{ab} \partial_t g_{al} - 2\omega \delta_l^b}_{
				e^{2 \Omega} g^{ab} \partial_t h_{al} - 2 \omega g^{0b}g_{0l}})(\partial_b g_{0m} + \partial_m g_{0b}) \notag \\
		& \ \ - \frac{1}{4} g^{ab} g^{lm}(\partial_a g_{0l} + \partial_l g_{0a})(\partial_b g_{0m} + \partial_m g_{0b}) 
			\notag \\
		& \ \ - \frac{1}{4}(\underbrace{g^{ab} \partial_t g_{al} - 2\omega \delta_l^b}_{
				e^{2 \Omega} g^{ab} \partial_t h_{al} - 2 \omega g^{0b}g_{0l}}) 
				(\underbrace{g^{lm} \partial_t g_{bm} - 2 \omega \delta_b^l}_{e^{2 \Omega} g^{lm} \partial_t h_{bm} 
				- 2 \omega g^{0l}g_{0b}}). \notag 
\end{align}

We now use Proposition \ref{P:F1FkLinfinityHN}, the definition \eqref{E:totalnorm} of $\totalnorm{N},$ Sobolev embedding, 
\eqref{E:g00upperplusoneHN}, \eqref{E:g00upperplusoneLinfinity}, \eqref{E:gjkupperLinfinity}, \eqref{E:partialgjkupperHNminusone}, \eqref{E:g0jupperHN}, \eqref{E:g0jupperCb1}, \eqref{E:partialtgjkminusomegagjklowerHN}, \eqref{E:partialtgjklowerC1}, \eqref{E:gjaupperpartialtgaklowerminus2omegadeltakjHN}, \eqref{E:gjaupperpartialtgaklowerminus2omegadeltakjLinfinity}, and the  relation $\Gamma_{\mu \alpha \nu} = \frac{1}{2}(\partial_{\mu} g_{\alpha \nu} + \partial_{\nu} g_{\alpha \mu} - \partial_{\alpha} g_{\mu \nu})$ to conclude that

\begin{align}
	\| \triangle_{A,00} \|_{H^N} & \leq C e^{-2q \Omega} \gnorm{N}^2.
\end{align}
This proves \eqref{E:triangleA00HN}. The proofs of \eqref{E:triangleA0jHN} - \eqref{E:triangleC0jHN} are similar, and we omit the details. We also remark that slight variations of these inequalities were proved in \cite[Lemma 12]{hR2008}. 
\\

\noindent \emph{Proofs of \eqref{E:PLinfinity} - \eqref{E:ujlowerHN}}:
Inequality \eqref{E:PLinfinity} follows from the definition \eqref{E:totalnorm} of $\totalnorm{N}$ and the
Sobolev embedding result $\| P - \bar{p} \|_{L^{\infty}} \leq C \| P - \bar{p} \|_{H^N} \leq C \totalnorm{N}.$

To prove \eqref{E:u0upperHN}, we first recall equation \eqref{E:U0UPPERISOLATED}:

\begin{align} \label{E:U0UPPERISOLATEDagain}
	u^0 = - \frac{g_{0a}u^a}{g_{00}} + \sqrt{1 + \Big(\frac{g_{0a}u^a}{g_{00}}\Big)^2 - \frac{g_{ab}u^a u^b}{g_{00}} 
		- \Big(\frac{g_{00} + 1}{g_{00}}\Big)}.
\end{align}
We now apply Corollary \ref{C:SobolevTaylor}, Proposition \ref{P:F1FkLinfinityHN}, and Sobolev embedding to conclude that

\begin{align} \label{E:u0HNfirstinequality}
	\| u^0 - 1 \|_{H^N} & \leq C \bigg\lbrace \Big\| \frac{g_{0a}u^a}{g_{00}} \Big\|_{H^N} 
		+ \Big\| \Big(\frac{g_{0a}u^a}{g_{00}}\Big)^2 - \frac{g_{ab}u^a u^b}{g_{00}} - \frac{g_{00} + 1}{g_{00}} \Big\|_{H^N} 
		\bigg\rbrace \\
	& \leq C \bigg\lbrace 
		\Big\| \frac{1}{g_{00}} \Big \|_{L^{\infty}} +  \Big\| \underpartial \Big(\frac{1}{g_{00}}\Big) \Big\|_{H^{N-1}} 
		+ \Big\| \frac{1}{g_{00}} \Big \|_{L^{\infty}}^2 +  \Big\| \underpartial \Big(\frac{1}{g_{00}}\Big) \Big\|_{H^{N-1}}^2 
		\bigg\rbrace \notag \\
	& \ \times \bigg\lbrace \| g_{0a} \|_{H^N} \| u^a \|_{H^N}  
		+ (\| g_{0a} \|_{H^N} \| u^a \|_{H^N})^2 \notag \\
	& \ \ \ \ \ + (\| g_{ab} \|_{L^{\infty}} + \| \underpartial g_{ab} \|_{H^{N-1}})
		\| u^a \|_{L^{\infty}} \| \underpartial u^b \|_{H^{N-1}} + \| g_{00} + 1 \|_{H^N} \bigg\rbrace.
\end{align}
Using Corollary \ref{C:DifferentiatedSobolevComposition}, \eqref{E:gjklowerLinfinity}, the definition of $\totalnorm{N},$ 
Sobolev embedding, and in particular the estimate $\| u^a \|_{L^{\infty}} \leq C e^{-(1 + q)\Omega} \totalnorm{N},$
it follows that

\begin{align}
	\| u^0 - 1 \|_{H^N} & \leq C \totalnorm{N} e^{-q \Omega}.
\end{align}
This proves \eqref{E:u0upperHN}. Inequality \eqref{E:u0upperLinfinity} then follows from \eqref{E:u0upperHN} and Sobolev embedding.

Inequalities \eqref{E:u0lowerHN} - \eqref{E:u0lowerLinfinity} can be proved similarly using \eqref{E:u0upperHN} -
\eqref{E:u0upperLinfinity} and the triangle inequality estimate
\begin{align}
	\| u_0 + 1 \|_{H^N} & = \|g_{0\alpha}u^{\alpha} + 1 \|_{H^N} \leq \|(g_{00} + 1)u^0 \|_{H^N}
		+ \|u^0 - 1 \|_{H^N} + \|g_{0a}u^a \|_{H^N}.
\end{align}
Inequalities \eqref{E:ujlowerHNminusone} - \eqref{E:ujlowerHN} can also be proved similarly 
using the relation $u_j = g_{j \alpha} u^{\alpha}.$
\\

\noindent \emph{Proofs of \eqref{E:triangle00HN} - \eqref{E:trianglejkHN}}:
To prove \eqref{E:triangle00HN}, we first use equation \eqref{E:triangle00}
and Proposition \ref{P:F1FkLinfinityHN} to arrive at the following estimate:

\begin{align}
	\| \triangle_{00} \|_{H^N} & \leq C\Big\lbrace \| \triangle_{A,00} \|_{H^N} + \| \triangle_{C,00} \|_{H^N}
		+ e^{-3(1 + \speed^2) \Omega} \| g_{00} + 1 \|_{H^N} \\
	& \ \ + \| P \|_{L^{\infty}} \| u_0 + 1 \|_{L^{\infty}} \| u_0 - 1 \|_{H^N}
		+ \| u_0 + 1 \|_{L^{\infty}} \| P \|_{L^{\infty}} \| u_0 - 1 \|_{L^{\infty}} \| \underpartial u_0 \|_{H^{N-1}} 
		 \notag \\
	& \ \ + \| u_0 + 1 \|_{L^{\infty}} \| u_0 - 1 \|_{L^{\infty}} \| \underpartial P \|_{H^{N-1}}
		+ e^{-3(1 + \speed^2)\Omega} \| P \|_{L^{\infty}} \| g_{00} + 1 \|_{H^N} 
		 \notag \\
	& \ \ + e^{-3(1 + \speed^2)\Omega} \| g_{00} + 1 \|_{L^{\infty}} \| \underpartial P \|_{H^{N-1}} \notag \\
	& \ \ + |\omega - H| \| \partial_t g_{00}\|_{H^N} + |\omega - H| \| g_{00} + 1 \|_{H^N} \Big\rbrace. \notag
\end{align}
We now use \eqref{E:g00upperplusoneHN}, \eqref{E:g00upperplusoneLinfinity}, \eqref{E:triangleA00HN}, \eqref{E:triangleC00HN}, \eqref{E:PLinfinity}, \eqref{E:u0lowerHN}, \eqref{E:u0lowerLinfinity}, the definition \eqref{E:totalnorm} of $\totalnorm{N},$ and Sobolev embedding to conclude that

\begin{align}
	\| \triangle_{00} \|_{H^N} < C e^{-2q \Omega}\totalnorm{N},
\end{align}
which proves \eqref{E:triangle00HN}. Inequalities \eqref{E:triangle0jHN} and \eqref{E:trianglejkHN} can be proved using similar reasoning; we omit the details.
\\

\noindent \emph{Proofs of \eqref{E:triangleEgamma00delta00L1} - \eqref{E:triangleEgamma**delta**L1}}:
To begin, we first recall equation \eqref{E:trianglemathscrEdef}:

\begin{align} \label{E:trianglemathscrEdefagain}
	\triangle_{\mathcal{E};(\upgamma, \delta)}[v,\partial v] & = - \upgamma H (\partial_a g^{ab}) v \partial_b v
		- 2 \upgamma H (\partial_a g^{0a}) v \partial_t v - 2 \upgamma H g^{0a}(\partial_a v)(\partial_t v) \\
	& \ \ - (\partial_a g^{0a})(\partial_t v)^2 - (\partial_a g^{ab})(\partial_b v)(\partial_t v)
		- \frac{1}{2}(\partial_t g^{00})(\partial_t v)^2 \notag \\
	& \ \ + \Big(\frac{1}{2} \partial_t g^{ab} + \omega g^{ab} \Big) (\partial_a v) (\partial_b v)
		+ (H - \omega) g^{ab} (\partial_a v) (\partial_b v) \notag \\
	& \ \ - \upgamma H (\partial_t g^{00}) v \partial_t v - \upgamma H (g^{00} + 1)(\partial_t v)^2. \notag
\end{align}

We now claim that the following inequality holds for any function $v$ for which the right-hand side is finite:
\begin{align} \label{E:trianglegammadeltaL1}
	\| \triangle_{\mathcal{E};(\upgamma, \updelta)}[v,\partial v] \|_{L^1} & \leq C \Big\lbrace e^{-q \Omega}  
		\| \partial_t v \|_{L^2}^2 + e^{-(2 + q) \Omega} 
		\| \underpartial v \|_{L^2}^2 + C_{(\upgamma)} e^{-q \Omega} \| v \|_{L^2}^2 \Big\rbrace,
\end{align}
where $C_{(\upgamma)}$ is defined in \eqref{E:mathcalEfirstlowerbound} To obtain \eqref{E:trianglegammadeltaL1}, we use the Cauchy-Schwarz inequality for integrals, \eqref{E:g00upperplusoneLinfinity}, \eqref{E:gjkupperLinfinity}, \eqref{E:partialgjkupperLinfinity}, \eqref{E:g0jupperCb1}, \eqref{E:partialtgjkupperplusomegagjkLinfinity}, and \eqref{E:partialtg00upperLinfinity}. Inequalities \eqref{E:triangleEgamma00delta00L1} - \eqref{E:triangleEgamma**delta**L1} now easily follow from definitions \eqref{E:mathfrakSNg00} - \eqref{E:totalnorm} and \eqref{E:trianglegammadeltaL1}.
\\

\noindent \emph{Proofs of \eqref{E:triangleGamma000HN} - \eqref{E:triangleGammaikjHN}}:
To estimate $\triangle_{0 \ 0}^{\ 0},$ we first recall equation \eqref{E:triangle000}: 

\begin{align} 
	\triangle_{0 \ 0}^{\ 0} & = \frac{1}{2}g^{00} \partial_t g_{00} + g^{0a} \partial_t g_{0a} - 
		\frac{1}{2} g^{0a}\partial_a g_{00}. \label{E:triangleGamma000again}
\end{align}
	Using Proposition \ref{P:F1FkLinfinityHN}, the definition \eqref{E:totalnorm} of $\totalnorm{N},$ Sobolev embedding, 
 	\eqref{E:g00upperLinfinity}, \eqref{E:g00upperplusoneHN}, \eqref{E:g0jupperHN}, and \eqref{E:g0jupperCb1}, 
 	it follows that
 
\begin{align}
	\| \triangle_{0 \ 0}^{\ 0} \|_{H^N} \leq C e^{- q \Omega} \gnorm{N},
\end{align}
which proves \eqref{E:triangleGamma000HN}. 

The estimates \eqref{E:triangleGammaj00HN} - \eqref{E:triangleGammaikjHN} can be proved similarly using 
Propositions \ref{P:BoostrapConsequences} and \ref{P:F1FkLinfinityHN}. 
\\

\noindent \emph{Proofs of \eqref{E:triangleHN} - \eqref{E:triangleprimejHNminusone}}:
To prove \eqref{E:triangleHN}, we first apply Proposition \ref{P:F1FkLinfinityHN} to
equation \eqref{E:triangledef}, concluding that:

\begin{align} \label{E:triangleHNfirstinequality}
	\| \triangle \|_{H^N} & \leq C \bigg\lbrace 
		\| P \|_{L^{\infty}} \| \triangle_{\alpha \ 0}^{\ \alpha}\|_{L^{\infty}} \| \underpartial u^0 \|_{H^{N-1}} \|
		+ \| P \|_{L^{\infty}} \| u^0 \|_{L^{\infty}}  \| \triangle_{\alpha \ 0}^{\ \alpha}\|_{H^N}
		 \\
	& \ \ + \| \triangle_{\alpha \ 0}^{\ \alpha}\|_{L^{\infty}} 
		\| u^0 \|_{L^{\infty} } \| \underpartial P \|_{H^{N-1}}
		+ \| P \|_{L^{\infty}} \| \triangle_{\alpha \ a}^{\ \alpha}\|_{L^{\infty}} \| u^a \|_{H^N} 
			\notag \\
	& \ \ + \| P \|_{L^{\infty}}  \| u^a \|_{L^{\infty}} \| \triangle_{\alpha \ a}^{\ \alpha}\|_{H^N}
		+ \| \triangle_{\alpha \ a}^{\ \alpha}\|_{L^{\infty}} \| u^a \|_{L^{\infty} } \| \underpartial P \|_{H^{N-1}}
		\notag \\
	& \ \  + \Big\| \frac{1}{u_0} \Big\|_{L^{\infty}} \| u^0 \|_{L^{\infty}}^2 \| \partial_t g_{00} \|_{H^N}
		+ \Big\| \frac{1}{u_0} \Big\|_{L^{\infty}} \| \partial_t g_{00} \|_{L^{\infty}} 
		\| u^0 \|_{L^{\infty}} \| \underpartial u^0 \|_{H^{N-1}}  \notag \\
	& \ \ + \| u^0 \|_{L^{\infty}}^2 \| \partial_t g_{00} \|_{L^{\infty}} 
		\Big\| \underpartial \Big( \frac{1}{u_0} \Big) \Big\|_{H^{N-1}} 
		+ \Big\| \frac{1}{u_0} \Big\|_{L^{\infty}} \| \partial_t g_{0a} \|_{L^{\infty}} \| u^0 \|_{L^{\infty}}\| u^a \|_{H^N}
		\notag \\
	& \ \ + \Big\| \frac{1}{u_0} \Big\|_{L^{\infty}} \| \partial_t g_{0a} \|_{L^{\infty}} 
		\| u^a \|_{L^{\infty}} \| \underpartial u^0 \|_{H^{N-1}} 
		+ \Big\| \frac{1}{u_0} \Big\|_{L^{\infty}}  \| u^0 \|_{L^{\infty}}\| u^a \|_{L^{\infty}} \| \partial_t g_{0a} \|_{H^N}
		\notag \\
	& \ \ + \| \partial_t g_{0a} \|_{L^{\infty}} 
		\| u^0 \|_{L^{\infty}}\| u^a \|_{L^{\infty}} \Big\| \underpartial \Big( \frac{1}{u_0} \Big) \Big\|_{H^{N-1}} 
		+ \Big\| \frac{1}{u_0} \Big\|_{L^{\infty}} \| \partial_t g_{ab} \|_{L^{\infty}} \| u^a \|_{L^{\infty}} 
		\| u^b \|_{H^N} \notag \\
	& \ \ + \Big\| \frac{1}{u_0} \Big\|_{L^{\infty}} \| u^a \|_{L^{\infty}} \| u^b \|_{L^{\infty}} 
			\| \underpartial (\partial_t g_{ab}) \|_{H^{N-1}} \notag \\
	& \ \ + \| u^a \|_{L^{\infty}} \| u^b \|_{L^{\infty}} \| \partial_t g_{ab} \|_{L^{\infty}} 
			\Big\| \underpartial \Big( \frac{1}{u_0} \Big) \Big\|_{H^{N-1}} \bigg\rbrace. \notag
\end{align}

Using Corollary \ref{C:DifferentiatedSobolevComposition} to estimate $\Big\| \underpartial \Big( \frac{1}{u_0} \Big) \Big\|_{H^{N-1}},$ \eqref{E:partialtgjkminusomegagjklowerHN}, \eqref{E:partialtgjklowerC1}, 
\eqref{E:u0upperHN} - \eqref{E:u0lowerLinfinity}, \eqref{E:triangleGamma000HN} - \eqref{E:triangleGammaikjHN}
the definition \eqref{E:totalnorm} of $\totalnorm{N},$ Sobolev embedding, and in particular the estimate $\| u^a \|_{L^{\infty}} \leq C \| u^a \|_{H^{N-1}} \leq C e^{-(1 + q)\Omega}\totalnorm{N},$ it follows that

\begin{align}
	\| \triangle \|_{H^N} & \leq C e^{-q \Omega} \totalnorm{N}.
\end{align}
This proves \eqref{E:triangleHN}. The proofs of \eqref{E:triangle0HN} - \eqref{E:triangleprimejHNminusone}
are similar, and we omit the details.
\\

\noindent \emph{Proofs of \eqref{E:partialtPHNminusone} - \eqref{E:partialtujlowerLinfinity}}:

Inequality \eqref{E:partialtPHNminusone}, follows trivially from equation \eqref{E:partialtP} and from \eqref{E:triangleprimeHNminusone}. \eqref{E:partialtu0upperLinfinity} then follows from \eqref{E:partialtPHNminusone}
and Sobolev embedding. The proofs of \eqref{E:partialtu0upperHNminusone} - \eqref{E:partialtujupperLinfinity} are similar.

The estimates \eqref{E:partialtu0lowerHNminusone} - \eqref{E:partialtujlowerLinfinity} follow from the identity

\begin{align}
	\partial_t u_{\mu} & = g_{\mu \alpha} \partial_t u^{\alpha} + (\partial_t g_{\mu \alpha})u^{\alpha},&& (\mu = 0,1,2,3),
\end{align}
Proposition \ref{P:F1FkLinfinityHN}, the definition \eqref{E:totalnorm} of $\totalnorm{N},$ Sobolev embedding, the estimates of Proposition \ref{P:BoostrapConsequences}, and \eqref{E:partialtu0upperHNminusone} - \eqref{E:partialtujupperLinfinity}.
\\

\noindent \emph{Proof of \eqref{E:dotu0L2intermsofdotuaL2}}:  
To prove \eqref{E:dotu0L2intermsofdotuaL2}, we first recall equation \eqref{E:dot0intermsofdotjagain}:

\begin{align} \label{E:dot0intermsofdotjagain}
	\dot{u}^0 \eqdef  -\frac{1}{u_0}u_a \dot{u}^a.
\end{align}

Using \eqref{E:u0lowerLinfinity} and \eqref{E:ujlowerLinfinity}, it follows that

\begin{align}
	\| \dot{u}^0 \|_{L^2} & \leq \Big\| \frac{1}{u_0} \Big\|_{L^{\infty}} \| u_a \|_{L^{\infty}} \| \dot{u}^a \|_{L^2}
		\leq C e^{(1-q)\Omega} \totalnorm{N} \| \dot{u}^a \|_{L^2},
\end{align}
which proves \eqref{E:dotu0L2intermsofdotuaL2}.
\\

\noindent \emph{Proofs of \eqref{E:mathfrakFalphamathfrakGjalphaL2} - \eqref{E:mathfrakG0alphaL2}}: To prove
\eqref{E:mathfrakFalphamathfrakGjalphaL2}, we first use equation \eqref{E:FalphamathfrakGjalphainhomogeneousterms}
to deduce that

\begin{align} \label{E:FalphamathfrakGjalphainhomogeneoustermsagain}
	\left\| \begin{pmatrix}
					\mathfrak{F}_{\vec{\alpha}} \\
					\mathfrak{G}_{\vec{\alpha}}^1 \\
					\mathfrak{G}_{\vec{\alpha}}^2 \\
					\mathfrak{G}_{\vec{\alpha}}^3																			
				\end{pmatrix} \right\|_{L^2}
	& \leq \big \|  A^0 \partial_{\vec{\alpha}} \big[ (A^0)^{-1}\mathbf{b} \big] 
		- \partial_{\vec{\alpha}} \mathbf{b} \big\|_{L^2} 
		+ \big\| A^0 \partial_{\vec{\alpha}} \big[ (A^0)^{-1} \mathbf{b}_{\triangle} \big] \big\|_{L^2} \\
	& \ \ + \Big\| A^0 \Big\lbrace(A^0)^{-1}A^a \partial_a \partial_{\vec{\alpha}} \mathbf{W}
 		- \partial_{\vec{\alpha}} \big[(A^0)^{-1} A^a \partial_a \mathbf{W}\big] \Big\rbrace \Big\|_{L^2}. \notag
\end{align}
See the remarks at the end of Section \ref{SS:NORMS} concerning our use of notation for the norms of array-valued functions.

We will estimate each of the three terms on the right-hand side of \eqref{E:FalphamathfrakGjalphainhomogeneoustermsagain}
using the following estimates for $\mathbf{W} \eqdef (P - \bar{p}, u^1,u^2,u^3)^{T},$ \\
$\mathbf{b}=(0,(\decayparameter - 2)u^1,(\decayparameter - 2)u^2,(\decayparameter - 2)u^3)^{T},$ $\mathbf{b}_{\triangle} = (\triangle, \triangle^1,\triangle^2,\triangle^3)^{T},$ $A^{\mu},$ and $(A^0)^{-1}:$

\begin{align} 
		& \| \mathbf{b} \|_{H^{N-1}}
		\leq C \totalnorm{N} \begin{pmatrix}
                        0 \\
                       	e^{-(1 + q) \Omega}   \\
                        e^{-(1 + q) \Omega}   \\
                        e^{-(1 + q) \Omega} 
                    \end{pmatrix}, \qquad  
      \| \mathbf{b} \|_{H^N}
							\leq C \totalnorm{N} \begin{pmatrix}
                        0 \\
                       	e^{- \Omega}   \\
                        e^{- \Omega}   \\
                        e^{- \Omega} 
                    \end{pmatrix},
\end{align}

\begin{align}
	\| A^{0} \|_{L^{\infty}}
	& \leq C \begin{pmatrix}
                        1  & e^{(1-q)\Omega}\totalnorm{N} & e^{(1-q)\Omega}\totalnorm{N} & e^{(1-q)\Omega}\totalnorm{N} \\
                       	e^{-(1+q)\Omega}\totalnorm{N} & 1 & 0 & 0 \\
                        e^{-(1+q)\Omega}\totalnorm{N} & 0 & 1 & 0 \\
                        e^{-(1+q)\Omega}\totalnorm{N} & 0 & 0 & 1 \\
                    \end{pmatrix},
\end{align}

\begin{align}
	\big\| (A^{0})^{-1} \big\|_{L^{\infty}}
	& \leq C \begin{pmatrix}
                        1  & e^{(1-q)\Omega}\totalnorm{N} & e^{(1-q)\Omega}\totalnorm{N} & e^{(1-q)\Omega}\totalnorm{N} \\
                       	e^{-(1+q)\Omega}\totalnorm{N} & 1 & e^{-q \Omega}\totalnorm{N} & e^{-q \Omega}\totalnorm{N} \\
                        e^{-(1+q)\Omega}\totalnorm{N} & e^{-q \Omega}\totalnorm{N} & 1 & e^{-q \Omega}\totalnorm{N} \\
                        e^{-(1+q)\Omega}\totalnorm{N} & e^{-q \Omega}\totalnorm{N} & e^{-q \Omega}\totalnorm{N} & 1 \\
                    \end{pmatrix},
\end{align}

\begin{align}
	\big\| \underpartial (A^{0})^{-1} \big\|_{H^{N-1}}
	& \leq C \totalnorm{N} \begin{pmatrix}
                        e^{-q \Omega}  & e^{\Omega} & e^{\Omega} & e^{\Omega} \\
                       	e^{- \Omega} & e^{-q \Omega} & e^{-q \Omega} & e^{-q \Omega} \\
                        e^{- \Omega} & e^{-q \Omega} & e^{-q \Omega} & e^{-q \Omega}\\
                        e^{- \Omega} & e^{-q \Omega} & e^{-q \Omega} & e^{-q \Omega} \\
                    \end{pmatrix},
\end{align}

\begin{align}
	\| \mathbf{b}_{\triangle} \|_{H^N} 
		\leq
			C \totalnorm{N} \begin{pmatrix}
                        e^{-q \Omega} \\
                       	e^{-(1 + q) \Omega}   \\
                        e^{-(1 + q) \Omega}   \\
                        e^{-(1 + q) \Omega} 
                    \end{pmatrix},
\end{align}

\begin{align} 
            \| \mathbf{W} \|_{H^N}
							\leq C \totalnorm{N} \begin{pmatrix}
                        1 \\
                       	e^{- \Omega}   \\
                        e^{- \Omega}   \\
                        e^{- \Omega} 
                    \end{pmatrix},
\end{align}

\begin{align}
	\| A^1 \|_{L^{\infty}} & \leq C \begin{pmatrix}
                        e^{-(1+q)\Omega} \totalnorm{N} & 1  & 0 & 0 \\
                       	e^{- 2\Omega} & e^{-(1 + q) \Omega} \totalnorm{N} & 0 & 0 \\
                        e^{- 2\Omega} & 0 & e^{-(1 + q) \Omega} \totalnorm{N}  & 0 \\
                        e^{- 2\Omega} & 0 & 0 & e^{-(1 + q) \Omega} \totalnorm{N}  \\
                    \end{pmatrix},
\end{align}

\begin{align}
	\| \underpartial A^1 \|_{H^{N-1}} & \leq C \totalnorm{N} \begin{pmatrix}
                        e^{- \Omega}  & 1  & 0 & 0 \\
                       	e^{- 2 \Omega} & e^{- \Omega} & 0 & 0 \\
                        e^{- 2 \Omega} & 0 & e^{- \Omega} & 0 \\
                        e^{- 2 \Omega} & 0 & 0 & e^{- \Omega}  \\
                    \end{pmatrix},
\end{align}
and analogously for $A^2,A^3.$ All of the above estimates follow from repeated applications of 
Corollary \ref{C:DifferentiatedSobolevComposition}, Proposition \ref{P:F1FkLinfinityHN}, the definition \eqref{E:totalnorm} of $\totalnorm{N},$ Sobolev embedding, the estimates of Proposition \ref{P:BoostrapConsequences}, \eqref{E:PLinfinity} - \eqref{E:partialtujlowerLinfinity}, and \eqref{E:triangleHN} - \eqref{E:trianglejHN}.

To estimate the first term on the right-hand side of \eqref{E:FalphamathfrakGjalphainhomogeneoustermsagain}, we use Propositions \ref{P:F1FkLinfinityHN} and \ref{P:SobolevMissingDerivativeProposition}, together with the above estimates and Sobolev embedding, to conclude that for $0 \leq |\vec{\alpha}| \leq N$

\begin{align} \label{E:1firsttermmathfrakFalphamathfrakGjalphaL2again}
	\Big\| A^0 & \partial_{\vec{\alpha}} \big[ (A^0)^{-1}\mathbf{b} \big] - \partial_{\vec{\alpha}} \mathbf{b} \Big\|_{L^2} 
		\\
	& \leq C \| A^{0} \|_{L^{\infty}} 
			* \Big\| \partial_{\vec{\alpha}} \big[ (A^0)^{-1}\mathbf{b} \big] - (A^0)^{-1} \partial_{\vec{\alpha}} \mathbf{b} 
			\Big\|_{L^2} \notag \\
	& \leq C \| A^{0} \|_{L^{\infty}} * \bigg\lbrace \big\| \underpartial (A^0)^{-1} \big\|_{L^{\infty}} 
		* \| \mathbf{b} \|_{H^{N-1}}
		+ \| \underpartial (A^0)^{-1} \|_{H^{N-1}} * \| \mathbf{b} \|_{L^{\infty}} \bigg\rbrace \notag \\
		& \leq C \totalnorm{N} \begin{pmatrix}
                        e^{-q \Omega} \\
                       	e^{-(1 + q) \Omega}   \\
                        e^{-(1 + q) \Omega}   \\
                        e^{-(1 + q) \Omega} 
                    \end{pmatrix}, \notag
\end{align}
where we write $*$ to emphasize that we are performing matrix multiplication on the matrices of norms.

We estimate the second and third terms on the right-hand side of \eqref{E:FalphamathfrakGjalphainhomogeneoustermsagain} using similar reasoning, which results in the following bounds:

\begin{align} 
		\Big\| A^0 & \Big\lbrace(A^0)^{-1}A^a \partial_a \partial_{\vec{\alpha}} \mathbf{W}
 		- \partial_{\vec{\alpha}} \big[(A^0)^{-1} A^a \partial_a \mathbf{W}\big] \Big\rbrace \Big\|_{L^2} \\
 		& \leq C \| A^0 \|_{L^{\infty}} * \| \underpartial [(A^0)^{-1} A^a ] \|_{H^{N-1}} 
 			* \| \partial_a \mathbf{W} \|_{H^{N-1}} \notag \\
 		& \leq C \| A^0 \|_{L^{\infty}} * \Big\lbrace 
 			\| (A^0)^{-1} \|_{L^{\infty}} * \| \underpartial A^a \|_{H^{N-1}}
 			+ \| \underpartial (A^0)^{-1} \|_{H^{N-1}} * \| \underpartial A^a \|_{H^{N-1}} \notag \\
 		& \hspace{1in} + \| \underpartial (A^0)^{-1} \|_{H^{N-1}} * \| A^a \|_{L^{\infty}}
 			\Big\rbrace * \| \partial_a \mathbf{W} \|_{H^{N-1}} \notag \\
 		& \leq C \totalnorm{N} \begin{pmatrix}
                        e^{-\Omega} \\
                       	e^{-2\Omega}   \\
                        e^{-2\Omega}   \\
                        e^{-2\Omega} 
                    \end{pmatrix}, \notag 
\end{align}

\begin{align} \label{E:2firsttermmathfrakFalphamathfrakGjalphaL2again}
	\Big\| A^0 & \partial_{\vec{\alpha}} \big[ (A^0)^{-1} \mathbf{b}_{\triangle} \big] \Big\|_{L^2} \\
		& \leq \| A^0 \|_{L^{\infty}} * \bigg\lbrace \|(A^0)^{-1}\|_{L^{\infty}} * \|\mathbf{b}_{\triangle} \|_{H^N}
		+ \|\underpartial (A^0)^{-1} \|_{H^{N-1}} * \|\mathbf{b}_{\triangle} \|_{L^{\infty}}  \bigg\rbrace 
		 \notag \\
		& \leq C \totalnorm{N} \begin{pmatrix}
                        e^{-q \Omega} \\
                       	e^{-(1 + q) \Omega}   \\
                        e^{-(1 + q) \Omega}   \\
                        e^{-(1 + q) \Omega} 
                    \end{pmatrix}.  \notag
\end{align}
Finally, adding \eqref{E:1firsttermmathfrakFalphamathfrakGjalphaL2again} -
\eqref{E:2firsttermmathfrakFalphamathfrakGjalphaL2again} implies that

\begin{align} 
	\left\| \begin{pmatrix}
					\mathfrak{F}_{\vec{\alpha}} \\
					\mathfrak{G}_{\vec{\alpha}}^1 \\
					\mathfrak{G}_{\vec{\alpha}}^2 \\
					\mathfrak{G}_{\vec{\alpha}}^3																			
				\end{pmatrix} \right\|_{L^2}
				& \leq C \totalnorm{N}
				\begin{pmatrix}
                        e^{-q \Omega} \\
                       	e^{-(1 + q) \Omega}   \\
                        e^{-(1 + q) \Omega}   \\
                        e^{-(1 + q) \Omega} 
     		\end{pmatrix},
\end{align}
which proves \eqref{E:mathfrakFalphamathfrakGjalphaL2}.

To prove \eqref{E:mathfrakG0alphaL2}, we use equation \eqref{E:mathfrakG0alphainhomogeneousterm},
the definition \eqref{E:totalnorm} of $\totalnorm{N},$ Sobolev embedding, and in particular the 
estimates $\| u^j \|_{L^{\infty}} \leq C e^{-(1 + q)\Omega} \totalnorm{N}, \| u_j \|_{L^{\infty}} \leq C e^{(1 - q)\Omega} \totalnorm{N},$ \eqref{E:partialtu0upperLinfinity}, \eqref{E:partialtujlowerLinfinity}, and \eqref{E:mathfrakFalphamathfrakGjalphaL2}
to conclude that

\begin{align} \label{E:mathfrakG0alphainhomogeneoustermagain}
		\| \mathfrak{G}_{\vec{\alpha}}^0 \|_{L^2} & \leq \Big\| u^{\nu} \partial_{\nu} \Big(\frac{u_a}{u_0}\Big) 
			\Big\|_{L^{\infty}} \| u^a \|_{H^N}
		+ \Big\| \frac{1}{u_0} \Big \|_{L^{\infty}} \| u_a \|_{L^{\infty}} \sum_{|\vec{\alpha}| \leq N} 
			\| \mathfrak{G}_{\vec{\alpha}}^a \|_{L^2} \\
		& \leq \| u^{\nu} \|_{L^{\infty}} \| \partial_{\nu} u_a \|_{L^{\infty} } \Big \| \frac{1}{u_0} \Big \|_{L^{\infty}} 
			\| u^a \|_{H^N} \notag \\
		& \ \ + \| u^{\nu} \|_{L^{\infty}} \| u_a \|_{L^{\infty} } \Big\| \Big(\frac{1}{u_0} \Big)^2 \Big\|_{L^{\infty}} 
			\| \partial_{\nu} u^0 \|_{L^{\infty}} \| u^a \|_{H^N} \notag \\
		& \ \ + \Big\| \frac{1}{u_0} \Big \|_{L^{\infty}} \| u_a \|_{L^{\infty}} \sum_{|\vec{\alpha}| \leq N} 
			\| \mathfrak{G}_{\vec{\alpha}}^a \|_{L^2} \notag \\
		& \leq C e^{-q\Omega} \totalnorm{N}. \notag
\end{align}
This completes the proof of \eqref{E:mathfrakG0alphaL2}.
\\

\noindent \emph{Proofs of \eqref{E:g00commutatorL2} - \eqref{E:hjkcommutatorL2}}:
To prove \eqref{E:g00commutatorL2}, we first estimate $\| \partial_t^2 g_{00} \|_{H^{N-1}}.$ Using
equation \eqref{E:finalg00equation}, we have that

\begin{align} \label{E:partialtsquaredg00isolated}
	\partial_t^2 g_{00} = (g^{00})^{-1} \Big\lbrace -g^{ab} \partial_a \partial_b g_{00}  
	-2 g^{0a} \partial_a \partial_t g_{00} + 5 H \partial_t g_{00} + 6 H^2 (g_{00} + 1) + \triangle_{00} \Big\rbrace.
\end{align}

Using \eqref{E:partialtsquaredg00isolated}, Corollary \ref{C:DifferentiatedSobolevComposition}, Proposition \ref{P:SobolevMissingDerivativeProposition},
the definition \eqref{E:totalnorm} of $\totalnorm{N},$ Sobolev embedding, 
\eqref{E:gjkupperLinfinity}, \eqref{E:g0jupperHN}, \eqref{E:g00upperplusoneLinfinity}, 
and \eqref{E:triangle00HN}, it follows that

\begin{align} \label{E:partialtsquaredg00HNminusone}
	\| \partial_t^2 g_{00} \|_{H^{N-1}} \leq C e^{- q \Omega} \totalnorm{N}.
\end{align}

We now use Corollary \ref{C:DifferentiatedSobolevComposition} and Proposition \ref{P:SobolevMissingDerivativeProposition}, together with the definition \eqref{E:totalnorm} of $\totalnorm{N},$ Sobolev embedding, \eqref{E:partialg00upperHNminusone}, \eqref{E:partialgjkupperHNminusone},
\eqref{E:g0jupperHN}, \eqref{E:partialtsquaredg00HNminusone} to obtain

\begin{align}
	\| [\hat{\square}_g, \partial_{\vec{\alpha}}] (g_{00} + 1) \|_{L^2} & \leq 
		\| g^{00}\partial_{\vec{\alpha}} (\partial_{t}^2g_{00}) 
		- \partial_{\vec{\alpha}}(g^{00} \partial_{t}^2 g_{00}) \|_{H^N} \\
	& \ \ + \| g^{ab}\partial_{\vec{\alpha}} (\partial_{a} \partial_b g_{00}) 
		- \partial_{\vec{\alpha}}(g^{ab} \partial_{a} \partial_b g_{00}) \|_{H^N} \notag \\
	& \ \ + 2\| g^{0a}\partial_{\vec{\alpha}} (\partial_{t} \partial_a g_{00})
		- \partial_{\vec{\alpha}}(g^{0a} \partial_{t} \partial_a g_{00}) \|_{H^N} \notag \\
	& \leq C  \|\underpartial g^{00} \|_{H^{N-1}} \| \partial_t^2 g_{00} \|_{H^{N-1}}
		+ C \|\underpartial g^{ab} \|_{H^{N-1}} \| \partial_a \partial_b g_{00} \|_{H^{N-1}} \notag \\
	& \ \ + C \| \underpartial g^{0a} \|_{H^{N-1}} \| \partial_a \partial_t g_{00} \|_{H^{N-1}} \notag \\
	& \leq C e^{-2q \Omega} \totalnorm{N}.  \notag
\end{align}
This completes the proof of \eqref{E:g00commutatorL2}. Inequalities \eqref{E:g0jcommutatorL2} and 
\eqref{E:hjkcommutatorL2} can be proved similarly; we omit the details.

\end{proof}

\section{The Equivalence of Sobolev and Energy Norms} \label{S:EnergyNormEquivalence}

Our global existence proof is based on a standard strategy: showing that the Sobolev norms of Section \ref{S:NormsandEnergies} satisfy suitable bounds (they happen to be uniformly bounded for $t \geq 0$ by $C \epsilon$ in the problem studied here), and then appealing to the continuation principle of Theorem \ref{T:ContinuationCriterion}. However, we do not have direct control over the growth of the norms; we can only control the norms indirectly through the use of the energies. In this section, we bridge the gap between the energies and the norms. More specifically, in the following proposition, we prove that under suitable bootstrap assumptions, the Sobolev-type norms and energies defined in Section \ref{S:NormsandEnergies} are equivalent.

\begin{proposition}[\textbf{Equivalence of Sobolev Norms and Energy Norms}] \label{P:energynormomparison} 
Let $N \geq 3$ be an integer, and assume that the bootstrap assumptions \eqref{E:metricBAeta} - \eqref{E:g0jBALinfinity} hold 
on the spacetime slab $[0,T) \times \mathbb{T}^3$ for some constant $c_1 \geq 1$ and for $\upeta = \upeta_{min}.$ Let 
$(\updelta, \upgamma)$ be any of the pairs of constants given in Definition \ref{D:energiesforg}, and let $C_{(\upgamma)}$ be the corresponding constant from Lemma \ref{L:buildingblockmetricenergy}. There exist constants $\epsilon''' > 0$ and $C > 0$ 
depending on $N,$ $c_1,$ $\upeta_{min},$ $\upgamma,$ and $\updelta,$ such that if $\totalnorm{N} \leq \epsilon''',$ then
the following inequalities hold on the interval $[0,T)$ for the norms and energies defined in 
\eqref{E:mathfrakSNg00} - \eqref{E:totalnorm}, \eqref{E:mathcalEdef}, \eqref{E:g00energydef} - \eqref{E:gtotalenergydef}, \eqref{E:fluidenergydef}, and \eqref{E:totalenergy}:
	
	\begin{subequations}
	\begin{align}
		C^{-1} \big(\| \partial_t v \|_{L^2} + e^{- \Omega} \| \underpartial v \|_{L^2} + C_{(\upgamma)} \| v \|_{L^2} \big) 
			\leq  \mathcal{E}_{(\upgamma,\updelta)}[v,\partial v] & \leq 	C \big(\| \partial_t v \|_{L^2} + e^{- \Omega} \| 
			\underpartial v 
			\|_{L^2} + C_{(\upgamma)} \| v \|_{L^2} \big), \label{E:mathcalEcomparison} \\
		C^{-1} \gzerozeroenergy{N} \leq \gzerozeronorm{N} & \leq C \gzerozeroenergy{N}, 
			\label{E:mathfrakENg00mathfrakSNcomparison} \\
		C^{-1} \gzerostarenergy{N} \leq \gzerostarnorm{N} & \leq C \gzerostarenergy{N}, \\
		C^{-1} \hstarstarenergy{N} \leq \hstarstarnorm{N} & \leq C \hstarstarenergy{N}, 
			\label{E:mathfrakENh**mathfrakSNcomparison} \\
		C^{-1} \genergy{N} \leq \gnorm{N} & \leq C \genergy{N}, \label{E:mathfrakENmathfrakSNcomparison} \\
		C^{-1} \fluidenergy{N} \leq \fluidnorm{N} & \leq C \fluidenergy{N}, \label{E:ENSNcomparison} \\
		C^{-1} \totalenergy{N}  \leq \totalnorm{N} & \leq C \totalenergy{N}.  \label{E:mathcalQNQNcomparison}
	\end{align}
	\end{subequations}
	
\end{proposition}

\begin{proof}
	The inequalities in \eqref{E:mathcalEcomparison} follow from the definition \eqref{E:mathcalEdef} of 
	$\mathcal{E}_{(\gamma,\delta)}[v, \partial v],$ \eqref{E:mathcalEfirstlowerbound}, the definition \eqref{E:totalnorm} of 
	$\totalnorm{N},$ Sobolev embedding, and \eqref{E:gjkuppercomparetostandard}. The inequalities in 
	\eqref{E:mathfrakENg00mathfrakSNcomparison} - \eqref{E:mathfrakENh**mathfrakSNcomparison} then follow from definitions 
	\eqref{E:mathfrakSNg00} - \eqref{E:mathfrakSNh**}, definitions \eqref{E:g00energydef} - \eqref{E:h**energydef}, and 
	\eqref{E:mathcalEcomparison}.

To prove \eqref{E:ENSNcomparison}, for notational convenience, we set $\dot{u}^j \eqdef \partial_{\vec{\alpha}} u^j,$ 
$\dot{P} \eqdef \partial_{\vec{\alpha}}(P - \bar{p}),$ and as in \eqref{E:dot0intermsofdotj}, we define $\dot{u}^0 \eqdef  -\frac{1}{u_0}u_a \dot{u}^a.$ We now recall the definition \eqref{E:EnergyCurrent} of the energy current component $\dot{J}^0$ associated to $\dot{P},$ $\dot{u}^{\mu}:$

\begin{align}
	\dot{J}^0 & \eqdef \frac{u^0}{(1 + \speed^2)P}\dot{P}^2 + 2 \dot{u}^0 \dot{P} 
		+ \frac{(1 + \speed^2)P u^0}{\speed^2} g_{\alpha \beta} \dot{u}^{\alpha} \dot{u}^{\beta}.
\end{align}
Using the fact that $0 < \speed^2 < \frac{1}{3},$ \eqref{E:g00lowerLinfinity}, \eqref{E:g0jlowerLinfinity}, \eqref{E:gjklowerLinfinity}, \eqref{E:u0upperLinfinity}, \eqref{E:u0lowerLinfinity}, \eqref{E:ujlowerLinfinity},
Sobolev embedding, and the algebraic estimate $|2 \dot{u}^0 \dot{P}| \leq \frac{\speed^2}{(1 + \speed^2)P}\dot{P}^2
+ \frac{(1 + \speed^2)P}{\speed^2}\Big(\frac{1}{u_0} u_a \dot{u}^a \Big)^2,$ it follows that

\begin{align} \label{E:Jdot0inequality}
	\dot{J}^0[\partial_{\vec{\alpha}} \mathbf{W}, \partial_{\vec{\alpha}} \mathbf{W}] \eqdef \dot{J}^0  
		& \geq \frac{(u^0 - \speed^2)}{(1 + \speed^2)P} \dot{P}^2 + \frac{(1 + \speed^2)P u^0}{\speed^2} g_{ab}\dot{u}^a \dot{u}^b 
		 \\
	& \ \ - \frac{(1 + \speed^2)P}{\speed^2}\Big(\frac{1}{u_0} u_a \dot{u}^a \Big)^2
		- \frac{(1 + \speed^2)P u^0}{\speed^2}|g_{00}|\Big(\frac{1}{u_0}u_a \dot{u}^a\Big)^2 
		\notag \\
	& \ \ - \frac{2(1 + \speed^2)P u^0}{\speed^2} \Big|\frac{1}{u_0}g_{0a}\dot{u}^a u_b \dot{u}^b \Big| \notag \\
	& \geq C_1 (\dot{P}^2 + e^{2 \Omega} \delta_{ab} \dot{u}^a \dot{u}^b) 
		- C_2 \epsilon^{'''} e^{2(1 - q) \Omega} \delta_{ab} \dot{u}^a \dot{u}^b \notag \\
	& \geq C (\dot{P}^2 + e^{2 \Omega} \delta_{ab} \dot{u}^a \dot{u}^b), \notag
\end{align}
where $\mathbf{W} \eqdef (P - \bar{p}, u^1,u^2,u^3)^T.$ Integrating inequality \eqref{E:Jdot0inequality} over $\mathbb{T}^3,$ it follows that

\begin{align}
	\| \dot{J}^0 \|_{L^1} \geq C \Big\lbrace \| \dot{P} \|_{L^2}^2 + e^{2 \Omega} \sum_{a=1}^3 \| \dot{u}^a \|_{L^2}^2 
		\Big\rbrace.
\end{align}
Similarly, it can be shown that
\begin{align}
	\| \dot{J}^0 \|_{L^1} \leq C \Big\lbrace \| \dot{P} \|_{L^2}^2 + e^{2 \Omega} \sum_{a=1}^3 \| \dot{u}^a 
	\|_{L^2}^2 \Big\rbrace.
\end{align}
Summing over all derivatives $\partial_{\vec{\alpha}} \mathbf{W}$ with $|\vec{\alpha}| \leq N,$ we have that

\begin{align} \label{E:Jdot0toporderderivativesinequality}
	C^{-1} \sum_{|\vec{\alpha}| \leq N} & \Big\lbrace \| \partial_{\vec{\alpha}} (P - \bar{p}) \|_{L^2}^2 
		+ e^{2 \Omega} \sum_{a=1}^3 \| \partial_{\vec{\alpha}} u^a \|_{L^2}^2 \Big\rbrace \\
	& \leq \sum_{|\vec{\alpha}| \leq N} 
		\| \dot{J}^0[\partial_{\vec{\alpha}} \mathbf{W},\partial_{\vec{\alpha}} \mathbf{W}] \|_{L^1} \notag \\
	& \leq C \sum_{|\vec{\alpha}| \leq N} \Big\lbrace \| \partial_{\vec{\alpha}} (P - \bar{p}) \|_{L^2}^2 
		+ e^{2 \Omega} \sum_{a=1}^3 \| \partial_{\vec{\alpha}} u^a \|_{L^2}^2 \Big\rbrace. \notag
\end{align}
Inequality \eqref{E:ENSNcomparison} now follows from \eqref{E:Jdot0toporderderivativesinequality} and the definitions
of $\fluidenergy{N}$ and $\fluidnorm{N}.$

Finally, \eqref{E:mathfrakENmathfrakSNcomparison} and \eqref{E:mathcalQNQNcomparison} follow trivially from definitions \eqref{E:supgnorm}, \eqref{E:totalnorm}, \eqref{E:gtotalenergydef}, and \eqref{E:totalenergy}, and from the previous inequalities.

\end{proof}

\begin{remark}
	The pointwise positivity \eqref{E:Jdot0inequality} of $\dot{J}^0$ is no accident. See \cite{dC2007} and \cite{jS2008b} for 		
	further discussion concerning the coerciveness properties of the current $\dot{J}^{\mu}.$
\end{remark}

\section{Global Existence and Future Causal Geodesic Completeness} \label{S:GlobalExistence}

In this section, we use the estimates derived in Section \ref{S:BootstrapConsequences} to prove our main theorem. More specifically, we show that the modified equations \eqref{E:finalg00equation} - \eqref{E:finalEulerj} have future causally
geodesically complete solutions for initial data near that of the background FLRW solution $(\widetilde{g}_{\mu \nu}, \widetilde{P}, \widetilde{u}^{\mu})$ on $[0,\infty) \times \mathbb{T}^3,$ which is described in Section \ref{S:backgroundsolution}. We emphasize that this aspect of the theorem concerns only the modified equations, and does not necessarily produce a solution to the Euler-Einstein equations \eqref{E:EulerEinstein} + \eqref{E:fluidu} - \eqref{E:fluiduperp}. However, as described in Section \ref{SS:PreservationofHarmonicGauge}, if the Einstein constraint equations and the wave coordinate condition $Q_{\mu} = 0$ are both satisfied along the Cauchy hypersurface $\mathring{\Sigma} = \lbrace x \in \mathcal{M} \ | \ t = 0 \rbrace,$ and the data for the modified equations are constructed from the data for the unmodified Euler-Einstein equations as described in Section \ref{SS:IDREDUCED}, then the solution to the modified equations is also a solution to the Euler-Einstein equations. The main idea of the proof is to show that the energies satisfy a system of integral inequalities that forces them (via Gronwall's inequality) to remain uniformly small on the time interval of existence. By Proposition \ref{P:energynormomparison}, the norms must also remain uniformly small. The continuation principle of Theorem \ref{T:ContinuationCriterion} can then be applied to conclude that the solution exists globally in time.

\subsection{Integral inequalities for the energies} \label{SS:IntegralInequalities}

In this section, we derive the system of integral inequalities that was mentioned in the previous paragraph.

\begin{proposition}[\textbf{Integral Inequalities}] \label{P:IntegralEnergyInequalities}
	Assume that $(g_{\mu \nu}, P, u^j),$ $(\mu,\nu = 0,1,2,3),$ $(j = 1,2,3),$
	is a classical solution to the modified system \eqref{E:finalg00equation} - \eqref{E:finalEulerj} on $[0,T) \times 
	\mathbb{T}^3$ and that the bootstrap assumptions \eqref{E:metricBAeta} - \eqref{E:g0jBALinfinity} hold on $[0,T) \times 
	\mathbb{T}^3.$ Then there exists a constant $\epsilon''''$ with $0 < \epsilon'''' < 1$ and a uniform constant $C$ depending 
	only on $N$ such that if $\totalnorm{N}(t) \leq \epsilon''''$ holds on $[0,T)$ and $t_1 \in [0,T),$ then the following 
	system of integral inequalities is satisfied for $t \in [t_1,T)$ by the energies defined in Section \ref{S:NormsandEnergies}: 	
	\begin{subequations}
	\begin{align}
		U_{N-1}^2(t) & \leq U_{N-1}^2(t_1) \label{E:UNminusoneintegral} \\
			& \ \ + \int_{\tau = t_1}^t \overbrace{2(\decayparameter - 1 + q)}^{< 0} e^{(1 + q) \Omega(\tau)} 
			U_{N-1}^2(\tau) + 
			C \totalenergy{N}(\tau) U_{N-1}(\tau) \, d \tau, \notag \\
		\fluidenergy{N}^2(t) & \leq \fluidenergy{N}^2(t_1) 
			+ C \int_{t_1}^t e^{-q H \tau} \totalenergy{N}^2(\tau) \, d \tau, \label{E:ENintegral} 
		\end{align}
		
		\begin{align}	
		\gzerozeroenergy{N}^2(t) 
			& \leq \gzerozeroenergy{N}^2(t_1) 
			+ \int_{t_1}^t -4qH \gzerozeroenergy{N}^2(\tau) + C  e^{-q H \tau} \totalenergy{N}(\tau) 
			\gzerozeroenergy{N}(\tau) \, d \tau, \label{E:mathfrakENg00integral} \\
		\gzerostarenergy{N}^2(t) & \leq \gzerostarenergy{N}^2(t_1) \label{E:mathfrakENg0*integral} \\
			& \ \ + \int_{t_1}^t -4qH \gzerostarenergy{N}^2(\tau) 
				+ C \hstarstarenergy{N}(\tau)\gzerostarenergy{N}(\tau)
				+ C e^{-q H \tau} \totalenergy{N}(\tau) \gzerostarenergy{N}(\tau) \, d \tau, \notag \\
		\hstarstarenergy{N}^2(t) & \leq \hstarstarenergy{N}^2(t_1)	
			+ \int_{t_1}^t H e^{- q H \tau} \hstarstarenergy{N}^2(\tau)
			+ C e^{-q H \tau}\totalenergy{N}(\tau) \hstarstarenergy{N}(\tau) \, d \tau. \label{E:mathfrakENh**integral}
	\end{align}
	\end{subequations}

\end{proposition}

\begin{proof}
	
	To prove \eqref{E:UNminusoneintegral}, we simply use \eqref{E:Unormtimederivative},
	\eqref{E:triangleprimejHNminusone}, and Proposition \ref{P:energynormomparison} to conclude that
	
	\begin{align} \label{E:Unormtimederivativeagain}
		\frac{d}{dt}\big(U_{N-1}^2 \big) & \leq 2 \overbrace{(\decayparameter - 1 + q)}^{< 0}\omega  e^{(1 + q) \Omega} 
			U_{N-1}^2 + 2 e^{(1 + q) \Omega} U_{N-1} \sum_{a=1}^3 \| \triangle'^a \|_{H^{N-1}}  \\
		& \leq 2 \overbrace{(\decayparameter - 1 + q)}^{< 0} \omega e^{(1 + q) \Omega} 
			U_{N-1}^2 +  C \totalenergy{N} U_{N-1}. \notag
	\end{align}
	Inequality \eqref{E:UNminusoneintegral} now follows from integrating \eqref{E:Unormtimederivativeagain} 
	from $t_1$ to $t.$ 
	
	To prove \eqref{E:ENintegral}, we make use of the following inequality, which we prove below: 
	
	\begin{align} \label{E:divergenceofdotJL1estimate}
		\big\| \partial_{\mu} \big(\dot{J}^{\mu} [\partial_{\vec{\alpha}}\mathbf{W},
			\partial_{\vec{\alpha}} \mathbf{W}] \big) \big\|_{L^1} & \leq C e^{-q \Omega} \totalenergy{N}^2.
	\end{align}
	We also recall inequality \eqref{E:fluidenergytimederivative}:
	
	\begin{align} \label{E:fluidenergytimederivativeagain}
		\frac{d}{dt}\big(\fluidenergy{N}^2 \big) & \leq 
			\sum_{|\vec{\alpha}| \leq N} \big\| \partial_{\mu} \big(\dot{J}^{\mu}[\partial_{\vec{\alpha}} 
			\mathbf{W},\partial_{\vec{\alpha}} \mathbf{W}]\big) \big\|_{L^1}. 
	\end{align}
	Inequality \eqref{E:ENintegral} follows from integrating \eqref{E:fluidenergytimederivativeagain} from $t_1$ to $t$ 
	and using \eqref{E:divergenceofdotJL1estimate}. 
	
	To prove \eqref{E:divergenceofdotJL1estimate}, we first recall equation \eqref{E:divergenceofdotJ},
	where $\dot{\mathbf{W}} = (\dot{P}, \dot{u}^1, \dot{u}^2, \dot{u}^3)^T 
	\eqdef \partial_{\vec{\alpha}}\mathbf{W} = (\partial_{\vec{\alpha}} (P - \bar{p}), 
	\partial_{\vec{\alpha}} u^1,\partial_{\vec{\alpha}} u^2,\partial_{\vec{\alpha}} u^3)^T,$ and as in 
	\eqref{E:dot0intermsofdotj}, $\dot{u}^0 \eqdef -\frac{1}{u_0}u_a \dot{u}^a:$
	
	\begin{align} \label{E:divergenceofdotJagain}
	\partial_{\mu} \big( \dot{J}^{\mu}[\dot{\mathbf{W}}, \dot{\mathbf{W}}] \big) 
		& = \bigg(\partial_t \Big[\frac{u^0}{(1 + \speed^2)P} \Big] \bigg)\dot{P}^2 
		+	\bigg(\partial_a \Big[\frac{u^a}{(1 + \speed^2)P} \Big] \bigg)\dot{P}^2  \\
	& \ \ + \bigg( \partial_t \Big[\frac{(1 + \speed^2)P u^0}{\speed^2} \Big] \bigg) 
		\bigg(g_{00} (\dot{u}^{0})^2 + 2 g_{0a} \dot{u}^{0} \dot{u}^{a} 
		+ g_{ab} \dot{u}^{a} \dot{u}^{b} \bigg)		\notag \\
	& \ \ + \bigg( \partial_a \Big[\frac{(1 + \speed^2)P u^a}{\speed^2} \Big] \bigg) 
		\bigg(g_{00} (\dot{u}^{0})^2 + 2 g_{0a} \dot{u}^{0} \dot{u}^{a} 
		+ g_{ab} \dot{u}^{a} \dot{u}^{b} \bigg) \notag  \\
	& \ \ + \frac{(1 + \speed^2)P u^a}{\speed^2} (\partial_a g_{00}) (\dot{u}^{0})^2
		+ 2 \frac{(1 + \speed^2)P u^a}{\speed^2} (\partial_a g_{0b}) \dot{u}^{0} \dot{u}^{b} 
		\notag \\
	& \ \ + \frac{(1 + \speed^2)P u^a}{\speed^2} (\partial_a g_{lm}) \dot{u}^{l} \dot{u}^{m}
		+ \frac{(1 + \speed^2)P u^0}{\speed^2} (\partial_t g_{00}) (\dot{u}^0) ^2 \notag \\
	& \ \ + \frac{2(1 + \speed^2)P u^0}{\speed^2} (\partial_t g_{0 a}) \dot{u}^0 \dot{u}^a
		+ \frac{(1 + \speed^2)P (u^0 - 1)}{\speed^2} (\partial_t g_{ab}) \dot{u}^a \dot{u}^b \notag \\
	& \ \ - 2 \bigg(\partial_t \Big[ \frac{u_a}{u_0} \Big] \bigg)\dot{u}^a \dot{P}
		+ \frac{(1+\speed^2)P}{\speed^2} (\partial_t g_{ab} - 2 \omega g_{ab})\dot{u}^a \dot{u}^b 
		 \notag \\
	& \ \ + \underbrace{\frac{2(1+\speed^2)P}{\speed^2}(\decayparameter - 1 ) \omega g_{ab}\dot{u}^a \dot{u}^b}_{< 0} 
		+ \frac{2\mathfrak{F}}{(1 + \speed^2)P} \dot{P} \notag \\ 
	& \ \ + \frac{2(1+\speed^2)P}{\speed^2} g_{00} \mathfrak{G}^0 \dot{u}^0
		+ \frac{4(1+\speed^2)P}{\speed^2} g_{0a} \mathfrak{G}^0 \dot{u}^a 
		+ \frac{2(1+\speed^2)P}{\speed^2} g_{ab} \mathfrak{G}^a \dot{u}^b. \notag
\end{align}
In the above expression, the terms $\mathfrak{F} \eqdef \mathfrak{F}_{\vec{\alpha}}$ and $\mathfrak{G} \eqdef \mathfrak{G}_{\vec{\alpha}}^{\mu},$ are given by \eqref{E:FalphamathfrakGjalphainhomogeneousterms} and \eqref{E:mathfrakG0alphainhomogeneousterm}. The estimate \eqref{E:divergenceofdotJL1estimate} now follows from Proposition \ref{P:energynormomparison} and \eqref{E:divergenceofdotJagain} using the following four steps: (i) we drop the negative (since $\decayparameter \eqdef 3 \speed^2 < 1$ and $g_{ab}$ is positive definite) fifth-from-last term on the right-hand side of \eqref{E:divergenceofdotJagain}; (ii) we bound each variation $\dot{u}^{\mu},$ $\mathfrak{F}_{\vec{\alpha}},$ and $\mathfrak{G}_{\vec{\alpha}}^{\mu}$ in $L^2,$ and make use of \eqref{E:dotu0L2intermsofdotuaL2}, \eqref{E:mathfrakFalphamathfrakGjalphaL2}, and \eqref{E:mathfrakG0alphaL2}; (iii) we bound all of the remaining terms in $L^{\infty},$ and use Sobolev embedding together with the estimates \eqref{E:partialtgjkminusomegagjklowerLinfinity}, \eqref{E:partialtgjklowerC1}, \eqref{E:u0upperLinfinity}, \eqref{E:u0lowerLinfinity},
\eqref{E:partialtPLinfinity}, \eqref{E:partialtu0upperLinfinity}, \eqref{E:partialtujupperLinfinity}, \eqref{E:partialtu0lowerLinfinity}, and \eqref{E:partialtujlowerLinfinity}; (iv) we make repeated use of the estimate
$\| v_1 v_2 v_3 \|_{L^1} \leq \| v_1 \|_{L^{\infty}} \| v_2 \|_{L^2} \| v_3 \|_{L^2},$ where $v_2$ and $v_3$ are terms estimated in step ii), and $v_1$ is estimated in step iii).

To prove \eqref{E:mathfrakENg0*integral}, we apply Corollary \ref{C:metricfirstdiferentialenergyinequality},
using \eqref{E:gabupperGammaajblowerHN}, \eqref{E:triangle0jHN}, \eqref{E:g0jcommutatorL2},  \eqref{E:triangleEgamma0jdelta0*L1}, and Proposition \ref{P:energynormomparison} to estimate the terms on the right-hand side of \eqref{E:underlinemathfrakEg0*firstdifferential}, and using 
definition \eqref{E:qdef} to deduce that $2(q-1) - \upeta_{0*} \leq - 4q,$ thereby arriving at the following inequality:

\begin{align} \label{E:gzerostardifferentialenergybound}
	\frac{d}{dt}\Big(\gzerostarenergy{N}^2 \Big)
	& \leq -4qH	\gzerostarenergy{N}^2 + 
		C \hstarstarenergy{N} \gzerostarenergy{N} + C e^{-q H \tau} \totalenergy{N} \gzerostarenergy{N}.
\end{align}
Inequality \eqref{E:mathfrakENg0*integral} now follows from integrating \eqref{E:gzerostardifferentialenergybound} in time. Inequalities \eqref{E:mathfrakENg00integral} and \eqref{E:mathfrakENh**integral} can be proved similarly; we omit the details.
\end{proof}

\begin{remark} \label{R:Dangerousterm}
	The term $C \hstarstarenergy{N} \gzerostarenergy{N}$
	in inequalities \eqref{E:mathfrakENg0*integral} and \eqref{E:gzerostardifferentialenergybound} arises from the 
	$\gzerostarnorm{N} \sum_{j=1}^3 e^{(q-1) \Omega} \| g^{ab}\Gamma_{a j b} \|_{H^N}$ term on the right-hand side of 
	\eqref{E:underlinemathfrakEg0*firstdifferential}. 
	This term is dangerous in the sense that it does not contain an exponentially decaying factor, and looks like it could lead 
	to the growth of $\gzerostarenergy{N}.$ However, as we shall see in 
	the proof of Theorem \ref{T:GlobalExistence}, there is a partial decoupling in the integral inequalities in the sense that
	the $C \hstarstarenergy{N}$ factor in the dangerous term can be controlled independently using inequality 
	\eqref{E:mathfrakENh**integral} alone. We will then insert this information into inequality \eqref{E:mathfrakENg0*integral}, 
	and also make use of the negative term $-4qH \gzerostarenergy{N}^2$ to obtain a bound for $\gzerostarenergy{N}.$
\end{remark}

For completeness, we state the following version of Gronwall's inequality; we omit the simple proof. We will use it in Section \ref{SS:globalexistencetheorem}.

\begin{lemma} \label{L:Gronwall}
	Let $b(t) \geq 0$ be a continuous function on the interval $[t_1,T],$ and let $B(t)$ be an anti-derivative of $b(t).$
	Suppose that $A \geq 0$ and that $y(t) \geq 0$ is a continuous function satisfying the inequality
	\begin{align}
		y(t) \leq A + \int_{t_1}^t b(\tau)y(\tau) \, d \tau
	\end{align}
	for $t \in [t_1,T].$ Then for $t \in [t_1,T],$
	we have
	
	\begin{align}
		y(t) \leq A \exp \Big[B(t) - B(t_1)\Big].
	\end{align}
	
\end{lemma}

\hfill $\qed$

In addition, in Section \ref{SS:globalexistencetheorem}, we will apply the following integral inequality to \eqref{E:mathfrakENg0*integral} in order to estimate the energy $\gzerostarenergy{N}(t).$ 

\begin{lemma} \cite[Lemma 11.1.3]{iRjS2009} \label{L:integralinequality}
	Let $b(t) > 0$ be a continuous \textbf{non-decreasing} function on the interval $[0,T],$ and let $\epsilon > 0.$
	Suppose that for each $t_1 \in [0,T],$ $y(t) \geq 0$ is a continuous function satisfying the inequality
	
	\begin{align} \label{E:integralinequality}
		y^2(t) \leq y^2(t_1) + \int_{\tau = t_1}^t - b(\tau)y^2(\tau) + \epsilon y(\tau) \, d \tau 
	\end{align}
	for $t \in [t_1,T].$ Then for any $t_1, t \in [0,T]$ with $t_1 \leq t,$
	we have that
	
	\begin{align} \label{E:integralinequalityconclusion}
		y(t) \leq y(t_1) + \frac{\epsilon}{b(t_1)}.
	\end{align}
	
\end{lemma}

\hfill $\qed$

\begin{proof}
	Let $\mathcal{C}$ be the ``highest'' curve in the $(t,y)$ plane on which the integrand in \eqref{E:integralinequality} 
	vanishes; i.e. $\mathcal{C} = \lbrace (t,y) | y = \frac{\epsilon}{b(t)} \rbrace.$  Then by 
	\eqref{E:integralinequality},
	above $\mathcal{C}$ (i.e. for larger $y$ values), $y(t)$ is \emph{strictly} decreasing. Let $y(t)$ achieve
	its maximum at $t_{max} \in [t_1,T].$ We separate the proof of \eqref{E:integralinequalityconclusion} into two cases. 
	Case i) assume that $t_{max} = t_1.$ Then $y(t) \leq y(t_{max}) = y(t_1)$ for $t \in 
	[t_1, T],$ which implies \eqref{E:integralinequalityconclusion}. Case ii) assume that $t_{max} \in (t_1,T].$ We claim that 
	$y(t_{max}) \leq \frac{\epsilon}{b(t_{max})}.$ For otherwise, the point 
	$\big(t_{max}, y(t_{max})\big)$ lies above $\mathcal{C}.$ Since $y(t)$ is then strictly decreasing in a neighborhood of 
	$t_{max},$ it follows that there are times $t_* < t_{max},$ with $t_* \in (t_1,T),$ at which $y(t_*) < y(t_{max}).$ This 
	contradicts the 
	definition of $t_{max}.$ Using also the fact that $\frac{1}{b(t)}$ is non-increasing, it follows that $y(t) 
	\leq y(t_{max}) \leq \frac{\epsilon}{b(t_{max})} \leq \frac{\epsilon}{b(t_1)};$ this concludes the proof of 
	\eqref{E:integralinequalityconclusion}.
\end{proof}

\subsection{The global existence theorem} \label{SS:globalexistencetheorem}

In this section, we state and prove our main theorem, which provides global existence criteria for the 
modified equations \eqref{E:finalg00equation} - \eqref{E:finalEulerj}, and future-stability criteria for the
unmodified Euler-Einstein equations \eqref{E:EulerEinstein} $+$ \eqref{E:fluidu} - \eqref{E:fluiduperp}.

\begin{theorem}[\textbf{Future-Stability of the FLRW Family}] \label{T:GlobalExistence}
Assume that $0 < \speed^2 < 1/3$ and $N \geq 3.$ Let $\mathring{g}_{\mu \nu} = g_{\mu \nu}|_{t=0},$ $2 \mathring{K}_{\mu \nu} = \partial_t g_{\mu \nu}|_{t=0},$ $(\mu, \nu = 0,1,2,3),$ $\mathring{P} = P|_{t=0} = p|_{t=0},$ $\mathring{u}^j = u^j|_{t=0},$ $(j=1,2,3),$ $g_{\alpha \beta}u^{\alpha}u^{\beta}|_{t=0} = -1,$ be initial data (not necessarily satisfying the wave coordinate condition or the Einstein constraints) on the manifold $\mathbb{T}^3$ for the modified Euler-Einstein system \eqref{E:finalg00equation} - \eqref{E:finalEulerj}, and let $\totalnorm{N} \eqdef \gnorm{N} + \fluidnorm{N} + U_{N-1}$ be the norm defined in \eqref{E:totalnorm}. Assume that there is a constant $c_1 \geq 2$ such that 

\begin{align} \label{E:mathringgjklowerequivalenttostandardmetric}
	\frac{2}{c_1} \delta_{ab} X^a X^b & \leq \mathring{g}_{ab}X^a X^b \leq \frac{c_1}{2} \delta_{ab} X^a X^b,&&
		\forall(X^1,X^2,X^3) \in \mathbb{R}^3.
\end{align}

Then there exist a small constant $\epsilon_0$ with $0 < \epsilon_0 < 1$ 
and a large constant $C_*$ such that if $\epsilon \leq \epsilon_0$ and $\totalnorm{N}(0) \leq C_*^{-1} \epsilon,$ then the classical solution $(g_{\mu \nu}, P=e^{3(1+\speed^2)\Omega}p, u^{\mu})$ provided by Theorem \ref{T:LocalExistence} exists on $\mathcal{M} \eqdef [0,\infty) \times \mathbb{T}^3,$ and

\begin{align} \label{E:QNgloballessthanepsilon}
	\totalnorm{N}(t) & \leq \epsilon 
\end{align}
holds for all $t \geq 0.$ Furthermore, the time $T_{max}$ from the hypotheses of Theorem \ref{T:ContinuationCriterion} is infinite, and the spacetime-with-boundary $(\mathcal{M}, g_{\mu \nu})$ is future causally geodesically complete.

Finally, if the initial data $(\mathring{g}_{\mu \nu}, \mathring{K}_{\mu \nu}, \mathring{P}= \mathring{p}, \mathring{u}^j)$
for the modified system are constructed from initial data $(\mathring{\underline{g}}_{jk}, \mathring{\underline{K}}_{jk}, \mathring{p}, \underline{\mathring{u}}^j),$ $(j,k = 1,2,3),$ for the unmodified Euler-Einstein equations \eqref{E:EulerEinstein} $+$ \eqref{E:fluidu} - \eqref{E:fluiduperp} as described in Section \ref{SS:IDREDUCED},
then the solution $(g_{\mu \nu}, p = e^{-3(1+\speed^2)\Omega}P, u^{\mu})$ to the modified system is also a future causally geodesically complete solution to the unmodified equations.

\end{theorem}

\begin{remark}
	It is possible to restate the stability criteria in terms of quantities that manifestly
	depend only on the closeness of the initial data $(\mathbb{T}^3, \mathring{\underline{g}}_{jk}, 
	\mathring{\underline{K}}_{jk}, \mathring{p}, \underline{\mathring{u}}^j)$ for the unmodified system to
	the corresponding data for the FLRW background solution $(\widetilde{g}, \widetilde{p}, \widetilde{u}).$   
	For example, a sufficient condition for global existence and future causal geodesic completeness would be
	
	\begin{align} \label{E:AbstractDataSmall}
		\sum_{j,k=1}^3 \| \mathring{\underline{g}}_{jk} & - \delta_{jk} \|_{H^{N+1}}
			+ \sum_{j,k=1}^3 \| \mathring{\underline{K}}_{jk} - \omega(0) \delta_{jk} \|_{H^N} \\
		& \ \ + \| \mathring{p} - \bar{p} \|_{H^N} 
			+ \sum_{j=1}^3 \| \underline{\mathring{u}}^j \|_{H^N} \leq \epsilon, \notag 
	\end{align}
	where $\epsilon$ is sufficiently small. This is because the condition 
	\eqref{E:AbstractDataSmall} implies that $\totalnorm{N}(0) \leq C \epsilon$
	and furthermore (by Sobolev embedding) that a condition of the form \eqref{E:mathringgjklowerequivalenttostandardmetric} holds
	(i.e. that the hypotheses of Theorem \ref{T:GlobalExistence} hold).
	
	To see that $\totalnorm{N}(0) \leq C \epsilon$ follows from \eqref{E:AbstractDataSmall}, we first 
	use the definition \eqref{E:totalnorm} of $\totalnorm{N}(0),$ 
	the construction of the modified data described in Section \ref{SS:IDREDUCED},
	and the triangle inequality to deduce that
	
	\begin{align} \label{E:TotalInitialNormInTermsofAbstract}
		\totalnorm{N}(0) & \leq 2 \| 3 \omega(0)  - \mathring{\underline{g}}^{ab} \mathring{\underline{K}}_{ab} \|_{H^N}
			+ \sum_{j,k=1}^3 \| \underpartial \mathring{\underline{g}}_{jk} \|_{H^N}
			+ 2 \sum_{j,k=1}^3 \| \omega(0)\mathring{\underline{g}}_{jk} - \mathring{\underline{K}}_{jk} \|_{H^N} \\
		& \ \ + \sum_{j=1}^3 \| \mathring{\underline{g}}^{ab}\big(\partial_a \mathring{\underline{g}}_{bj} - \frac{1}{2} \partial_j 
			\mathring{\underline{g}}_{ab}\big) \|_{H^N}
			+ \| \mathring{p} - \bar{p} \|_{H^N} + \sum_{j=1}^3 \| \underline{\mathring{u}}^j \|_{H^N} \notag \\
		& \leq 2 \| \big(\mathring{\underline{g}}^{ab} - \delta^{ab} \big)\mathring{\underline{K}}_{ab} \|_{H^{N}}
			+ 2 \| \delta^{ab} \big(\mathring{\underline{K}}_{ab} - \omega(0) \delta_{ab}\big) \|_{H^{N}} 
			\notag \\
		& \ \ + \sum_{j,k=1}^3 \| \underpartial \mathring{\underline{g}}_{jk} \|_{H^N}
			+ 2 \omega(0) \sum_{j,k=1}^3 \| \mathring{\underline{g}}_{jk} - \delta_{jk} \|_{H^N} 
			 \notag \\
		& \ \ + 2 \sum_{j,k=1}^3 \|\mathring{\underline{K}}_{jk} - \omega(0) \delta_{jk} \|_{H^N}
			+ \sum_{j=1}^3 \| \mathring{\underline{g}}^{ab}\big(\partial_a \mathring{\underline{g}}_{bj} - \frac{1}{2} \partial_j 
			\mathring{\underline{g}}_{ab}\big) \|_{H^N} \notag \\
		& \ \ + \| \mathring{p} - \bar{p} \|_{H^N} + \sum_{j=1}^3 \| \underline{\mathring{u}}^j \|_{H^N}. \notag
	\end{align}
	Now using Corollary \ref{C:SobolevTaylor}, Proposition \ref{P:F1FkLinfinityHN}, and Sobolev embedding,
	it follows that if \eqref{E:AbstractDataSmall} holds and if $\epsilon$ is sufficiently small, 
	then the right-hand side of \eqref{E:TotalInitialNormInTermsofAbstract} is $\leq C \epsilon.$
\end{remark}

\begin{proof}
	See Remark \ref{R:ProofsRemark} for some conventions that we use throughout this proof. We only discuss the issue of global 
	existence and obtaining the uniform bound $\totalnorm{N}(t) \leq \epsilon$ for $t \in [0,\infty).$ 
	The future causal geodesic completeness of the resulting spacetime-with-boundary follows from this bound via the arguments 
	given in \cite[Propositions 3 and 4]{hR2008} and \cite[Theorem 11.2]{iRjS2009}.
	Our global existence proof relies upon a standard bootstrap-style argument that ultimately relies on Theorem 
	\ref{T:ContinuationCriterion}; i.e., we will make assumptions concerning the size of the energies and concerning 
	$g_{\mu \nu},$ and we will use these assumptions, together with assumptions on the data, to deduce an improvement. 
	In effect, we will avoid the four breakdown possibilities of Theorem \ref{T:ContinuationCriterion},
	which will allow us to conclude global existence.
	
	To begin the analysis, we invoke Theorem \ref{T:LocalExistence}, which shows that if $\epsilon$ is small enough and 
	$\totalnorm{N}(0) < \epsilon,$ then there is a local solution $(g_{\mu \nu}, P, u^{\mu})$ existing on a (non-trivial)
	maximal interval $[0,T)$ on which the following bootstrap assumptions hold:
	
	\begin{align} 
		\totalnorm{N}(t) & \leq \epsilon, \label{E:proofbootstrapQN} \\
		|g_{00} + 1| & \leq \upeta_{min}, \label{E:g00plusoneproofBAeta} \\
		c_1^{-1} \delta_{ab}X^{a}X^{b} & \leq e^{-2 \Omega} g_{ab} X^{a}X^{b} 
			\leq c_1 \delta_{ab}X^{a}X^{b},&& \forall (X^1,X^2,X^3) \in \mathbb{R}^3, 
			\label{E:gjproofBAkvsstandardmetric} \\
			\sum_{a=1}^3 |g_{0a}|^2 & \leq \upeta_{min} c_1^{-1} e^{2(1 - q) \Omega}, \label{E:g0jproofBALinfinity} \\
			P > 0. \label{E:Ppositive}
		\end{align}
	Observe that the rough bootstrap assumptions \eqref{E:metricBAeta} - \eqref{E:g0jBALinfinity} are included in
	the above assumptions. By maximal interval, we mean that 
	\begin{align}
		T \eqdef \sup \big\lbrace t \geq 0 \ | \ \mbox{The solution exists on} \ [0,t] \times \mathbb{T}^3, \ \mbox{and} \ 
		\eqref{E:proofbootstrapQN} - \eqref{E:Ppositive} \ \mbox{hold} \big\rbrace.
	\end{align}
	We may assume that $T < \infty,$ since otherwise the theorem follows. The remainder of this proof is dedicated to reaching a 
	contradiction if $\epsilon$ is small enough and $C_*$ is large enough. 
	
	Our first assumption, which we use repeatedly throughout this proof, is that $\epsilon$ is small enough so that 
	Propositions \ref{P:BoostrapConsequences}, \ref{P:Nonlinearities}, and \ref{P:energynormomparison}, are valid on $[0,T).$
	We will make repeated use of Proposition \ref{P:energynormomparison} throughout this proof without explicitly mentioning it
	each time.
	
	We now address the bootstrap assumptions \eqref{E:g00plusoneproofBAeta} - \eqref{E:Ppositive}, 
	with the intent of showing an improvement. First, we note that the assumption $\totalnorm{N} \leq \epsilon$ implies that
	
	\begin{align} \label{E:partialthnksmall}
		\| \partial_t (e^{-2 \Omega}g_{jk}) \|_{L^{\infty}}= \| \partial_t h_{jk} \|_{L^{\infty}} & \leq C \epsilon e^{-qHt}.
	\end{align}
	We then use \eqref{E:partialthnksmall} to integrate in time from $t=0,$ concluding that
	
	\begin{align} \label{E:partialthnksmallintegrated}
		\| e^{-2 \Omega} g_{jk}(t,\cdot) - \mathring{g}_{jk}(\cdot) \|_{L^{\infty}}
		\leq C \epsilon. 
	\end{align}
	By \eqref{E:mathringgjklowerequivalenttostandardmetric} and \eqref{E:partialthnksmallintegrated}, 
	it follows that if $\epsilon$ is small enough, then on $[0,T) \times \mathbb{T}^3,$ we have that
	
	\begin{align} \label{E:gabequivalentbootstrapimprovement}
		\frac{3}{2c_1} \delta_{ab} X^a X^b & \leq e^{-2 \Omega} g_{ab}X^a X^b \leq \frac{2c_1}{3} \delta_{ab} X^a X^b, &&
			\forall (X^1,X^2,X^3) \in \mathbb{R}^3.
	\end{align}
	Furthermore, by the definition \eqref{E:totalnorm} of $\totalnorm{N},$ Sobolev embedding, and 
	\eqref{E:proofbootstrapQN}, the following inequalities hold on $[0,T):$
	
	\begin{align}
		\| g_{00} + 1 \|_{L^{\infty}} & \leq C \epsilon, \label{E:g00bootstrapimprovement} \\
		\| g_{0j} \|_{L^{\infty}} & \leq C \epsilon e^{(1-q)Ht}, \label{E:g0jbootstrapimprovement} \\
		\| P - \bar{p} \|_{L^{\infty}} & \leq C \epsilon e^{-qHt}. \label{E:Pbootstrapimprovement}
	\end{align}
	Inequalities \eqref{E:gabequivalentbootstrapimprovement} - \eqref{E:Pbootstrapimprovement} 
	show that if $\epsilon$ is small enough, then the bootstrap assumptions \eqref{E:g00plusoneproofBAeta} - 
	\eqref{E:Ppositive} can be \emph{strictly improved} on the interval $[0,T).$ 
	
	To complete our proof of the theorem, we will show that if $\epsilon$ is small enough and $C_*$ is large enough, then the 
	bootstrap assumption \eqref{E:proofbootstrapQN} can be improved by replacing $\epsilon$ with $\epsilon/2;$ the primary tool 
	for deducing an improvement is Proposition \ref{P:IntegralEnergyInequalities}. Throughout the remainder of the proof, we set 
	$\totalnorm{N}(0) \eqdef \mathring{\epsilon}.$ To begin our proof of an improvement 
	of \eqref{E:proofbootstrapQN}, we use a very non-optimal application of Proposition 
	\ref{P:IntegralEnergyInequalities} with $t_1 = 0,$ deducing (with the help of Proposition \ref{P:energynormomparison})
	that on $[0,T),$ we have that
	
	\begin{align} \label{E:mathcalQNGronwallready}
		\totalnorm{N}^2(t) & \leq C \totalnorm{N}^2(0) + \int_{\tau=0}^{t} c \totalnorm{N}^2(\tau) \, d\tau.
	\end{align}
	Applying Lemma \ref{L:Gronwall} (Gronwall's inequality) to \eqref{E:mathcalQNGronwallready}, using
	$\totalnorm{N}(0) = \mathring{\epsilon},$ and using Proposition \ref{P:energynormomparison}, we 
	conclude that the following weak inequalities hold on $[0,T):$
	
	\begin{align} 
		\totalnorm{N}(t) & \leq C \mathring{\epsilon} e^{ct}, \label{E:weaknorminequality} \\
		\totalenergy{N}(t) & \leq C \mathring{\epsilon} e^{ct}. \label{E:weakenergyinequality}
	\end{align}
	
	\begin{remark} \label{R:weaklifespan}
		The weak inequality \eqref{E:weaknorminequality} already shows that the time of existence is at least of 
		order $c^{-1} |\mbox{ln}(C \mathring{\epsilon})|,$ if $\mathring{\epsilon}$ is sufficiently small.
	\end{remark}

	We now fix a time $t_1 \in [0,T);$ $t_1$ will be adjusted at the end of the proof. Roughly speaking, it will play the role of
	a time that is large enough so that the exponentially damped terms on the right-hand sides of the inequalities of
	Proposition \ref{P:IntegralEnergyInequalities} are of size $\ll \epsilon.$
	
	We now estimate the energy $\gzerozeroenergy{N}(t).$ To begin, we drop 
	the negative term $-4qH \gzerozeroenergy{N}^2(\tau)$ on the right-hand side of inequality \eqref{E:mathfrakENg00integral}
	and use the bootstrap assumption \eqref{E:proofbootstrapQN} to deduce that
	
	\begin{align} \label{E:mathfrakENg00differentialtwoterms}
			\gzerozeroenergy{N}^2(t) & \leq \gzerozeroenergy{N}^2(t_1)
			+ C \epsilon^2 \int_{t_1}^t e^{-q H \tau} \, d \tau.
	\end{align}
	Using \eqref{E:weakenergyinequality} at time $t_1$ and performing the integration on the right-hand side of 
	\eqref{E:mathfrakENg00differentialtwoterms}, we have the following inequality for $t \in [t_1,T):$
	
	\begin{align} \label{E:mathfrakEg00Gronwallready}
			\gzerozeroenergy{N}(t) \leq C \lbrace \mathring{\epsilon} e^{ct_1} + \epsilon e^{-qHt_1/2} \rbrace.
	\end{align}
	Also using \eqref{E:weakenergyinequality} to obtain an upper bound
	for $\gzerozeroenergy{N}(t)$ on $[0,t_1],$ we have established the following inequality, which is valid 
	for $t \in [0,T):$
	
	\begin{align} \label{E:mathfrakEg00Gronwall}
		\gzerozeroenergy{N}(t) & \leq C \lbrace \mathring{\epsilon} e^{ct_1} + \epsilon e^{-qHt_1 /2} \rbrace.
	\end{align}

	To estimate $\hstarstarenergy{N},$ we can appeal to inequality \eqref{E:mathfrakENh**integral} argue as we did for 
	$\gzerozeroenergy{N},$ thus obtaining that the following inequality holds for $t \in [0,T):$
	
	\begin{align} \label{E:mathfrakEh**Gronwall}
		\hstarstarenergy{N}(t) & \leq C \lbrace \mathring{\epsilon} e^{ct_1} + \epsilon e^{-qHt_1/2} \rbrace. 
	\end{align}

	To estimate $\gzerostarenergy{N}(t),$ we use \eqref{E:mathfrakENg0*integral}, the bootstrap assumption 
	\eqref{E:proofbootstrapQN}, and \eqref{E:mathfrakEh**Gronwall} to arrive at the 
	following inequality valid for $t \in [t_1,T)$:
	
	\begin{align} \label{E:mathfrakE0*Gronwallready}
		\gzerostarenergy{N}^2(t) & \leq 
			\gzerostarenergy{N}^2(t_1) 
			+ \int_{\tau = t_1}^t -4qH \gzerostarenergy{N}^2(\tau) + C \lbrace \mathring{\epsilon} e^{ct_1} 
			+ \epsilon e^{-qHt_1 /2} \rbrace \gzerostarenergy{N}(\tau)  \, d \tau. 
	\end{align}
	Applying Lemma \ref{L:integralinequality} to \eqref{E:mathfrakE0*Gronwallready}, with 	
	$y(t) = \gzerostarenergy{N}(t)$ and $b(t) = 4 q H$ in the lemma, and also using
	\eqref{E:weakenergyinequality} at time $t_1,$ we conclude that the following inequality holds on $[0,T):$
	
	\begin{align} \label{E:mathfrakE0*Gronwall}
		\gzerostarenergy{N}(t) & \leq C \lbrace \mathring{\epsilon} e^{ct_1} + \epsilon e^{-qHt_1/2} \rbrace. 
	\end{align}

	To estimate $U_{N-1}(t)$ on $[t_1,T),$ we first use \eqref{E:UNminusoneintegral} and \eqref{E:proofbootstrapQN} to 
	deduce that for $t \in [t_1,T),$ we have
	
	\begin{align} \label{E:UNminusoneGronwallready}
		U_{N-1}^2(t) & \leq U_{N-1}^2(t_1) + \int_{\tau = t_1}^t   
			\overbrace{2(\decayparameter - 1 + q)}^{< 0} e^{(1 + q) H \tau} U_{N-1}^2(\tau) + 
			C \epsilon U_{N-1}(\tau) \, d \tau.
	\end{align} 
	Applying Lemma \ref{L:integralinequality}, to \eqref{E:UNminusoneGronwallready}, with 
	$y(t) = U_{N-1}(t)$ and $b(t) = |2(\decayparameter - 1 + q)|e^{(1 + q) H t}$ in the lemma, and also using 
	\eqref{E:weakenergyinequality}, it follows that on $[0,T),$ we have
	
	\begin{align} \label{E:UNminusoneGronwall}
		U_{N-1}(t) & \leq C \lbrace \mathring{\epsilon} e^{ct_1} 
			+ \epsilon e^{-qHt_1 /2} \rbrace.
	\end{align}

	To estimate $\fluidenergy{N}(t)$ on $[t_1,T),$ we simply use \eqref{E:proofbootstrapQN} and 
	\eqref{E:weakenergyinequality} to estimate the two terms on the right-hand side of \eqref{E:ENintegral}:
	
	\begin{align} \label{E:ENdifferentialtwoterms}
		\fluidenergy{N}^2(t) \leq \fluidenergy{N}^2(t_1) + C \epsilon^2 \int_{\tau=t_1}^{t} e^{-q H \tau}  \, d\tau
			\leq C \lbrace \mathring{\epsilon} e^{ct_1} + \epsilon e^{-qHt_1/2} \rbrace^2. 
	\end{align}
	Also using \eqref{E:weakenergyinequality} to estimate $\fluidenergy{N}(t)$ on $[0,t_1],$ 
	we conclude that the following inequality is valid on $[0,T):$
	
	\begin{align} \label{E:ENGronwall}
		\fluidenergy{N}(t) & \leq C \lbrace \mathring{\epsilon} e^{ct_1} + \epsilon e^{-qHt_1/2} \rbrace.
	\end{align}

	Adding \eqref{E:mathfrakEg00Gronwall}, \eqref{E:mathfrakEh**Gronwall}, 
	\eqref{E:mathfrakE0*Gronwall}, \eqref{E:UNminusoneGronwall}, and \eqref{E:ENGronwall}, referring to definition 
	\eqref{E:totalnorm}, and using Proposition \ref{P:energynormomparison}, it follows that on $[0,T)$
	
	\begin{align} \label{E:QNGronwall}
		\totalnorm{N}(t) & \leq C \lbrace \mathring{\epsilon} e^{ct_1} + \epsilon e^{-qHt_1/2} \rbrace.
	\end{align}
	We now choose $t_1$ such that $C e^{-qHt_1/2} < \frac{1}{4},$ and $\mathring{\epsilon}$ such that
	$\mathring{\epsilon} C e^{c t_1} \leq \frac{1}{4} \epsilon,$ where the constant $C$ is from the right-hand
	side of \eqref{E:QNGronwall}. This implies (with the help of Proposition \ref{P:energynormomparison}) that on $[0,T),$
	we have that
	
	\begin{align}
		\totalnorm{N}(t) & \leq \frac{1}{2} \epsilon. \label{E:QNbootstrapimprovement}
	\end{align}
	We remark that in order to guarantee that the solution exists long enough (i.e. that $T$ is large enough) 
	so that $t_1 \in [0,T),$ we may have to further shrink $\mathring{\epsilon};$ see Remark \ref{R:weaklifespan}. We also
	remark that by the above reasoning, it follows that the constant $C_*$ from the conclusions of the theorem can be chosen to 
	be $4 C e^{c t_1},$ where $C$ is from the right-hand side of \eqref{E:QNGronwall}.
	
	Combining \eqref{E:gabequivalentbootstrapimprovement} - \eqref{E:Pbootstrapimprovement}
	and \eqref{E:QNbootstrapimprovement}, and using Sobolev embedding, it now follows that none of the four existence-breakdown
	scenarios stated in the conclusions of Theorem \ref{T:ContinuationCriterion} occur. Using  
	the continuity of $\totalnorm{N}(t),$ it also follows that the solution can be extended to an interval $[0, T + \delta]$ 
	on which the bootstrap assumptions \eqref{E:proofbootstrapQN} - \eqref{E:Ppositive} hold. This contradicts the 
	maximality of $T$ and completes the proof of the theorem. 

\end{proof}

\section{Asymptotics} \label{S:Asymptotics}

In this section, we provide a theorem that strengthens the conclusions of Theorem \ref{T:GlobalExistence}. More specifically, we
show that $g_{\mu \nu},$ $g^{\mu \nu},$ $P - \bar{p},$ $u^{\mu},$ $(\mu, \nu = 0,1,2,3),$ and various coordinate derivatives of these quantities converge as $t \rightarrow \infty.$ Furthermore, some of the decay rates (for example, those in \eqref{E:g0jlowerconvergence} and \eqref{E:partialtg0jconvergence}) are a significant improvement compared to the rates that can be directly obtained from the bound $\totalnorm{N} \leq \epsilon,$ which was derived in Theorem \ref{T:GlobalExistence}. The results of this section parallel the ones obtained in \cite[Proposition 2]{hR2008} and \cite[Theorem 12.1]{iRjS2009}. We remark that they are not optimal, and more information could be extracted with additional work.

\begin{theorem}[\textbf{Asymptotics}]   \label{T:Asymptotics} 
 Assume that the initial data $(\mathring{g}_{\mu \nu}, \mathring{K}_{\mu \nu}, \mathring{P}, \mathring{u}^j),$
 $(\mu, \nu = 0,1,2,3),$ $(j=1,2,3),$ $g_{\alpha \beta}u^{\alpha}u^{\beta}|_{t=0} = -1,$
 for the modified Euler-Einstein system \eqref{E:finalg00equation} - \eqref{E:finalEulerj} 
 satisfy the assumptions of Theorem \ref{T:GlobalExistence}, including the smallness assumption 
 $\totalnorm{N}(0) \leq C_*^{-1} \epsilon,$ where $0 \leq \epsilon \leq \epsilon_0.$ Let $\mathring{g}^{\mu \nu}$ denote the 
 inverse of $\mathring{g}_{\mu \nu}.$ Assume in addition that $N \geq 5,$ and let $(g_{\mu \nu}, P, u^{\mu})$ 
 be the future-global solution launched by the data. Then there exists a constant $\epsilon_2$ satisfying
 $0 < \epsilon_2 \leq \epsilon_0$ such that if $\epsilon < \epsilon_2,$ then there exists a Riemann metric $g_{jk}^{(\infty)},$ 
 ($j,k = 1,2,3$), with corresponding Christoffel symbols $\Gamma_{ijk}^{(\infty)},$ $(i,j,k = 1,2,3),$
 and inverse $g_{(\infty)}^{jk}$ on $\mathbb{T}^3,$ a function $P_{(\infty)}$ on $\mathbb{T}^3,$
 and a time-independent vectorfield $u_{(\infty)}^{(j)}$ on $\mathbb{T}^3$ such that 
 $g_{jk}^{(\infty)} - \mathring{g}_{jk} \in H^{N},$ $g_{(\infty)}^{jk} - \mathring{g}^{jk} \in H^N,$ 
 $P_{(\infty)} - \bar{p} \in H^{N-1},$ $u_{(\infty)}^j \in H^{N-2},$ and such that the following estimates hold for all 
 $t \geq 0:$
	
	\begin{subequations}
		\begin{align}
			\| g_{jk}^{(\infty)} - \mathring{g}_{jk} \|_{H^N} & \leq C \epsilon, \label{E:gjklowerinfinityHN} \\
			\| g_{(\infty)}^{jk} - \mathring{g}^{jk} \|_{H^N} & \leq C \epsilon, \label{E:gjkupperinfinityHN} 
		\end{align}
	\end{subequations}
	
	\begin{subequations}
	\begin{align}
		\| e^{-2 \Omega} g_{jk} - g_{jk}^{(\infty)} \|_{H^N} & \leq C \epsilon e^{-qHt},  \label{E:gjklowerconvergence} \\
		\| e^{-2 \Omega} g_{jk} - g_{jk}^{(\infty)} \|_{H^{N-2}} & \leq C \epsilon e^{-2Ht},  
			\label{E:improvedgjklowerconvergence} \\
		\| e^{2 \Omega} g^{jk} - g_{(\infty)}^{jk} \|_{H^N} & \leq C \epsilon e^{-qHt},  \label{E:gjkupperconvergence} \\
		\| e^{2 \Omega} g^{jk} - g_{(\infty)}^{jk} \|_{H^{N-2}} & \leq C \epsilon e^{-2Ht},  
			\label{E:improvedgjkupperconvergence} \\
		\| e^{-2 \Omega} \partial_t g_{jk} - 2 \omega g_{jk}^{(\infty)} \|_{H^N} & \leq C \epsilon e^{-qHt}, 
			\label{E:e2Omegapartialgjkminusrhojklowerconvergence} \\
		\| e^{-2 \Omega} \partial_t g_{jk} - 2 \omega g_{jk}^{(\infty)} \|_{H^{N-2}} & \leq C \epsilon e^{-2Ht},
			\label{E:improvede2Omegapartialgjkminusrhojklowerconvergence} \\
			\| e^{2 \Omega} \partial_t g^{jk} + 2 \omega g_{(\infty)}^{jk}\|_{H^N} & \leq C \epsilon e^{-qHt}, 
			\label{E:e2Omegapartialgjkminusrhojkupperconvergence} \\
		\| e^{2 \Omega} \partial_t g^{jk} + 2 \omega g_{(\infty)}^{jk} \|_{H^{N-2}} & \leq C \epsilon e^{-2Ht},
			\label{E:improvede2Omegapartialgjkminusrhojkupperconvergence}
	\end{align}
	\end{subequations}

	\begin{subequations}
	\begin{align}
		\| g_{0j} - H^{-1} g_{(\infty)}^{ab} \Gamma_{ajb}^{(\infty)} \|_{H^{N-3}} & \leq C \epsilon e^{-qHt},
			\label{E:g0jlowerconvergence} \\
		\| \partial_t g_{0j} \|_{H^{N-3}} & \leq C \epsilon e^{-qHt},
			\label{E:partialtg0jconvergence} 
	\end{align}
	\end{subequations}
	
	\begin{subequations}
	\begin{align}
		\| g_{00} + 1\|_{H^N} & \leq C \epsilon e^{-qHt}, \label{E:g00plusoneconvergence} \\
		\| g_{00} + 1\|_{H^{N-2}} & \leq C \epsilon (1 + t) e^{-2Ht}, \label{E:improvedg00plusoneconvergence} \\
		\| \partial_t g_{00} \|_{H^N} & \leq C \epsilon e^{-qHt}, \label{E:partialtg00convergence} \\
		\| \partial_t g_{00} + 2 \omega(g_{00} + 1) \|_{H^{N-2}} & \leq C \epsilon e^{-2Ht}, 
			\label{E:improvedpartialtg00convergence}
	\end{align}
	\end{subequations}
	
	\begin{subequations}
	\begin{align} 
		\| e^{-2 \Omega} K_{jk} - \omega g_{jk}^{(\infty)} \|_{H^{N-1}} & \leq C \epsilon e^{-qHt}, \label{E:Kjkconvergence} \\
		\| e^{-2 \Omega} K_{jk} - \omega g_{jk}^{(\infty)} \|_{H^{N-2}} & \leq C (1+t) \epsilon e^{-2Ht}. 				
			\label{E:improvedKjkconvergence}
	\end{align}
	\end{subequations}
	In the above inequalities, $K_{jk}$ is the second fundamental form of the hypersurface 
	$\lbrace t= const \rbrace.$
	
	Furthermore, we have that
	\begin{subequations}
	\begin{align}
		\| u^0 - 1 \|_{H^{N}} & \leq C \epsilon e^{- qHt}, \label{E:u0convergence} \\
		\| u^0 - 1 \|_{H^{N - 2}} & \leq C \epsilon e^{- 2(1 - \decayparameter) Ht}, \label{E:improvedu0convergence} \\
		\| e^{(2 - \decayparameter) \Omega} u^j - u_{(\infty)}^j \|_{H^{N-2}} 
			& \leq C \epsilon \big\lbrace e^{- \decayparameter Ht} + e^{-2(1 - \decayparameter)Ht} \big\rbrace, 
			\label{E:improvedujconvergence} \\
		\| u_{(\infty)}^j \|_{H^{N-2}} & \leq C \epsilon,   \label{E:ujinfinityHNminusone} \\
		\|P - P_{(\infty)} \|_{H^{N-1}} & \leq C \epsilon e^{-qHt}, \label{E:Pconvergence} \\
		\|P - P_{(\infty)} \|_{H^{N-2}} & \leq C \epsilon e^{-2(1 - \decayparameter) Ht}, \label{E:improvedPconvergence} \\
		\|P_{(\infty)} - \bar{p} \|_{H^{N-1}} & \leq C \epsilon. \label{E:PinfinityHNminusone}
	\end{align}
	\end{subequations}
\end{theorem}

\begin{proof}
We only provide a sketch of the proof, since most of the details can be found in the proof of \cite[Theorem 12.1]{iRjS2009}.
The main idea behind the improved decay rates is that many of terms in the modified equations \eqref{E:finalg00equation} - \eqref{E:finalEulerj} can be treated as inhomogeneities for equations that have a more favorable structure. In fact, many of the terms that are of principal order from the point of view of the number of derivatives can be treated as lower-order terms in the sense of decay rates. Of course, this approach is only viable for treating the lower-order derivatives of the solution. A related consequence is that some of our estimates are proved only in Sobolev spaces of lower order than the ones the data belong to.

The estimates \eqref{E:gjklowerinfinityHN} - \eqref{E:improvedKjkconvergence} can be proved by a straightforward adaptation
of the proof of \cite[Theorem 12.1]{iRjS2009}, and we omit the details. Some very important ingredients in the proof of \cite[Theorem 12.1]{iRjS2009} were the following upgraded estimates

\begin{align}
	\| \partial_t g_{0j} \|_{H^{N-1}} & \leq C \epsilon, 
		\label{E:partialtg0jlowerHNimprovement} \\
	\|g_{0j}\|_{H^{N-1}} & \leq C \epsilon, 
			\label{E:g0jlowerHNimprovement} \\
	\|\underpartial g_{0j}\|_{H^{N-1}} & \leq C \epsilon e^{\Omega}, 
			\label{E:barpartialg0jlowerHNimprovement} \\ 
	\|g^{0j}\|_{H^{N-1}} & \leq C \epsilon e^{-2 \Omega}, \label{E:g0jupperHNimprovement} \\
	\|\partial_t h_{jk} \|_{H^{N-2}} & \leq C \epsilon e^{-2 \Omega}, \label{E:partialthjkimprovedbound}
\end{align}
which we will use below. 

To prove \eqref{E:improvedujconvergence}, we first use \eqref{E:partialtuj} to deduce that

\begin{align} \label{E:Improvedpartialtujequation}
	\partial_t (e^{(2 - \decayparameter)\Omega}u^j) = e^{(2 - \decayparameter)\Omega} \triangle'^j.
\end{align}
We now introduce the following non-negative fluid energy $\newufluidenergy{N-2}(t),$ which is defined by

\begin{align} \label{E:newfluidenergyNminusone}
	\newufluidenergy{N-2}^2 & \eqdef \sum_{j=1}^3 \| e^{(2 - \decayparameter)\Omega} u^j \|_{H^{N-2}}^2.
\end{align}
We then differentiate under the spatial integrals implicit in the definition \eqref{E:newfluidenergyNminusone} and
use \eqref{E:Improvedpartialtujequation} plus the Cauchy-Schwarz inequality for integrals, thus arriving at the following inequality:

\begin{align} \label{E:newenergyinequalityGronwallready}
	\frac{d}{dt} \big(\newufluidenergy{N-2}^2 \big) & \leq 2 \newufluidenergy{N-2} e^{(2 - \decayparameter)\Omega} 
		\| \triangle'^j \|_{H^{N-2}}. 
\end{align}

Using Corollary \ref{C:DifferentiatedSobolevComposition}, Proposition \ref{P:F1FkLinfinityHN}, the definition \eqref{E:totalnorm} of $\totalnorm{N}$ (which satisfies $\totalnorm{N} \leq \epsilon$),
the definition \eqref{E:newfluidenergyNminusone} of $\newufluidenergy{N-2},$ Sobolev embedding, the estimates of Proposition \ref{P:BoostrapConsequences} and Proposition \ref{P:Nonlinearities}, \eqref{E:gjklowerconvergence} - \eqref{E:improvedpartialtg00convergence}, and the improved estimates \eqref{E:partialtg0jlowerHNimprovement} - \eqref{E:partialthjkimprovedbound}, it follows that

\begin{align} \label{E:triangleprimejnewHNminusone}
	\| \triangle'^j \|_{H^{N-2}} & \leq C \epsilon e^{-2 \Omega}
		+ C \newufluidenergy{N-2} e^{(3\decayparameter - 4)\Omega}.
\end{align}
We remark that the $C \newufluidenergy{N-2} e^{(3\decayparameter - 4)\Omega}$ term on the right-hand
side of \eqref{E:triangleprimejnewHNminusone} arises from the $\frac{\speed^2}{(1 + \speed^2)P}\Pi^{0j}\triangle'$
term on the right-hand side of \eqref{E:triangleprimejdef}; this term contains terms that are the product of the spatial metric components $g_{ab}$ with the cube of the spatial four-velocity components $u^j.$ We now insert the bound \eqref{E:triangleprimejnewHNminusone} into the right-hand side of \eqref{E:newenergyinequalityGronwallready}, use the initial condition $\newufluidenergy{N-2}(0) \leq C \epsilon,$ and apply Gronwall's inequality, thus concluding that

\begin{align} \label{E:newenergyinequality}
	\newufluidenergy{N-2}(t) & \leq C \epsilon. 
\end{align}

Inserting the bound \eqref{E:newenergyinequality} into the right-hand side of \eqref{E:triangleprimejnewHNminusone} and
revisiting \eqref{E:Improvedpartialtujequation}, we have that

\begin{align} \label{E:improvedpartialtujbound}
	\| \partial_t (e^{(2 - \decayparameter)\Omega}u^j) \|_{H^{N-2}} & \leq 
	C \epsilon e^{- \decayparameter Ht} + C \epsilon e^{2(\decayparameter - 1)Ht}.
\end{align}
From \eqref{E:improvedpartialtujbound}, it easily follows that there exist functions $u_{(\infty)}^j(x^1,x^2,x^3)$
such that

\begin{align}
	\| e^{(2 - \decayparameter)\Omega}u^j - u_{(\infty)}^j \|_{H^{N-2}} 
		& \leq C \epsilon e^{- \decayparameter Ht} + C \epsilon e^{2(\decayparameter - 1)Ht},
\end{align}
where

\begin{align}
	\| u_{(\infty)}^j \|_{H^{N-2}} & \leq C \epsilon.
\end{align}
We have thus shown \eqref{E:improvedujconvergence} and \eqref{E:ujinfinityHNminusone}.

The estimates \eqref{E:u0convergence} - \eqref{E:improvedu0convergence} and \eqref{E:Pconvergence} - \eqref{E:PinfinityHNminusone} can be proved similarly with the help of equations \eqref{E:U0UPPERISOLATED} 
and \eqref{E:partialtP}; we omit the details.

\end{proof}

\section*{Acknowledgments}
I would like to thank Alan Rendall for inspiring this research, and for pointing out the reference \cite{uBaRoR1994}. I would also like to thank Cambridge University and Princeton University for their support during the conceptualization and writing of this article.

\begin{center}
	\textbf{\huge{Appendices}}
\end{center}
\setcounter{section}{0}
   \setcounter{subsection}{0}
   \setcounter{subsubsection}{0}
   \setcounter{paragraph}{0}
   \setcounter{subparagraph}{0}
   \setcounter{figure}{0}
   \setcounter{table}{0}
   \setcounter{equation}{0}
   \setcounter{theorem}{0}
   \setcounter{definition}{0}
   \setcounter{remark}{0}
   \setcounter{proposition}{0}
   \renewcommand{\thesection}{\Alph{section}}
   \renewcommand{\theequation}{\Alph{section}.\arabic{equation}}
   \renewcommand{\theproposition}{\Alph{section}-\arabic{proposition}}
   \renewcommand{\thecorollary}{\Alph{section}.\arabic{corollary}}
   \renewcommand{\thedefinition}{\Alph{section}.\arabic{definition}}
   \renewcommand{\thetheorem}{\Alph{section}.\arabic{theorem}}
   \renewcommand{\theremark}{\Alph{section}.\arabic{remark}}
   \renewcommand{\thelemma}{\Alph{section}-\arabic{lemma}}

\section{Sobolev-Moser Inequalities} \label{A:SobolevMoser}
		In this Appendix, we provide some standard Sobolev-Moser estimates that play a fundamental role in
		our analysis of the nonlinear terms in our equations. The propositions and corollaries stated below can be proved
		using methods similar to those used in \cite[Chapter 6]{lH1997} and in \cite{sKaM1981}. The proofs given in the 
		literature are commonly based on a version of the Gagliardo-Nirenberg inequality \cite{lN1959},
		which we state as Lemma \ref{L:GN}, together with repeated use of H\"{o}lder's inequality and/or Sobolev embedding. 
		Throughout this appendix, we abbreviate $L^p=L^p(\mathbb{T}^3),$ and $H^M=H^M(\mathbb{T}^3).$ 

\begin{lemma}                                                                                           \label{L:GN}
    If $M,N$ are integers such that $0 \leq M \leq N,$ and $v$ is a function on $\mathbb{T}^3$ such that $v \in
    L^{\infty}, \|\underpartial^{(N)} v \|_{L^2} < \infty,$ then
    \begin{align}
        \| \underpartial^{(M)} v \|_{L^{2N/M}} \leq C(M,N) \| v \|_{L^{\infty}}^{1 -
        \frac{M}{N}}\|\underpartial^{(N)} v \|_{L^2}^{\frac{M}{N}}.
    \end{align}
\end{lemma}

\begin{proposition} \label{P:derivativesofF1FkL2}
	Let $M \geq 0$ be an integer. If $\lbrace v_a \rbrace_{1 \leq a \leq l}$ are functions such that $v_a \in
    L^{\infty}, \|\underpartial^{(M)} v_a \|_{L^2} < \infty$ for $1 \leq a \leq l,$ and
	$\vec{\alpha}_1, \cdots, \vec{\alpha}_l$ are spatial derivative multi-indices with 
	$|\vec{\alpha}_1| + \cdots + |\vec{\alpha}_l| = M,$ then
	
	\begin{align}
		\| (\partial_{\vec{\alpha}_1}v_1) (\partial_{\vec{\alpha}_2}v_2) \cdots (\partial_{\vec{\alpha}_l}v_l)\|_{L^2}
		& \leq C(l,M) \sum_{a=1}^l \Big( \| \underpartial^{(M)} v_a  \|_{L^2} \prod_{b \neq a} \|v_{b} \|_{L^{\infty}} \Big).
	\end{align}
\end{proposition}

\begin{corollary}                                                  \label{C:DifferentiatedSobolevComposition}
    Let $M \geq 1$ be an integer, let $\mathfrak{K}$ be a compact set, and let $F \in C_b^M(\mathfrak{K})$ be a 
    function. Assume that $v$ is a function such that $v(\mathbb{T}^3) \subset \mathfrak{K}$ and $ \underpartial v \in H^{M-1}.$
    Then $\underpartial (F \circ v) \in H^{M-1},$ and
    
    \begin{align} 														\label{E:DifferentiatedModifiedSobolevEstimate}
    	\| \underpartial (F \circ v) \|_{H^{M-1}} 
    		& \leq C(M) \| \underpartial v \|_{H^{M-1}} \sum_{l=1}^M |F^{(l)}|_{\mathfrak{K}} 
    		\| v \|_{L^{\infty}}^{l - 1}.
    \end{align}
\end{corollary}

\begin{corollary}                                                                                             \label{C:SobolevTaylor}
     Let $M \geq 1$ be an integer, let $\mathfrak{K}$ be a compact, convex set, and let $F \in C_b^M(\mathfrak{K})$ be a 
     function. Assume that $v$ is a function such that $v(\mathbb{T}^3) \subset \mathfrak{K}$ and $v - \bar{v} \in H^M,$
    where $\bar{v} \in \mathfrak{K}$ is a constant. Then $F \circ v - F \circ \bar{v} \in H^M,$ and
    
    \begin{align} 														\label{E:ModifiedSobolevEstimateConstantArray}
    	\|F \circ v - F \circ \bar{v} \|_{H^M} 
    		\leq C(M) \Big\lbrace |F^{(1)}|_{\mathfrak{K}}\| v - \bar{v} \|_{L^2} 
    		+ \| \underpartial v \|_{H^{M-1}} \sum_{l=1}^M  |F^{(l)}|_{\mathfrak{K}} 
    		\| v \|_{L^{\infty}}^{l - 1} \Big\rbrace.
    \end{align}
\end{corollary}

\begin{proposition} \label{P:F1FkLinfinityHN}
	Let $M \geq 1, l \geq 2$ be integers. Suppose that $\lbrace v_a \rbrace_{1 \leq a \leq l}$ are functions such that $v_a \in
    L^{\infty}$ for $1 \leq a \leq l,$ that $v_l \in H^M,$ and that
	$\underpartial v_a \in H^{M-1}$ for $1 \leq a \leq l - 1.$
	Then
	
	\begin{align}
		\| v_1 v_2 \cdots v_l \|_{H^M} \leq C(l,M) \Big\lbrace \| v_l \|_{H^M} \prod_{a=1}^{l-1} \| v_a \|_{L^{\infty}}  
		+ \sum_{a=1}^{l-1} \| \underpartial v_a \|_{H^{M-1}} \prod_{b \neq a} \| v_b \|_{L^{\infty}} \Big\rbrace.
	\end{align}	
\end{proposition}

\begin{remark}
	The significance of this proposition is that only one of the functions, namely $v_l,$ is estimated in $L^2.$
\end{remark}

\begin{proposition}                                                                             \label{P:SobolevMissingDerivativeProposition}
    Let $M \geq 1$ be an integer, let $\mathfrak{K}$ be a compact, convex set, and
    let $F \in C_b^M(\mathfrak{K})$ be a function.
    Assume that $v_1$ is a function such that $v_1(\mathbb{T}^3) \subset \mathfrak{K},$ that $\underpartial v_1 \in 
    L^{\infty},$ and that $\underpartial^{(M)} v_1 \in L^2.$ Assume that $v_2 \in L^{\infty},$ that $\underpartial^{(M-1)} v_2 \in L^2,$ 
    and let $\vec{\alpha}$ be a spatial derivative multi-index with with $|\vec{\alpha}| = M.$ Then 
    $\partial_{\vec{\alpha}} \left((F \circ v_1 )v_2\right) - (F \circ v_1)\partial_{\vec{\alpha}} v_2 \in L^2,$ 
    and
        
 			\begin{align}       \label{E:SobolevMissingDerivativeProposition}
      	\|\partial_{\vec{\alpha}} & \left((F \circ v_1 )v_2\right) - (F \circ v_1)\partial_{\vec{\alpha}} v_2\|_{L^2} \notag \\
        	& \leq C(M) \Big\lbrace |F^{(1)}|_{\mathfrak{K}} \|\underpartial v_1 \|_{L^{\infty}} \| \underpartial^{(M-1)} v_2 \|_{L^2} 
					+ \| v_2 \|_{L^{\infty}} \| \underpartial v_1 \|_{H^{M-1}} \sum_{l=1}^M |F^{(l)}|_{\mathfrak{K}} 
    			\| v_1 \|_{L^{\infty}}^{l - 1} \Big\rbrace.
       \end{align}
\end{proposition}

\begin{remark}
	The significance of this proposition is that the $M^{th}$ order derivatives of $v_2$ do not 
	appear on the right-hand side of \eqref{E:SobolevMissingDerivativeProposition}.
\end{remark}

\bibliographystyle{amsalpha}
\bibliography{JBib}

\ \\

\end{document}